\documentclass{article}
\usepackage{PRIMEarxiv}
\usepackage[utf8]{inputenc} % allow utf-8 input
\usepackage[T1]{fontenc}    % use 8-bit T1 fonts
\usepackage{hyperref}       % hyperlinks
\usepackage{url}            % simple URL typesetting
\usepackage{booktabs}       % professional-quality tables
\usepackage{amsfonts}       % blackboard math symbols
\usepackage{nicefrac}       % compact symbols for 1/2, etc.
\usepackage{microtype}      % microtypography
\usepackage{graphicx}
\usepackage{enumitem}

\graphicspath{{media/}}     % organize your images and other figures under media/ folder
\usepackage{float}
\usepackage{multirow}
\usepackage{amsmath}
\usepackage{amsthm}
\usepackage{amssymb}
\usepackage{pst-node}
\usepackage{mathtools,tikz-cd} 
\usepackage{pst-plot}
\usepackage{subfigure}
\usepackage{bbm}
\usepackage[numbers,square]{natbib}
\usepackage{algorithm}
\usepackage{algpseudocode}
\usepackage[titletoc,toc,title]{appendix}
\usepackage{xcolor}

\DeclareMathOperator*{\argmin}{arg\,min}
\DeclareMathOperator{\diag}{diag}
\DeclareMathOperator{\Diag}{Diag}

\DeclareMathOperator*{\minimize}{minimize}

\DeclareMathOperator{\subjectto}{subject\ to}
\DeclareMathOperator{\Gl}{GL}
\newtheorem{thm}{Theorem}[section]
\newtheorem{coro}[thm]{Corollary}
\newtheorem{lem}[thm]{Lemma}
\newtheorem{rem}[thm]{Remark}
\newtheorem{prop}[thm]{Proposition}

\newtheorem{defn}{Definition}[section]
\newtheorem{assm}{Assumption}[section]
\newcommand{\grad}{\mbox{grad\,}}
\newcommand{\Hess}{\mbox{Hess\,}}
\newcommand{\D}{\mbox{D\,}}
\newcommand{\St}{\mbox{St\,}}
\newcommand{\ci}{\mathbbm{i}}
\newcommand{\dt}[1]{\left.{\frac{d}{dt}}\right|_{t={#1}}}

\newcommand{\bmat}[1]{\begin{bmatrix} #1 \end{bmatrix}}

\newcommand{\norm}[1]{\left\lVert#1\right\rVert}
\newcommand{\abs}[1]{\left \lvert #1 \right \rvert}
\newcommand{\ip}[1]{\left\langle #1 \right\rangle}

\newcommand{\MINone}[3]{\begin{array}{ll} \displaystyle \minimize_{#1} & {#2} \\ \subjectto & {#3} \end{array}}

\newcommand{\StatexIndent}[1][3]{%
  \setlength\@tempdima{\algorithmicindent}%
  \Statex\hskip\dimexpr#1\@tempdima\relax}

\allowdisplaybreaks

\title{Riemannian optimization using three different metrics for Hermitian PSD fixed-rank constraints: an extended version
}

\author{
   Shixin Zheng \\
  Department of Mathematics \\
  Purdue University\\
  West Lafayette, USA\\
  \texttt{zheng513@purdue.edu}
   \And
  Wen Huang \\
  School of Mathematical Sciences \\
  Xiamen University \\
  Xiamen, China\\
  \texttt{wen.huang@xmu.edu.cn} \\
  \AND
  Bart Vandereycken \\
  Section of Mathematics\\
  University of Geneva\\
 Geneva, Switzerland\\
  \texttt{bart.vandereycken@unige.ch} \\
  \And
  Xiangxiong Zhang \\
  Department of Mathematics \\
  Purdue University\\
  West Lafayette, USA\\
  \texttt{zhan1966@purdue.edu}
}

\begin{document}
\maketitle

\begin{abstract}

For smooth optimization problems with a Hermitian positive semi-definite fixed-rank constraint, we consider three existing approaches 
including the simple Burer--Monteiro method, and Riemannian optimization over quotient geometry and the embedded geometry.  
These three methods can be all represented via quotient geometry with three Riemannian metrics $g^i(\cdot, \cdot)$ $(i=1,2,3)$. 
By taking the nonlinear conjugate gradient method (CG) as an example, we show that CG in the factor-based Burer--Monteiro approach is equivalent to Riemannian CG on the quotient geometry with the  Bures-Wasserstein  metric $g^1$. 
Riemannian CG on the quotient geometry with the metric $g^3$ is equivalent to Riemannian CG on the embedded geometry. 
For comparing the three approaches, we analyze
 the condition number of the Riemannian Hessian near the minimizer under the three different metrics. Under certain assumptions, the condition number from the  Bures-Wasserstein  metric $g^1$ is significantly different from the other two metrics.
 Numerical experiments show that the Burer--Monteiro CG method has obviously slower asymptotic convergence rate when the minimizer is rank deficient, which is consistent with the condition number analysis. 
\end{abstract}

\keywords{Riemannian optimization \and Hermitian  fixed-rank positive semidefinite matrices \and embedded manifold \and quotient manifold \and Burer--Monteiro \and conjugate gradient\and Riemannian Hessian\and  Bures-Wasserstein  metric }

\section{Introduction}
\subsection{The Hermitian PSD low-rank  constraints}
In this paper we are interested in algorithms for  
minimizing a real-valued function $f$ with a Hermitian positive semi-definite (PSD) low-rank  constraint
\begin{equation}\label{lowrank_prob}
	\MINone{X}{f(X)}{X \in \mathcal{H}^{n,p}_+},
\end{equation} 
where $\mathcal{H}^{n,p}_+$ denotes the set of $n$-by-$n$ Hermitian PSD matrices of fixed rank $p\ll n$.  
 Even though $ X\in \mathcal{H}^{n,p}_+$ is a nonconvex constraint, in practice~\eqref{lowrank_prob} is often used for approximating solutions to a minimization with a convex PSD constraint:
\begin{equation}\label{min-psd}
	\MINone{X\in \mathbb C^{n\times n}}{f(X)}{X\succcurlyeq 0}.
\end{equation} 
There are many applications of PSD constraints. They arise in semidefinite programming serving as covariance matrices in statistics and kernels in machine learning, etc. See \cite{massart_quotient_2020} and \cite{vandereycken_riemannian_2013} for some of these applications. If the solution of~\eqref{min-psd} is of low rank and $\mathcal O(n^2)$ complexity is too large for storage or computation, it is preferable to consider a low-rank representation of PSD matrices.
For example, real symmetric PSD fixed-rank  matrices were used in \cite{bonnabel2010adaptive, vandereycken_riemannian_2010}.

Since the elements in the constraint set $\mathcal{H}^{n,p}_+$ have a low-rank structure, they can be represented in a low-rank compact form on the order of $\mathcal O(n p^2)$, which is smaller than the $\mathcal O(n^2)$ storage when directly using $X\in \mathbb C^{n\times n}$. In many applications, the cost function in \eqref{min-psd} takes the form $f(X)=\frac12 \|\mathcal A(X)-b\|_F^2$ where $\mathcal A$ is a linear operator and the norm is the Frobebius norm, and  $f(X)$ can be evaluated efficiently by  $\mathcal O(p n\log n)$  flops for $X \in \mathcal{H}^{n,p}_+$, e.g., the PhaseLift problem \cite{candes2013phaselift, candes2015phase} and the interferometry recovery problem \cite{jugnon2013interferometric, demanet2017convex}.
For these kinds of problems,  solving \eqref{lowrank_prob} with an iterative algorithm that works with low-rank representations for $X\in \mathcal{H}^{n,p}_+$ can lead to a good approximate solution to \eqref{min-psd} with compact storage and computational cost.

\subsection{The real inner product and Fr\'{e}chet derivatives}
Since $f(X)$ is real-valued and thus not holomorphic,  $f(X)$ does not have a complex derivative with respect to $X\in\mathbb C^{n\times n}$. 
In this paper, all linear spaces of complex matrices will therefore be regarded as vector spaces over $\mathbb R$. 
For any real vector space $\mathcal{E}$, the inner product on $\mathcal{E}$ is denoted by $\ip{\cdot,\cdot}_\mathcal{E}$.  For real matrices $A,B \in \mathbb{R}^{m\times n}$, the Hilbert–Schmidt inner product is $\ip{A,B}_{\mathbb{R}^{m \times n}} = {tr(A^T B)}$.
Let  $\Re({A})$ and $\Im(B)$ represent the real and imaginary parts of a complex matrix $A$.
For $A,B \in \mathbb{C}^{m\times n}$, the real inner product for the real vector space $\mathbb{C}^{m\times n}$ then equals
 \begin{equation}
     \ip{A,B}_{\mathbb{C}^{m\times n}} := \Re({tr(A^* B)}),
     \label{real-inner-product}
 \end{equation} where ${}^*$ is the conjugate transpose. We emphasize that \eqref{real-inner-product} is a real inner product, rather than the complex Hilbert–-Schmidt inner product. It is straightforward to verify that \eqref{real-inner-product} can be written as
 \[\ip{A,B}_{\mathbb{C}^{m\times n}}=tr(\Re(A)^T \Re(B))+tr(\Im(A)^T \Im(B))=\ip{\Re(A),\Re(B)}_{\mathbb{R}^{m\times n}}+\ip{\Im(A),\Im(B)}_{\mathbb{R}^{m\times n}}. \]
With the real inner product \eqref{real-inner-product} for the real vector space $\mathbb C^{m\times n}$, 
a Fr\'{e}chet derivative for any real valued function $f(X)$ can be defined as
\begin{equation}
     \nabla{f}(X) = \frac{\partial f(X)}{\partial \Re(X)}   +\ci\frac{\partial f(X)}{\partial \Im(X)}
     \in \mathbb{C}^{m\times n}.
     \label{frechet-derivative}
\end{equation}  
In particular, for $f(X)=\frac12 \|\mathcal A(X)-b\|_F^2$ with a linear operator $\mathcal A$, the Fr\'{e}chet derivative \eqref{frechet-derivative} becomes
\[ \nabla{f}(X)=\mathcal A^*(\mathcal A(X)-b)\]
where $\mathcal A^*$ is the adjoint operator of $\mathcal A$. 
See Appendix \ref{appendix-derivative} for details.

\subsection{Three different methodologies}
\label{intro:method}
In this paper we will consider three straightforward ideas and methodologies for solving \eqref{lowrank_prob}. 
\subsubsection{The Burer--Monteiro method}
The first approach, often called the Burer--Monteiro method \cite{burer_local_2005}, 
is to solve the unconstrained problem 

\begin{equation}\label{min-BR}
	\min_{Y\in \mathbb C^{n\times p}}F(Y):= f(YY^*).
\end{equation} 

As proven in Appendix \ref{appendix-derivative},
the chain rule of Fr\'{e}chet derivatives gives 
$$\nabla F(Y) =2 \nabla f(YY^*)Y \in \mathbb C^{n\times p}.$$

The gradient descent method simply takes the form of 
\begin{equation*}
    \label{br-gd}Y_{n+1}=Y_n-\tau\nabla F(Y_n)=Y_n-\tau 2\nabla f(Y_nY_n^*)Y_n,
\end{equation*} 
 which is one of the simplest low-rank algorithms.  
The nonlinear conjugate gradient and quasi-Newton type methods, like L-BFGS, can also be easily used for \eqref{min-BR}. 
On the other hand, $F(Y)=F(YO)$ for any unitary matrix $O\in\mathbb O^{p\times p}$, 
where $$\mathcal{O}_p =\{ O \in \mathbb{C}^{p\times p}: O^*O = OO^* = I\}.$$ Even though this ambiguity of unitary matrices is never explicitly addressed in the Burer--Monteiro method, in this paper we will prove that the gradient descent and nonlinear conjugate gradient   methods for solving  \eqref{min-BR} are exactly equivalent to the Riemannian  gradient descent and Riemannian conjugate gradient  methods on a quotient manifold  with a Euclidean metric, which is also referred to as the  Bures-Wasserstein  metric \cite{massart2019curvature, massart_quotient_2020}. Thus the convergence of the Burer--Monteiro method
can be understood within the context of Riemannian optimizaiton on a quotient manifold.

\subsubsection{Riemannian optimization with the embedded geometry of $\mathcal{H}^{n,p}_+$}

Another natural approach is to regard $\mathcal{H}^{n,p}_+$ as an embedded manifold in the Euclidean space $\mathbb C^{n\times n}$. 
For instance, Riemannian optimization algorithms  on the embedded manifold of low-rank matrices and tensors are quite efficient and popular \cite{vandereycken_low-rank_2013, kressner_low-rank_2014}.
Even though it is possible to study $\mathcal{H}^{n,p}_+\subset \mathbb C^{n\times n}$ as a complex manifold, we will regard $\mathbb C^{n\times n}$  as a $2n^2$-dimensional real vector space and $\mathcal{H}^{n,p}_+\subset \mathbb C^{n\times n}$ as a manifold over $\mathbb R$ since $f(X)$ is real-valued. 
  In particular, the embedded geometry of $\mathcal{S}^{n,p}_+$, representing the set real symmetric PSD low-rank matrices, was studied in \cite{vandereycken_embedded_2009}. 
  
 A \textit{Riemannian metric} is a smoothly varying inner product defined on the tangent space. 
The Riemannian metric of the real embedded manifold  $\mathcal{H}^{n,p}_+$ can simply be taken as the inner product \eqref{real-inner-product} on $\mathbb C^{n\times n}$.   
The embedded geometry of the real manifold  $\mathcal{H}^{n,p}_+\subset \mathbb C^{n\times n}$  will be discussed 
in Section \ref{section:embedded_manifold}.

Even though the tangent space and the Riemannian gradient in Section \ref{section:embedded_manifold} for $\mathcal{H}^{n,p}_+$ look like a natural extension of those for $\mathcal{S}^{n,p}_+$, it is not obvious why this should be true. The subtlety lies in the fact that we have to regard $X\in\mathcal{H}^{n,p}_+$ as an element of a real vector space. For instance, for regarding $X\in\mathcal{H}^{n,p}_+$ as a real vector, one can   either regard a complex matrix $X$ as the pair of  its real and imaginary part, or regard $X\in \mathbb C^{n\times n}$ with its realification, which is a $2n$-by-$2n$ real matrix generated by replacing each complex entry $a+\mathbbm{i}b$ of $X$ by a 2-by-2 block $\bmat{a&-b\\b&a}$. Unfortunately, neither way gives a straightforward generalization from the real case in \cite{vandereycken_embedded_2009} to the complex case in Section \ref{section:embedded_manifold}.  Instead, with the real inner product \eqref{real-inner-product} and the corresponding Fr\'{e}chet derivative, it is possible to achieve the desired generalization.

\subsubsection{Riemannian optimization by using quotient geometry}

The third approach is to consider the quotient manifold $\mathbb{C}^{n\times p}_*/\mathcal{O}_p$, which will be reviewed in Section \ref{section:quotient_manifold}.
Here $\mathbb{C}^{n\times p}_*$ is the noncompact Stiefel manifold of full rank $n$-by-$p$ matrices:
\[
\mathbb{C}^{n\times p}_*=\{X\in \mathbb{C}^{n\times p}: \mbox{rank}(X)=p\}.\]
Define an equivalent class by
\[[Y]=\{Z\in \mathbb{C}^{n\times p}_*: Z=YO, O\in \mathcal O_p\}\]
and denote the natural projection as
\[
    \begin{aligned}
   \pi:  \mathbb{C}^{n\times p}_* &\rightarrow \mathbb{C}^{n\times p}_*/\mathcal{O}_p\\
 Y  &\mapsto [Y] 
\end{aligned}.
\]
Since there is a one-to-one correspondence between $X=YY^*\in \mathcal{H}^{n,p}_+$ and $\pi(Y)\in\mathbb{C}^{n\times p}_*/\mathcal{O}_p$, the optimization problem (\ref{lowrank_prob}) is equivalent to
\begin{equation}\label{quotient_prob}
	\MINone{\pi(Y)}{h(\pi(Y))}{\pi(Y) \in \mathbb{C}^{n \times p}_*/\mathcal{O}_p},
\end{equation} 
where the cost function $h$ is defined as $h(\pi(Y)) =F(Y)= f(YY^*)$.

For the quotient manifold $\mathbb{C}^{n\times p}_*/\mathcal{O}_p$, one can first choose a metric for its total space $\mathbb{C}^{n\times p}_*$, which induces a Riemannian metric on the quotient manifold under suitable conditions. In particular, a special metric was used in \cite{huang_solving_2017} to construct efficient Riemannian optimization algorithms
for the problem \eqref{min-BR}. The horizontal lift of the Riemannian gradient for $h(\pi(Y))$ under this particular metric satisfies
\begin{equation}
(\grad h(\pi(Y)))_{ Y} = \nabla F(Y) (Y^*Y)^{-1}=2\nabla f(YY^*)Y(Y^*Y)^{-1}.
\label{quotient-grad-metric2}
\end{equation}
From the representation of the Riemannian gradient \eqref{quotient-grad-metric2}, we see that this approach generates different algorithms from the simpler Burer--Monteiro approach.

\subsection{Main results: a unified representation and analysis of three methods using quotient geometry}

A natural question arises: which of the three methods is the best? Even though the unconstrained Burer--Monteiro method is quite straightforward to use, it has an ambiguity up to a unitary matrix, and its performance is usually observed to be inferior to Riemannian optimization on embedded and quotient geometries. In order to compare these three methods, in this paper we will show that it is possible to equivalently rewrite both the Burer--Monteiro approach and embedded manifold approach as Riemannian optimization over the quotient manifold  
$\mathbb{C}^{n\times p}_*/\mathcal{O}_p$ with suitable metrics, retractions and vector transports.

For any $Y\in \mathbb{C}^{n\times p}_*$, we  consider three different Riemannian metrics $g_Y^i(\cdot,\cdot)$ $(i=1,2,3)$ for any $A, B$ in the total space $\mathbb{C}^{n\times p}_*$:
\begin{align*}
    g_Y^1(A,B) &= \ip{A,B}_{\mathbb{C}^{n\times p}} = \Re(tr(A^*B))\\
    g_Y^2(A,B) &= \ip{AY^*,BY^*}_{\mathbb{C}^{n\times n}} = \Re(tr((Y^*Y)A^*B))\\
    g_Y^3(A,B) &=\ip{YA^*+AY^*,YB^*+BY^*}_{\mathbb{C}^{n\times n}}+ \ip{P^\mathcal{V}_Y(A)Y^*,P^\mathcal{V}_Y(B)Y^*}_{\mathbb{C}^{n\times n}},
\end{align*}
where  $P_Y^\mathcal{V}$ is given by 
\[P_Y^\mathcal{V} (A) = Y \left(\frac{(Y^*Y)^{-1}Y^*A - A^* Y(Y^*Y)^{-1}}{2} \right).\]

In particular, the Burer--Monteiro approach corresponds to the   Bures-Wasserstein metric $g_Y^1$
and the embedded manifold approach corresponds to the third metric $g_Y^3$. The second metric $g_Y^2$ is the one used in \cite{huang_solving_2017}.

 We will show that both the gradient descent and the conjugate gradient  method for the unconstrained problem \eqref{min-BR} are equivalent to a Riemannian 
gradient descent  and a Riemannian conjugate gradient  method on the quotient manifold $\mathbb{C}^{n\times p}_*/\mathcal{O}_p$ with the simplest  Bures-Wasserstein  $g_Y^1$ and a specific vector transport. 

Furthermore, we will prove that the Riemannian 
gradient descent  and the Riemannian conjugate gradient  methods using the embedded geometry of $\mathcal{H}^{n,p}_+$ are equivalent to a Riemannian 
gradient descent  and a Riemannian conjugate gradient  algorithms on the quotient manifold $\mathbb{C}^{n\times p}_*/\mathcal{O}_p$ with the metric $g_Y^3$ and a specific vector transport.

Finally, we will analyze and compare the condition numbers of the Riemannian Hessian using these three different metrics by estimating their Rayleigh quotient. It is well known that the condition number of the Hessian of the cost function is closely related to the asymptotic performance of optimization methods. Under the assumption that the Fr\'{e}chet Hessian $\nabla^2 f(X)$ is well conditioned and some other technical assumptions, we will show that the the condition numbers of the Riemannian Hessian using the first metric can be significantly worse than the other two if the minimizer of \eqref{min-psd} has a rank smaller than $p$.

\subsection{Related work}

The Burer--Monteiro approach for the PSD constraint has been  popular in applications due to its simplicity. For instance, an L-BFGS method for \eqref{min-BR} was used for solving  convex recovery from interferometric measurements in \cite{demanet2017convex}.
It is straightforward to verify that \eqref{br-gd}  with $p=1$ and a suitable step size $\tau$ for the PhaseLift problem \cite{candes2013phaselift}  is precisely the Wirtinger flow algorithm \cite{candes2015phase}. 
In \cite{boumal2020deterministic}, it was shown that first-order and second-order optimality conditions of the nonconvex Burer--Monteiro approach are sufficient to find the global minimizer of the convex semi-definite program under certain assumptions.

Riemannian optimization on various matrix manifolds such as the Stiefel manifold, the Grassmann manifold and the set of fixed-rank matrices,  have been used 
for  applications in data science, machine learning, signal processing, bio-science, etc.   
The geometry of real symmetric PSD matrices of fixed rank $\mathcal{S}^{n,p}_+$ has also been studied intensively in the literature. 
 Its embedded geometry was studied in \cite{vandereycken_embedded_2009}.   
  The quotient geometry was studied in  \cite{bonnabel_riemannian_2010, journee_low-rank_2010, massart_quotient_2020}.
    As we will show in Section \ref{sec-bwmetric},  the Bures-Wasserstein metric $g^1$ for low-rank PSD matrices is consistent with the Bures-Wasserstein metric for Hermitian positive-definite matrices $\mathcal{H}^{n,n}_+$ \cite{oostrum2022bures, bhatia2019bures, han2021riemannian}. 
   
Riemannian optimization based on the embedded geometry has been well studied in \cite{vandereycken_low-rank_2013} for  real matrices of fixed rank, which can be easily extended to 
 real symmetric PSD matrices of fixed rank \cite{ vandereycken_embedded_2009}. 
As expected, Section \ref{section:embedded_manifold} is its natural extensions to Hermitian PSD  matrices of fixed rank. This is not  surprising, but it is not a straightforward result either, because such a natural extension holds only when using the real inner product \eqref{real-inner-product} and its associated  Fr\'{e}chet derivatives.   

The quotient geometry of Hermitian PSD matrices of fixed-rank for the metric $g_Y^2$ has been studied in \cite{huang2013optimization, huang_solving_2017}. The quotient geometry with metric $g_Y^2$ in this paper is exactly the same one as the one in \cite{huang2013optimization, huang_solving_2017}.

It is not uncommon to explore different metrics of a manifold for Riemannian optimization  \cite{absil_geometric_2009, mishra2014riemannian}. In \cite{vandereycken_riemannian_2013}, a new embedded geometry and complete geodesics for real PSD fixed-rank matrices were for example obtained from a special quotient metric.

\subsection{Contributions}

In this paper, for simplicity, we only focus on the nonlinear conjugate gradient method. 

First, we will prove that the nonlinear conjugate gradient method for the unconstrained Burer--Monteiro formulation \eqref{min-BR} is equivalent to a Riemannian conjugate gradient method on the quotient manifold $(\mathbb{C}^{n\times p}_*/\mathcal{O}_p, g^1)$ for solving \eqref{quotient_prob}. Thus the convergence of the simple Burer--Monteiro approach can be understood in the context of Riemannian optimization on the quotient manifold with the  Bures-Wasserstein  metric $(\mathbb{C}^{n\times p}_*/\mathcal{O}_p, g^1)$.
In other words, the Riemannian conjugate gradient method under the  Bures-Wasserstein  metric $(\mathbb{C}^{n\times p}_*/\mathcal{O}_p, g^1)$ can be implemented as the simple unconstrained factor based Burer--Monteiro formulation.
This is one major contribution of this paper.

Second,
 we will show that a Riemannian conjugate gradient method on the embedded manifold $\mathcal H^{n,p}_+$ for solving \eqref{lowrank_prob} is equivalent to a Riemannian conjugate gradient method on the quotient manifold $(\mathbb{C}^{n\times p}_*/\mathcal{O}_p, g^3)$ for solving \eqref{quotient_prob}. For implementation, this is completely unnecessary. However, this allows a comparison a Riemannian optimization algorithm on an embedded geometry with a Riemannian optimization algorithm on a quotient geometry.

Finally, for the sake of understanding the differences among the three methodologies, we will analyze the condition number of the Riemannian Hessian on the quotient manifold $(\mathbb{C}^{n\times p}_*/\mathcal{O}_p, g^i)$ for the three different metrics $g^i$ $(i=1,2,3)$. One metric is equivalent to the simple Burer--Monteiro approach and another to Riemannian optimization on the embedded manifold $\mathcal H^{n,p}_+$. 
Since the three methods in Section \ref{intro:method} can all be regarded as Riemannian optimization algorithms on a quotient manifold with three different metrics, such a comparison is meaningful.

In certain problems, such as PhaseLift \cite{candes2013phaselift} and interferometry recovery \cite{demanet2017convex}, the rank $r$ of the minimizer of \eqref{min-psd} is known. However, it has been observed in practice that the basin of  attraction is larger when  solving the nonconvex problems \eqref{min-BR} or \eqref{quotient_prob} with rank $p>r$ instead of with rank $p=r$; see  \cite{demanet2017convex, huang_solving_2017}. We will also  demonstrate this in the numerical tests in Section \ref{section:numerical_experiments}.  Under suitable assumptions, we will show that the condition number of the Riemannian Hessian on the quotient manifold $(\mathbb{C}^{n\times p}_*/\mathcal{O}_p, g^1)$ can be unbounded if $p>r$. On the other hand, the condition numbers of the Riemannian Hessians on the quotient manifold $\mathbb{C}^{n\times p}_*/\mathcal{O}_p$ with metrics $g^1$ and $g^2$  are still bounded. This is consistent with the numerical observation that the Burer--Monteiro approach has a much slower asymptotic convergence rate than the Riemannian optimization approach on the embedded manifold and the quotient manifold $(\mathbb{C}^{n\times p}_*/\mathcal{O}_p, g^2)$ when $p>r$.

\subsection{Organization of the paper}

The outline of the paper is as follows. We summarize the notation in Section \ref{sec:notation}.
Then we discuss the geometric operators such as the Riemannian gradient and vector transport in Section \ref{section:embedded_manifold} for the embedded manifold $\mathcal{H}^{n,p}_+$ and in Section \ref{section:quotient_manifold} for  the quotient manifold $\mathbb{C}^{n\times p}_*/\mathcal{O}_p$.  In Section \ref{sec:RCG}, we outline the Riemannian Conjugate Gradient (RCG) methods on different geometries 
and discuss equivalences among them. 
In particular, we show that RCG on the quotient manifold $(\mathbb{C}^{n\times p}_*/\mathcal{O}_p, g^1)$ is exactly the Burer--Monteiro CG method, that is, CG directly on \eqref{min-BR}. We also show that Riemannian CG on the embedded manifold for solving (\ref{lowrank_prob}) is equivalent to RCG on the quotient manifold $(\mathbb{C}^{n\times p}_*/\mathcal{O}_p, g^3)$ with a specific retraction and vector transport for solving (\ref{quotient_prob}).
Implementation details are given in Section \ref{sec:details}. 
In Section \ref{section:rayleigh_quotient}, we analyze and compare the condition numbers of the Riemannian Hessian operators, which can be used to understand the difference in the asymptotic convergence rates between using the simple Burer--Monteiro method and the more sophisticated Riemannian optimization using an embedded geometry or a \textbf{}quotient geometry with metric $g^2$.
Numerical tests are given in Section \ref{section:numerical_experiments}. 

\section{Notation}
\label{sec:notation}
 Let $\mathbb{C}^{m\times n}$ denote all complex matrices of size $m\times n$. 
Let $p\leq n$ and define
\begin{eqnarray*}
\mathbb{C}^{n\times p}_*&=&\{X\in \mathbb{C}^{n\times p}: \mbox{rank}(X)=p\},\\
\St(p,n) &=& \{X\in \mathbb{C}^{n\times p}: X^*X = I_p \}, \\
\mathcal{H}_{+}^{n,p}&=&\{X\in \mathbb C^{n\times n}: X^*=X, X\succcurlyeq 0, \mbox{rank}(X)=p\},\\
\mathcal{S}_{+}^{n,p}&=&\{X\in \mathbb R^{n\times n}: X^T=X,  X\succcurlyeq 0, \mbox{rank}(X)=p\},\\
\mathcal{O}_p & =& \{ O \in \mathbb{C}^{p\times p}: O^*O = OO^* = I\},
% \mathcal{S}_{+}^{n,p}&=&\{X\in \mathcal{S}^{n\times n}: X\succcurlyeq 0, \mbox{rank}(X)=p\}.
\end{eqnarray*} 
where  $\St(p,n)$ is also called the compact Stiefel manifold. For a matrix $X$, $X^*$ denotes its conjugate transpose and $\overline X$ denotes its complex conjugate. If $X$ is real, $X^*$ becomes the matrix transpose and is denoted by $X^T$. We define
$$Herm(X):= \frac{X+X^*}{2},\quad Skew(X) := \frac{X-X^*}{2}.$$  

Let $\Re(X)$ and $\Im(X)$ denote the real and imaginary part of $X$ respectively so that $X = \Re(X) + \mathbbm{i} \Im(X)$. 
Let $I_p$ be the identity matrix of size $p$-by-$p$. For any $n$-by-$p$ matrix $Z$, $Z_\perp$ denotes the $n$-by-$(n-p)$ matrix such that $Z_\perp^* Z_\perp = I_{n-p}$ and $Z_\perp^* Z=\mathbf 0 $. 

Let $\Diag(m,n)$ be the set of all $m$-by-$n$ diagonal matrices. Let $\diag(M)$ be the $n$-by-$1$ vector that is the diagonal of the $n$-by-$n$ matrix $M$. Given a vector $v$, $\Diag(v)$ is a square matrix with its $i$th diagonal entry equal to $v_i$. Given a matrix $A$, $tr(A)$ denotes the trace of $A$ and $A_{ij}$ denotes the $(i,j)$th entry of $A$.

For any $X\in \mathcal{H}^{n,p}_+$, its eigenvalues coincide with its singluar values. The compact singular value decomposition (SVD) of $X$ is denoted by $X = U\Sigma U^*$, where $U\in \St(p,n)$ and $\Sigma = \Diag(\sigma)$ with $\sigma = (\sigma_1, \cdots, \sigma_p)^T$ and $\sigma_1 \geq \cdots \geq \sigma_p >0$. In the rest of the paper, $U$ and $\Sigma$ are reserved for denoting the compact SVD of $X\in \mathcal{H}^{n,p}_+$. 

In this paper, all manifolds of complex matrices are viewed as manifolds over $\mathbb{R}$. Given a Euclidean space $\mathcal{E}$, the inner product on $\mathcal{E}$ is denoted by $\ip{.,.}_\mathcal{E}$. Specifically,  $\ip{A,B}_{\mathbb{R}^{m \times n}} = {tr(A^T B)}$ for $A,B \in \mathbb{R}^{m\times n}$ and $\ip{A,B}_{\mathbb{C}^{m\times n}} = \Re({tr(A^* B)})$ for $A,B \in \mathbb{C}^{m\times n}$ denotes the canonical inner product on $\mathbb{R}^{m\times n}$ and $\mathbb{C}^{m\times n}$, respectively.

\section{Embedded geometry of $\mathcal{H}^{n,p}_+$ }\label{section:embedded_manifold}
The results in this section are natural extensions of results for $\mathcal{S}^{n,p}_+$ in \cite{vandereycken_embedded_2009}.
Such an extension is not entirely obvious since $\mathcal{H}^{n,p}_+$ is treated as a real manifold and the real inner product \eqref{real-inner-product} is not the complex Hilbert--Schmidt inner product. For completeness, we discuss these extensions. 

\subsection{Tangent space}
We first need to show that $\mathcal{H}^{n,p}_+$ is a smooth embedded submanifold of $\mathbb{C}^{n\times n}$. 
See \cite[Prop.~2.1]{helmke_critical_1995} and \cite[Chap.~5]{helmke_optimization_2012} for the case of $\mathcal{S}^{n,p}_+$ being a smooth embedded submanifold of $\mathbb{R}^{n\times n}$.
\begin{thm}\label{thm:submanifold_embedded}
Regard $\mathbb{C}^{n\times n}$ as a real vector space over $\mathbb R$ of dimension $2n^2$. Then $\mathcal{H}^{n,p}_+$ is a smooth embedded submanifold of $\mathbb{C}^{n\times n}$ of dimension $2np- p^2$. 
\end{thm}
\begin{proof}
    Let 
    \[
    E = \bmat{I_{p\times p} & 0_{p\times (n-p)} \\ 0_{(n-p)\times p} & 0_{(n-p)\times (n-p)} }
    \]
    and consider the smooth Lie group action 
    \begin{equation*}
    \begin{aligned}
        \Phi: \Gl(n,\mathbb{C}) \times \mathbb{C}^{n\times n} &\rightarrow& \mathbb{C}^{n\times n} \\ 
        (g,N) & \mapsto& gNg^* 
    \end{aligned}
    \end{equation*}
    where
    \begin{eqnarray*}
         gNg^* & =& \left(\Re(g)\Re(N)-\Im(g)\Im(N)\right)\Re(g)^T + (\Im(g)\Re(N) +\Re(g)\Im(N))\Im(g)^T \\ 
            & & +  \mathbbm{i } \left( (\Im(g)\Re(N) +\Re(g)\Im(N))\Re(g)^T  - (\Re(g)\Re(N)-\Im(g)\Im(N))\Im(g)^T \right).
    \end{eqnarray*}
    
    It is easy to see that $\Phi$ is a rational mapping. Since $\Gl(n,\mathbb{C})$ is a semialgebraic set by Lemma (\ref{lem:semialgebraic}) in the Appendix, we have that $\Gl(n,\mathbb{C}) \times \mathbb{C}^{n\times n}$ is also a semialgebraic set \cite[section 2.1.1]{coste_introduction_2000}. It follows from (B1) in \cite{gibson_singular_1979} that $\Phi$ is a semialgebraic mapping. 
    Observe that $\mathcal{H}^{n,p}_+$ is the orbit of $E$ through $\Phi$. It therefore follows from (B4) in \cite{gibson_singular_1979} that $\mathcal{H}^{n,p}_+$ is a smooth submanifold of $\mathbb{C}^{n\times n}$. 
    
    Next, we compute the dimension of $\mathcal{H}^{n,p}_+$. Consider the smooth surjective mapping
\[
	\eta: \text{GL}(n,\mathbb{C}) \rightarrow \mathcal{H}^{n,p}_+ \quad \gamma \mapsto \gamma E \gamma^*.
\]

The differential of $\eta$ at $\gamma \in \text{GL}(n,\mathbb{C})$ is the linear mapping $\text{D}\eta(\gamma): T_\gamma \text{GL}(n,\mathbb{C}) =  \mathbb{C}^{n\times n}\to T_X\mathcal{H}^{n,p}_+$, where $X = \eta(\gamma) = \gamma E \gamma^*$, by $\text{D}\eta(\gamma)[\Delta] = \Delta E \gamma^* + \gamma E \Delta^* $. Observe that the differential at arbitrary $\gamma$ is related to the differential at $I_{n}$ by a full-rank linear transformation:
	\begin{eqnarray}\label{eqn:2}
		\text{D}\eta(\gamma) [\Delta] = \gamma \text{D}\eta(I_n)[\gamma^{-1} \Delta] \gamma^* .
	\end{eqnarray}
Recall that the rank of a differentiable mapping $f$ between two differentiable manifolds is the dimension of the image of the differential of $f$. So, from equation (\ref{eqn:2}) we see that the rank of $\eta$ is constant. It follows from Theorem 4.14 in  \cite{lee_introduction_2012} that $\eta$ is a smooth submersion. As a consequence $\text{D}\eta(\gamma)$ maps $T_\gamma \text{GL}(n,\mathbb{C}) = \mathbb{C}^{n\times n} $ surjectively onto $T_X\mathcal{H}^{n,p}_+$ and we obtain
\begin{equation}\label{eqn:tangent_space}
	T_X\mathcal{H}^{n,p}_+ = \left\{ \Delta X + X \Delta^*: \Delta \in \mathbb{C}^{n\times n}\right\}.
\end{equation}
Let $\Delta = \bmat{\Delta_{11} & \Delta_{12} \\ \Delta_{21} & \Delta_{22} } $ be partitioned according to the partition of $E = \text{diag}(I_{p\times p}) = \bmat{I_{p\times p} & 0 \\ 0 & 0}$. Then it can be easily verified that $\Delta \in \text{Ker}\text{D}\eta(I)$ if and only if 
\[
	\Delta_{11} = - \Delta_{11}^* ,\quad \Delta_{21} = 0.
\]
This implies that $\Delta_{11}$ is a skew-Hermitian matrix, hence its diagonal entries are purely imaginary and its off diagonal entries satisfy $a_{ij} = - \overline{a_{ji}}$. This gives us $p + 2\times (1+2+ \cdots + (p-1))$ degrees of freedom. For $\Delta_{12}$ and $\Delta_{22}$ there are $2n(n-p)$ degrees of freedom. 
So, the dimension of $\text{Ker}(\text{D}\eta(I))$ is $2n(n-p) +p+2p(p-1)/2  = 2n^2 -2np + p^2$ and by rank-nullity we get
\[
    \dim \D\eta(I) = 2n^2 - \dim \ker \D\eta(I) = 2np - p^2.
\]
Since $\eta$ is of constant rank, the dimension of $T_X\mathcal{H}^{n,p}_+$ is therefore $2np - p^2$. Remember that the dimension of the tangent space at every point of a connected manifold is the same as that of the manifold itself. Let $\Gl^+(n,\mathbb{C})$ denote the connected subset of $\Gl(n,\mathbb{C})$ with positive determinant, then $\mathcal{H}^{n,p}_+$ is the image of the connected set $\Gl^+(n,\mathbb{C})$ under a continuous mapping $\eta$, so $\mathcal{H}^{n,p}_+$ is connected.  We conclude that the dimension of $\mathcal{H}^{n,p}_+$ is $2np-p^2$. 
\end{proof}

The next result characterizes the tangent space.   See  \cite[Proposition 2.1]{vandereycken_low-rank_2013} for the tangent space of $\mathcal{S}^{n,p}_+$.

\begin{thm}\label{thm:tangent_space}
	Let $X = U\Sigma U^* \in \mathcal{H}^{n,p}_+$. Then the tangent space of $\mathcal{H}^{n,p}_+$ at $X$ is given by 
	\[
		T_X\mathcal{H}^{n,p}_+ = \left\{ \bmat{U & U_\perp}\bmat{H & K^* \\ K & 0} \bmat{U^* \\ U_\perp^*}\right\}
	\]
	 where $H = H^* \in \mathbb{C}^{p\times p} $, $ K \in \mathbb{C}^{(n-p)\times p}$.
\end{thm}
\begin{rem}
Notice that there is no need to compute and store $U_\perp\in \mathbb C^{n\times (n-p)}$ and it suffices to store $U_\perp K\in \mathbb C^{n\times p}$. See Section~\ref{sec:details} for the implementation details.
\end{rem}
\begin{proof}
Let $t\mapsto U(t)$ be any smooth curve in $\St(p,n)$ through $U$ at $t=0$ such that $U(t) \in \mathbb{C}^{n\times p}, U(0)= U$ and $U(t)^*U(t) = I_p$ for all $t$. Let $t\mapsto \Sigma(t)$ be any smooth curve in $\Diag(p,p)$ through $\Sigma$ at $t=0$.  Then $X(t) := U(t)\Sigma(t)U(t)^*$ defines a smooth curve in $\mathcal{H}^{n,p}_+$ through $X$. 
It follows by differentiating $X(t) := U(t)\Sigma(t)U(t)^*$ that 
\[
    X'(t) = U'(t)\Sigma(t)U(t)^* + U(t)\Sigma'(t)U(t)^* + U(t)\Sigma(t)U'(t)^*.
\]
Without loss of generality, since $U'(t)$ is an element of $\mathbb{C}^{n\times p}$ and $U(t)$ has full rank, we can set 
\[
    U'(t) = U(t)A(t) + U_\perp(t)B(t).
\]
Hence, we have
\[
    X'(t) = \bmat{U(t) & U_\perp(t)}\bmat{A(t)\Sigma(t)+\Sigma'(t)+\Sigma(t)A(t)^* & \Sigma(t)B(t)^* \\ B(t)\Sigma(t) & 0} \bmat{U(t)^* \\ U_\perp(t)^*}.
\]
Thus we consider the tangent vectors in the form of $\bmat{U & U_\perp}\bmat{H & K^* \\ K & 0} \bmat{U^* \\ U_\perp^*}$ with $H=H^*$. For any $H=H^*\in\mathbb{C}^{p\times p}$ and $K\in \mathbb{C}^{(n-p)\times p}$, taking  $\Delta = (UH/2 + U_\perp K) \Sigma^{-1}(U^*U)^{-1} U^*$ in (\ref{eqn:tangent_space}), we see that 
\begin{equation}\label{eqn:tangent_space_2}
\left\{ \bmat{U & U_\perp}\bmat{H & K^* \\ K & 0} \bmat{U^* \\ U_\perp^*}\right\} \subseteq  T_X\mathcal{H}^{n,p}_+.
\end{equation}
Now counting the real dimension we see that  $H$ has $ p +  2\times \frac{p(p-1)}{2} = p^2$ number of freedom and $K$ has $2\times p(n-p)$ number of freedom. So the LHS of the inclusion (\ref{eqn:tangent_space_2}) has freedom $2np - p^2$, which is equal to the dimension of $T_X\mathcal{H}^{n,p}_+$. Hence, the inclusion in (\ref{eqn:tangent_space_2}) is an equality. 
\end{proof}

\subsection{Riemannian gradient}
The \textit{Riemannian metric} of the embedded manifold at $X\in \mathcal{H}^{n,p}_+$ is induced from the Euclidean inner product on $\mathbb{C}^{n\times n}$,
\begin{equation}
    g_X(\zeta_1, \zeta_2) = \ip{\zeta_1,
    \zeta_2}_{\mathbb{C}^{n\times n}} = \Re(tr(\zeta_1^*\zeta_2)),\quad \zeta_1,\zeta_2 \in T_X\mathcal{H}^{n,p}_+.
    \label{metric-embedded-z}
\end{equation}

Let $f(X)$ be a smooth real-valued function for $X\in \mathbb{C}^{n\times n}$
and Fr\'{e}chet gradient \eqref{frechet-derivative}, is denoted by $\nabla {f}(X)$. See Appendix \ref{frechet_derivatives} for more details about Fr\'{e}chet gradient. 

The \textit{Riemannian gradient} of $f$ at $X\in \mathcal{H}^{n,p}_+$, denoted by $\grad f(X)$, is the projection of $\nabla {f}(X)$ onto $T_X\mathcal{H}^{n,p}_+$ (
\cite[Sect.~3.6.1]{absil_optimization_2008}):
\[
    \grad f(X) = P^t_X(\nabla f(X)),
\]
where $P^t_X$ denotes the orthogonal projection onto $T_X\mathcal{H}^{n,p}_+$. 
In order to get a closed-form expression of $P^t_X$, we should characterize the \textit{normal space} to $\mathcal{H}^{n,p}_+$ at $X$, denoted by $(T_X\mathcal{H}^{n,p}_+)^\perp$ or $N_X\mathcal{H}^{n,p}_+$, 
\[
    N_X\mathcal{H}^{n,p}_+ = \{\xi_X \in T_X\mathbb{C}^{n\times n}: \ip{\xi_X,\eta_X}_{\mathbb{C}^{n\times n}} =0 \text{ for all } \eta_X \in T_X\mathcal{H}^{n,p}_+  \},
\]
which is the orthogonal complement of $T_X\mathcal{H}^{n,p}_+$ in $\mathbb{C}^{n\times n}$. 

\begin{lem}\label{thm:normal_space}
	The normal space $N_X\mathcal{H}^{n,p}_+$  at $X = U\Sigma U^*\in \mathcal{H}^{n,p}_+$ is given by
	\begin{equation}\label{eqn:normal_space}
	 N_X\mathcal{H}^{n,p}_+ = \left\{ \bmat{U&U_\perp} \bmat{\Omega &-L^*\\L & M} \bmat{U^* \\ U^*_\perp}\right\},
	 \end{equation}
where $\Omega = - \Omega^* \in \mathbb{C}^{p\times p}$, $M \in \mathbb{C}^{(n-p)\times (n-p)}$and $L \in \mathbb{C}^{(n-p) \times p}$.
\end{lem}

\begin{proof}
First we show that every vector in (\ref{eqn:normal_space}) is orthogonal to $T_X\mathcal{H}^{n,p}_+$. 
Since $U$ is orthonormal, we only need to show that $\left\langle \bmat{H & K^* \\ K & 0}  ,  \bmat{\Omega &-L^*\\L & M}  \right\rangle_{\mathbb{C}^{n\times n}}= 0$ for all $H,K,\Omega, L$ and $M$ defined in  Theorem \ref{thm:tangent_space}  and Lemma \ref{thm:normal_space}. Indeed we have
 \begin{eqnarray*}
 	\left\langle \bmat{H & K^* \\ K & 0}  ,  \bmat{\Omega &-L^*\\L & M}  \right\rangle_{\mathbb{C}^{n\times n}} &=&  \langle \Omega, H \rangle_{\mathbb{C}^{n\times n}} - \langle L^*, K^* \rangle _{\mathbb{C}^{n\times n}}+ \langle L, K \rangle_{\mathbb{C}^{n\times n}} \\
 	&=& \langle \Omega ,H \rangle_{\mathbb{C}^{n\times n}} = 0.
 \end{eqnarray*}

 Next, we count the real dimension of $N_X\mathcal{H}^{n,p}_+$. Remember that a skew-Hermitian matrix has purely imaginary numbers on its diagonal entries, and $\omega_{ij} = - \overline\omega_{ji}$ on its off diagonal entries. So the number of degree of freedoms in $\Omega$ is $p+2\times \frac{p(p-1)}{2} = p^2$. The number of degree of freedoms in $L$ is $2\times p(n-p)$, and the number of degree of freedoms in $M$ is $2\times (n-p)^2$. So, the dimension of $N_X\mathcal{H}^{n,p}_+$ is $2n^2 + p^2 -2np$.  This gives us the desired dimension since the sum of the dimension of the tangent space and its normal space should be $2 n^2$. 
\end{proof}

The orthogonal projection from $\mathbb{C}^{n\times n}$ onto $T_X\mathcal{H}^{n,p}_+$ is given in the following theorem.
\begin{thm}
Let $X = YY^*=U \Sigma U^*$ be the compact SVD for $X\in \mathcal{H}^{n,p}_+$
with $Y\in\mathbb C^{n\times p}_*$.
Let $Z \in \mathbb{C}^{n\times n}$. Then the operator $P_X^t$ defined below is the orthogonal projection onto $T_X\mathcal{H}^{n,p}_+$:
\begin{eqnarray}
    P_X^t(Z) &=& \frac{1}{2}\left(P_Y(Z+Z^*)P_Y + P_Y^\perp(Z+Z^*)P_Y + P_Y(Z+Z^*)P_Y^\perp\right)\notag \\
    &=& \frac{1}{2}\left(P_U(Z+Z^*)P_U + P_U^\perp(Z+Z^*)P_U + P_U(Z+Z^*)P_U^\perp\right) \label{eqn:projection_tangent_space}\\
    &=&  \bmat{U &U_\perp} \bmat{U^*\frac{(Z+Z^*)}{2}U & U^*\frac{(Z+Z^*)}{2}U_\perp \notag \\
		U_\perp^*\frac{(Z+Z^*)}{2}U & 0}\bmat{U^*\\U_\perp^*}, \notag
\end{eqnarray}
where $P_Y = Y(Y^*Y)^{-1}Y^*$, $P_Y^\perp = I - P_Y = P_{Y_\perp}$, $P_U = UU^*$ and $P_U^\perp = I -P_U = P_{U_\perp}$. 
\begin{proof}
	First, observe that 
	\begin{eqnarray*}
		P_X^t(Z) &=& \bmat{P_Y &P_{Y_\perp}}\bmat{\frac{Z+Z^*}{2} &\frac{Z+Z^*}{2} \\\frac{Z+Z^*}{2}&0}\bmat{P_Y \\P_{Y_\perp}}\\
		&=& \bmat{U &U_\perp} \bmat{U^*\frac{(Z+Z^*)}{2}U & U^*\frac{(Z+Z^*)}{2}U_\perp \\
		U_\perp^*\frac{(Z+Z^*)}{2}U & 0}\bmat{U^*\\U_\perp^*}  
	\end{eqnarray*}
	is a tangent vector at $X$. So it suffices to show that $Z - P_X^t(Z)$ is a normal vector. 
	Write $Z$ as $Z = P_YZP_Y + P_YZP_{Y_\perp} + P_{Y_\perp}Z P_Y + P_{Y_\perp}ZP_{Y_\perp} = \bmat{P_Y&P_{Y_\perp}}\bmat{Z&Z \\Z&Z}\bmat{P_Y \\P_{Y_\perp}}$ . Then we have 
	\begin{eqnarray*}
	Z - P_X^t(Z) &=& \bmat{P_Y &P_{Y_\perp}} \bmat{\frac{Z-Z^*}{2}&\frac{Z-Z^*}{2}\\\frac{Z-Z^*}{2}&Z}\bmat{P_Y \\P_{Y_\perp}} \\ 
	&=& \bmat{U &U_\perp} \bmat{U^*\frac{(Z-Z^*)}{2}U & U^*\frac{(Z-Z^*)}{2}U_\perp\\
		U_\perp^*\frac{(Z-Z^*)}{2}U & U_\perp^*ZU_\perp}\bmat{U^* \\ U_\perp^*}
	\end{eqnarray*}
	Hence, $Z - P_X^t(Z)$ is a normal vector, which completes the proof.
\end{proof}
\end{thm}

\begin{rem}
We can write $P^t_X = P^s_X+P^p_X$ by introducing the two operators
\begin{subequations}\label{eqn:projection_sp}
\begin{eqnarray}
    &&P_X^s:  Z \mapsto P_U\frac{Z+Z^*}{2}P_U\\ 
    &&P_X^p:  Z \mapsto P_{U_\perp}\frac{Z+Z^*}{2}P_U+ P_U\frac{Z+Z^*}{2}P_{U_\perp}
\end{eqnarray}
\end{subequations}
\end{rem}

\subsection{A retraction by projection to the embedded manifold}
A retraction is essentially a first-order approximation to the exponential map; see \cite[Def.~4.1.1]{absil_optimization_2008}.
\begin{defn}[\protect{\cite[Def.~4.1.1]{absil_optimization_2008}}] A retraction on a manifold $\mathcal{M}$ is a smooth mapping $R$ from the tangent bundle $T\mathcal{M}$ onto $\mathcal{M}$ with the following properties. Let $R_x$ denote the restriction of $R$ to $T_x\mathcal{M}$.
\begin{enumerate}
    \item $R_x(0_x) = x$, where $0_x$ denotes the zero element of $T_x\mathcal{M}$. 
    \item With the canonical identification $T_{0_x} T_x\mathcal{M} \equiv T_x\mathcal{M}$, $R_x$ satisfies 
    \[
        \D R_x(0_x) = \mbox{id}_{T_x\mathcal{M}},
    \]
    where $\mbox{id}_{T_x\mathcal{M}}$ denotes the identity mapping on $T_x\mathcal{M}.$
\end{enumerate}

\end{defn}

Suppose  $\mathcal{M}$ is an embedded submanifold of a Euclidean space $\mathcal{E}$, then by 
\cite[Props.~3.2 and 3.3]{absil_projection-like_2012}, the mapping $R$ from the tangent bundle $T\mathcal{M}$ to the manifold $\mathcal M$ defined by
\begin{equation}\label{eqn:embedded_retraction_by_projection}
R : \begin{cases} T\mathcal{M} \rightarrow \mathcal{M}\\ (x, u ) \mapsto P_{\mathcal{M}}(x+u) \end{cases}
\end{equation}
is a retraction, where $P_{\mathcal{M}}$ is the orthogonal projection onto the manifold $\mathcal{M}$ with respect to the Euclidean distance, that is,  the closest point.
In our case $\mathcal{M} = \mathcal{H}^{n,p}_+$ and $\mathcal{E} = \mathbb{C}^{n\times n}$. 
Hence, a retraction on $\mathcal{H}^{n,p}_+$ is defined by the truncated SVD:
\begin{equation*}
    R_X(Z) := P_{\mathcal{H}^{n,p}_+}(X+Z)=  \sum_{i=1}^p \sigma_i(X+Z) v_i v_i^*,
\end{equation*}
where $v_i $ is the singular vector of $X+Z$ corresponding to the $i$th largest singular value $\sigma_i(X+Z)$. 

Let $X =U \Sigma U^*\in \mathcal{H}^{n,p}_+$ be the compact SVD and let $Z =  \bmat{U & U_\perp}\bmat{H & K^* \\ K & 0} \bmat{U^* \\ U_\perp^*}   \in T_X \mathcal{H}^{n,p}_+$. Then 
\begin{equation}\label{XplusZ}
X+Z = \bmat{U & U_\perp}\bmat{H+\Sigma & K^* \\ K & 0} \bmat{U^* \\ U_\perp^*} = U(H+\Sigma)U^* + U_\perp K U^* + UK^*U_\perp^* .
\end{equation}

Consider the compact QR factorization of $U_\perp K= Q_K R_K$
where $Q_K$ is $n\times p $ and $R_K$ is $p\times p $. 
Then (\ref{XplusZ}) becomes
\begin{equation}\label{qr}
	X+ Z = U(H + \Sigma ) U^* + Q_KR_K U^*+  (Q_K R_K U^*)^*=\bmat{U & Q_K} \bmat{H+\Sigma & R_K^* \\ R_K &0}\bmat{U^*\\ Q_K^*}. 
\end{equation}
Now notice that $\bmat{H+\Sigma & R_K^* \\ R_K &0}$ from the RHS of (\ref{qr}) is a small  $2p \times 2p$ Hermitian matrix. We can therefore efficiently compute its SVD as
\begin{equation}\label{fast_svd}
	\bmat{H+\Sigma & R_K^* \\ R_K &0} = \bmat{V_1 & V_2} \bmat{S_1 & 0 \\ 0 & S_2}\bmat{V_1^*\\V_2^*},
\end{equation}
where $S_1$ and $S_2$ are $p\times p $ diagonal matrices that contain the singular values of  $\bmat{H+\Sigma & R_K^* \\ R_K &0}  $ in descending order. The matrices $V_1$ and $V_2 $ are $2 p \times p$ contain the corresponding singular vectors.  Combining (\ref{fast_svd}) and (\ref{qr}), we can write $X+Z $ as 
\begin{equation}\label{svdofXplusZ}
	X + Z  =\bmat{U & Q_K}  \bmat{V_1 & V_2} \bmat{S_1 & 0 \\ 0 & S_2}\bmat{V_1^*\\V_2^*}\bmat{U^*\\ Q_K^*}
\end{equation}
with $\bmat{U&Q_K}\bmat{V_1&V_2}$ a unitary matrix. So (\ref{svdofXplusZ}) is the SVD of $X+Z$ with singular values in descending order. The orthogonal projection of $X+Z$ onto the manifold $\mathcal{H}^{n\times p}_+$ is therefore given by 
\[
P_{\mathcal{H}^{n, p}_+} (X+Z) = (\bmat{U & Q_K} V_1)  S_1  (\bmat{U & Q_K} V_1) ^*.
\]

\subsection{Vector transport}
The vector transport is a mapping that transports a tangent vector from one tangent space to another tangent space.

\begin{defn}[\protect{\cite[definition 8.1.1]{absil_optimization_2008}}]
A vector transport on a manifold $\mathcal{M}$ is a smooth mapping
\[
    T\mathcal{M}\oplus T\mathcal{M} \rightarrow T\mathcal{M}:(\eta_x,\xi_x)\mapsto \mathcal{T}_{\eta_x}(\xi_x)\in T\mathcal{M}
\]
satisfying the following properties for all $x\in \mathcal{M}$:
\begin{enumerate}
    \item (Associated retraction) There exists a retraction $R$, called the retraction associated with $\mathcal{T}$, such that the following diagram commutes
    \[ \begin{tikzcd}
(\eta_x,\xi_x) \arrow{r}{\mathcal{T}} \arrow[swap]{d}{} & \mathcal{T}_{\eta_x}(\xi_x) \arrow{d}{\Pi} \\
\eta_x \arrow{r}{R}& \Pi(\mathcal{T}_{\eta_x}(\xi_x))
\end{tikzcd}
\]
where $\Pi(\mathcal{T}_{\eta_x}(\xi_x))$ denotes the foot of the tangent vector $\mathcal{T}_{\eta_x}(\xi_x)$.
    
    \item (Consistency) $\mathcal{T}_{0_x}\xi_x= \xi_x$ for all $\xi_x \in T_x\mathcal{M} $;
    
    \item (Linearity) $\mathcal{T}_{\eta_x}(a \xi_x + b \zeta_x) = a\mathcal{T}_{\eta_x}( \xi_x) + b\mathcal{T}_{\eta_x}( \zeta_x)$.
\end{enumerate}
\end{defn}
Let $\xi_X, \eta_X \in T_X\mathcal{H}^{n,p}_+$ and let $R$ be a retraction on $\mathcal{H}^{n,p}_+$. By \cite[section 8.1.3]{absil_optimization_2008}, the projection of one tangent vector onto another tangent space is a vector transport,
\begin{equation}\label{eqn:vector_transport_embedded}
	\mathcal{T}_{\eta_X} \xi_X := P^t_{R_X(\eta_X)} \xi_X, 
\end{equation}
where $P_Z^t$ is the projection operator onto $T_Z\mathcal{H}^{n,p}_+$.  Namely, we first apply retraction to $X+\eta_X$ to arrive at a new point on the manifold, then we project the old tangent vector $\xi_X$ onto the tangent space at that new point.

Now, we derive the expression of the vector transport (\ref{eqn:vector_transport_embedded}) in closed form. Given $X_1 = U_1 \Sigma_1 U_1^* \in \mathcal{H}^{n,p}_+$, the retracted point $X_2 = U_2 \Sigma_2 U_2^* \in \mathcal{H}^{n,p}_+$,  and a tangent vector $\nu_1 = \bmat{U_1 & {U_1}_\perp}\bmat{H_1 & K_1^* \\ K_1 & 0} \bmat{U_1^* \\ {U_1}_\perp^* } = U_1H_1U_1^* + {U_1}_\perp K_1 U_1^* + U_1 K_1^* {U_1}_\perp^* \in T_{X_1}\mathcal{H}^{n,p}_+ $, we need to determine $H_2$ and $K_2$ of the transported tangent vector $\nu_2 = \bmat{U_2 & {U_2}_\perp}\bmat{H_2 & K_2^* \\ K_2 & 0} \bmat{U_2^* \\ {U_2}_\perp^* } \in T_{X_2}\mathcal{H}^{n,p}_+ $.

By the projection formula (\ref{eqn:projection_tangent_space}), we have 
\[
\nu_2 = P_{X_2}^t (\nu_1) = \bmat{U_2 & {U_2}_\perp } \bmat{U_2^*\nu_1 U_2 & U_2^*\nu_1{U_2}_\perp \\ {U_2}_\perp^* \nu_1 U_2 & 0} \bmat{U_2^* \\ {U_2}_\perp^*}. 
\]
Hence, $H_2$ and $K_2$ satisfy
\begin{eqnarray*}
	H_2 &=& U_2^*\nu_1 U_2 = U_2^* U_1 H_1 U_1^* U_2 + U_2^* {U_1}_\perp K_1 U_1^* U_2+ U_2^* U_1 K_1^* {U_1}_\perp^* U_2,\\
	K_2 &=& {U_2}_\perp^* \nu_1 U_2 =  {U_2}_\perp^* U_1H_1U_1^* U_2+  {U_2}_\perp^* {U_1}_\perp K_1 U_1^* U_2+  {U_2}_\perp^*  U_1 K_1^* {U_1}_\perp^* U_2 . 
\end{eqnarray*}

In implementation, we observe a better numerical performance if we only keep the first term in the above sum of $H_2$ and the second term of $K_2$. That is, we define $H_2$ and $K_2$ by
\begin{subequations}
    \label{embedded-simple-vector-transport}
\begin{eqnarray}
	H_2 &=& U_2^* U_1 H_1 U_1^* U_2 \\
	K_2 &=& {U_2}_\perp^* {U_1}_\perp K_1 U_1^* U_2. 
\end{eqnarray}
One can verify that the vector transport in~\eqref{embedded-simple-vector-transport} is a vector transport by parallelization in~\cite{HAG2017VT}. In numerical tests we have observed that the nonlinear conjugate gradient method using this simpler version of vector transport is usually more efficient. 
So in all our numerical tests, we do not use the more complicated  \eqref{eqn:vector_transport_embedded},
instead we use the following simplified vector transport:
\begin{enumerate}
    \item Given  $X_1 = U_1 \Sigma_1 U_1^* \in \mathcal{H}^{n,p}_+$, and $\eta_{X_1}, \xi_{X_1}\in T_{X_1}\mathcal{H}^{n,p}_+$, first compute $$X_2=R_{X_1}(\eta_{X_1}):=P_{\mathcal{H}^{n,p}_+}(X_1+\eta_{X1})=U_2 \Sigma_2 U_2^*\in \mathcal{H}^{n,p}_+.$$
    \item Let $\xi_{X_1} = \bmat{U_1 & {U_1}_\perp}\bmat{H_1 & K_1^* \\ K_1 & 0} \bmat{U_1^* \\ {U_1}_\perp^* } \in T_{X_1}\mathcal{H}^{n,p}_+$, then compute
    \begin{equation}
    \mathcal{T}_{\eta_{X_1}} \xi_{X_1}=\bmat{U_2 & {U_2}_\perp}\bmat{H_2 & K_2^* \\ K_2 & 0} \bmat{U_2^* \\ {U_2}_\perp^* } \in T_{X_2}\mathcal{H}^{n,p}_+.
\end{equation}
\end{enumerate}
\end{subequations}

\subsection{Riemannian Hessian operator}

For a real-valued function $f(X)$ defined on the Euclidean space $\mathbb{C}^{n\times n}$,
 the Fr\'{e}chet Hessian $\nabla^2 f(X)$ is defined in the sense of the Fr\'{e}chet derivative; see Appendix \ref{frechet_hessian} for the definition of the Fr\'{e}chet Hessian.  The \textit{Riemannian Hessian} of $f$ at $X$, denoted by $\Hess f(X)$ is related to, but different from its  Fr\'{e}chet Hessian. 
 
 \begin{defn}[\protect{\cite[definition 5.5.1]{absil_optimization_2008}}]Given a real-valued function $f$ on a Riemannian manifold $\mathcal{M}$, the \textit{Riemannian Hessian} of $f$ at a point $x$ in $\mathcal{M}$ is the linear mapping $\Hess f(x)$ of $T_x\mathcal{M}$ into itself defined by 
 \[
 \Hess f(x)[\xi_x] = \nabla_{\xi_x} \grad f
 \]
 for all $\xi_x$ in $T_x\mathcal{M}$, where $\nabla$ is the Riemannian connection on $\mathcal{M}$. 
 \end{defn}

The following proposition gives the Riemannian Hessian of $f$. 
The proof follows similar ideas  as in  \cite[Prop.~5.10]{vandereycken_riemannian_2010} and \cite[Prop.~2.3]{vandereycken_low-rank_2012} where a second-order retraction based on a simple power expansion is constructed. We will leave the outline of the proof in Appendix \ref{Appen:riemannian_hessian_embedded}. 

\begin{prop}\label{prop:embedded_hessian}
Let $f(X)$ be a real-valued function defined on $\mathcal{H}^{n,p}_+$. Let $X\in \mathcal{H}^{n,p}_+$ and $\xi_X \in T_X\mathcal{H}^{n,p}_+$. Then the Riemannian Hessian operator of $f$ at $X$ is given by 
\[
 \Hess f(X) [\xi_X] = P^t_X(\nabla^2 f(X)[\xi_X]) + P^p_X\left(\nabla f(X)(X^\dagger\xi_X^p)^* + (\xi_X^p X^\dagger)^*\nabla f(X)\right)
\]
 where  $\xi_X^s = P_X^s(\xi_X)$ and $\xi_X^p = P_X^p(\xi_X)$ and $P^t_X $ and $P^p_X$ are defined in \eqref{eqn:projection_sp}. 
\end{prop}

\section{The quotient geometry of $\mathbb{C}^{n\times p}_*/\mathcal{O}_p$ using three Riemannian metrics}\label{section:quotient_manifold}
Besides being regarded as an embedded manifold in $\mathbb C^{n\times n}$, $\mathcal{H}^{n,p}_+$ can also be viewed as a quotient set $\mathbb{C}^{n\times p}_*/\mathcal{O}_p$ since any $X\in \mathcal{H}^{n,p}_+$ can be written as $X = YY^*$ with $Y \in \mathbb{C}^{n\times p}_*$. But there is an ambiguity in $Y$ because such an $Y$ is not uniquely determined by $X$. We define an equivalence relation on $\mathbb{C}^{n\times p}_*$ through the smooth Lie group action of $\mathcal{O}_p$ on the manifold $\mathbb{C}^{n\times p}_*$:
\[
\begin{aligned}
     \mathbb{C}^{n\times p}_* \times \mathcal{O}_p &\rightarrow& \mathbb{C}^{n\times p}_* \\
    (Y,O) &\mapsto& YO.
\end{aligned}
\]
This action defines an equivalence relation on $\mathbb{C}^{n\times p}_*$ by setting $Y_1 \sim Y_2$ if there exists an $O \in \mathcal{O}_p$ such that $Y_1 = Y_2 O$. Hence we have constructed a quotient space $\mathbb{C}^{n\times p}_*/\mathcal{O}_p$ that removes this ambiguity. The set $\mathbb{C}^{n\times p}_*$ is called the \textit{total space} of $\mathbb{C}^{n\times p}_*/\mathcal{O}_p$. 

Denote the natural projection as
$$\pi: \mathbb{C}^{n\times p}_* \rightarrow \mathbb{C}^{n\times p}_*/\mathcal{O}_p.$$ 
For any $Y\in \mathbb C_*^{n\times p}$,  $\pi(Y)$ is an element in $\mathbb{C}^{n\times p}_* /\mathcal{O}_p$. We denote the equivalence class containing $Y$ as 
$$[Y]=\pi^{-1}(\pi(Y)) =\left\{ YO \vert O\in \mathcal{O}_p \right\}. $$

Define $$\begin{aligned} \beta : \mathbb{C}^{n\times p}_* & \rightarrow \mathcal{H}^{n, p }_+\\
 Y &\mapsto YY^*.\end{aligned}$$  Then $\beta$ is invariant under the equivalence relation $\sim$ and induces a unique function $\Tilde{\beta}$ on $\mathbb{C}^{n\times p}_*/\mathcal{O}_p$, called the projection of $\beta$, such that $\beta = \Tilde{\beta} \circ \pi$ \cite[section 3.4.2]{absil_optimization_2008}. One can easily check that $\tilde{\beta}$ is a bijection.
For any real-valued function $f$ defined on $\mathcal{H}^{n,p}_+$, there is a real-valued function $F$ defined on $\mathbb{C}^{n\times p}_*$ that induces $f$: for any $X = YY^* \in \mathcal{H}^{n,p}_+$, $F(Y) := f\circ \beta (Y) = f(YY^*)$. This is summarized in the diagram below:
\[
  \begin{tikzcd}[baseline=\the\dimexpr\fontdimen22\textfont2\relax]
     \mathbb{C}^{n\times p}_* \arrow[rd,"\beta:= \tilde{\beta}\circ \pi",dashrightarrow] \arrow[d,"\pi"] & & \\
     \mathbb{C}^{n\times p}_*/\mathcal{O}_p \arrow[r,leftrightarrow,"\tilde{\beta}"]  & \mathcal{H}^{n,p}_+ \arrow[r,"f"] & \mathbb{R}
  \end{tikzcd}
\]

The next theorem shows that $\mathbb{C}^{n\times p}_*/\mathcal{O}_p$ is a smooth manifold. 
 
\begin{thm}
The quotient space $\mathbb{C}^{n\times p}_*/\mathcal{O}_p$ is a quotient manifold over $\mathbb R$ of dimension  $2np - p^2$ and has a unique smooth structure such that the natural projection $\pi$ is a smooth submersion. 
\begin{proof}
The proof follows from Corollary 21.6 and Theorem 21.10 of \cite{lee_introduction_2012}. 
\end{proof}
\end{thm}

The next theorem shows that $\mathcal{H}^{n,p}_+$ and $\mathbb{C}^{n\times p}_*/\mathcal{O}_p$ are essentially the same in the sense that there is a diffeomorphism between them. The proof uses the same technique in \cite[Prop.~A.7]{massart_quotient_2020}
\begin{thm}
The quotient manifold $\mathbb{C}^{n\times p}_*/\mathcal{O}_p$ is diffeomorphic to  $\mathcal{H}^{n,p}_+$ under $\tilde{\beta}$.
\begin{proof}
Recall from Theorem \ref{thm:tangent_space}, any tangent vector in $T_{\beta(Y)}\mathcal{H}^{n,p}_+$ can be written as 
\[
    \zeta_{\beta(Y)} = YHY^* + Y_\perp KY^* + YK^*Y_\perp^*,
\]
where $Y_\perp$ has orthonormal columns. 
Let $V = YH/2 + Y_\perp K$, then $D\beta(Y)[V] = \zeta_{\beta(Y)}$. This implies that $\beta$ is a submersion. 

Now notice that $\pi = \Tilde{\beta}^{-1} \circ \beta$ and $\beta = \Tilde{\beta} \circ \pi$. By \cite[Prop.~6.1.2]{brickell_differentiable_1970}, we conclude that $\Tilde{\beta}^{-1} $ and $\Tilde{\beta}$ are both differentiable. So $\Tilde{\beta}$ is a diffeomorphism between $\mathbb{C}^{n\times p}_*/\mathcal{O}_p$ 
and $\mathcal{H}^{n,p}_+$.
\end{proof}
\end{thm} 

% \subsection{Three metrics}

\subsection{Vertical space, three Riemannian metrics and horizontal space}\label{Sec:Vertical space, three Riemannian metrics and horizontal space}

The equivalence class $[Y]$ is an embedded submanifold of  $\mathbb{C}^{n\times p}_*$(\cite[Prop.~3.4.4]{absil_optimization_2008}). The tangent space of $[Y] $ at $Y$ is therefore a subspace of $T_Y \mathbb{C}^{n\times p}_*$ called the \textit{vertical space} at $Y$ and is denoted by $\mathcal{V}_Y $. The following proposition characterizes $\mathcal{V}_Y$. 
 
 \begin{prop}
The vertical space at $Y\in [Y]=\left\{ YO \vert O\in \mathcal{O}_p \right\} $, which is the tangent space of $[Y] $ at $Y$ is 
\[
    \mathcal{V}_Y = \left\{Y\Omega\vert \Omega^*  = - \Omega, \Omega \in \mathbb{C}^{p\times p} \right\}. 
\]
\end{prop}
\begin{proof}
The tangent space of $\mathcal{O}_p$ at $I_p$ is  $T_{I_p} \mathcal{O}_p = \{ \Omega: \Omega^* = - \Omega ,\Omega \in \mathbb{C}^{p\times p}\}$, which is also the set $\{\gamma'(0): \gamma \text{ is a curve in } \mathcal{O}_p, \gamma(0) = I_p\}$.  Hence $T_Y\left\{YO\vert O\in \mathcal{O}_p \right\} = \{ Y\gamma'(0): \gamma \text{ is a curve in } \mathcal{O}_p, \gamma(0) = I_p\} = \left\{Y\Omega\vert \Omega^*  = - \Omega, \Omega \in \mathbb{C}^{p\times p} \right\} $.
\end{proof}

 A \textit{Riemannian metric} $g$ is a smoothly varying inner product defined on the tangent space. That is, $g_Y(\cdot,\cdot)$ is an inner product on $T_Y\mathbb{C}^{n\times p}_*$. Once we choose a Riemannian metric $g$ for $\mathbb{C}^{n\times p}_*$, we can obtain the orthogonal complement in $T_Y \mathbb{C}^{n\times p}_*$ of $\mathcal{V}_Y$ with respect to the metric. In other words, we choose the \textit{horizontal distribution} as orthogonal complement w.r.t. Riemannian metric, see \cite[Section 3.5.8]{absil_optimization_2008}. This orthogonal complement to $\mathcal{V}_Y$ is called \textit{horizontal space} at $Y$ and is denoted by $\mathcal{H}_Y$. We thus have
\begin{equation}
\label{tangentspace-decomp}
    T_Y\mathbb{C}^{n\times p}_*= \mathcal{H}_Y \oplus \mathcal{V}_Y.
\end{equation}

Once we have the horizontal space, there exists a unique vector $\bar{\xi}_Y \in \mathcal{H}_Y$ that satisfies $\D\pi(Y)[\bar{\xi}_Y] = \xi_{\pi(Y)}$ for each $\xi_{\pi(Y)} \in T_{\pi(Y)}\mathbb{C}^{n\times p}_*/\mathcal{O}_p$. This  $\bar{\xi}_Y$ is called the \textit{horizontal lift} of $\xi_{\pi(Y)}$ at $Y$.

There exist more than one choice of Riemannian metric on $\mathbb{C}^{n\times p}_*$. Different Riemanian metrics do not affect the vertical space, but  generally result in different horizontal spaces. In this paper, we discuss three  Riemannian metrics on $\mathbb{C}^{n\times p}_*$ and study how each metric affects the convergence of Riemannian optimization algorithms.

\subsubsection{The  Bures-Wasserstein metric}
\label{sec-bwmetric}
The most straightforward choice of a Riemannian metric on $\mathbb{C}^{n\times p}_*$ is the canonical Euclidean inner product on $\mathbb{C}^{n\times p}$ defined by 
\[
    g_Y^1(A,B) := \ip{A,B}_{\mathbb{C}^{n\times p}} = \Re(tr(A^*B)), \quad \forall A,B\in T_Y\mathbb{C}^{n\times p}_* = \mathbb{C}^{n\times p}. 
\]

The metric $g^1$ is also called the  Bures-Wasserstein metric \cite{massart2019curvature} for the 
quotient manifold $\mathbb{C}^{n\times p}_*/\mathcal{O}_p$. On the other hand,
the following metric for Hermitian positive-definite matrices $\mathcal{H}^{n,n}_+$ \cite{oostrum2022bures, bhatia2019bures, han2021riemannian}
is also called the Bures-Wasserstein metric.
\begin{defn}[The Bures-Wasserstein metric for $\mathcal{H}^{n,n}_+$]\label{def:bw_metric}
Let $X \in \mathcal{H}^{n,n}_+$ and $A,B \in T_X \mathcal{H}^{n,n}_+$. Then 
\[
g^{\mbox{BW}}_{X}(A,B) := \frac{1}{2} \ip{\mathcal{L}_{X}(A),B},
\]
where $\mathcal{L}_X(A) = M$ solves the following Lyapunov equation
\begin{equation}\label{eqn:lyapunov_M}
X M + MX = A
\end{equation}
which has a unique solution provided $X$ is Hermitian  positive-definite. 
\end{defn}

Notice that it is not clear whether  Definition~\ref{def:bw_metric} can also apply to a low-rank matrix $X\in \mathcal{H}^{n,p}_+$. \textcolor{black}{In this subsection, we show how the metric $g^1$ can be used to generalize Definition~\ref{def:bw_metric} to Definition~\ref{def:pBW_metric}, which defines the Bures-Wasserstein metric in the low-rank case $\mathcal{H}^{n,p}_+$. This non-trivial generalization is presented as Theorem~\ref{lem:pBW_metric}. Although the theorem is not the primary focus of this paper, it is of interest to see how $g^1$ connects the Bures-Wasserstein metric on the quotient manifold to its counterpart on the embedded manifold. 
}

% Next we will show that    
% the metric $g^1$ indeed induces a metric on $ \mathcal{H}^{n,p}_+$ which is consistent with Definition \ref{def:bw_metric} when $p=n$.
% Recall that $\tilde{\beta}$ is a diffeomorphism between $\mathbb{C}^{n\times p}_*/\mathcal{O}_p$ and $\mathcal{H}^{n,p}_+$, and $\D \tilde{\beta}(\pi(Y))$ is an isomorphism between $T_{\pi(Y)} \mathbb{C}^{n\times p}_*/\mathcal{O}_p$ and $T_{YY^*} \mathcal{H}^{n,p}_+$. Hence for any $\in T_{YY^*} \mathcal{H}^{n,p}_+$, there exists a unique $\xi_Y \in \mathcal{H}^1_{Y} $ such that 
% \[
% Y\xi_Y^* + \xi_Y Y^* = A. 
% \] 

\begin{defn}[The Bures-Wasserstein metric on $\mathcal{H}^{n,p}_+$]\label{def:pBW_metric}
Let $A, B \in T_{YY^*}\mathcal{H}^{n,p}_+$, then by the 1-to-1 correspondence between $T_{YY^*}\mathcal{H}^{n,p}_+$ and the horizontal space $\mathcal{H}^1_Y$, there exist unique $\xi_Y, \eta_Y \in \mathcal{H}^1_Y$ such that   $A = Y\xi_Y^* + \xi_Y Y^*$ and $B = Y\eta_Y^* + \eta_Y Y^*$. 
\textcolor{black}{We define the Bures-Wasserstein metric at the low-rank $X = YY^*$ as}
    \begin{equation*}
        g_{YY^*}^{\mbox{BW}}(A,B) := g_Y^1(\xi_Y, \eta_Y).
    \end{equation*}
\end{defn}

\begin{lem}
    For any $A, B \in T_{X}\mathcal{H}^{n,p}_+$ with $X=YY^*$, there is a unique solution $M \in T_X\mathcal{H}^{n,p}_+$ satisfying both
\begin{equation}
\label{low-rank-lyapunov}
    Y^*X MY + Y^*MXY = Y^*AY
\end{equation}  and
\begin{equation}\label{eqn:gbw_equation2}
    g_{YY^*}^{\mbox{BW}}(A,B) = \frac{1}{2} \ip{M,B}_{\mathbb{C}^{n\times n}}.
\end{equation}
\end{lem}
\begin{proof}
Let $\xi_Y = Y(Y^*Y)^{-1} S + Y_\perp K \in \mathcal{H}^1_Y$ with $S^*=S$ be the unique horizontal vector such that $A = Y\xi_Y^* + \xi_Y Y^*$.  Let $Y = UR$ where $U$ has size $n$-by-$p$ with orthonormal columns and $R$ is an $p$-by-$p$ invertible matrix. 
Thus \eqref{low-rank-lyapunov} is equivalent to 
\begin{equation}\label{eqn:lyapunov_bw}
    RR^* (U^*MU) + (U^*MU) RR^* = RSR^{-1} + (R^*)^{-1} S R^*.
\end{equation}
Since $RR^*$ is positive definite, \eqref{eqn:lyapunov_bw} has a unique solution in $U^*MU$; see  Remark \ref{rem-Lyapunov} below, which can be written explicitly: 
\begin{equation}
    U^*MU = (R^*)^{-1}SR^{-1}. 
\end{equation}

Thus $M = \bmat{U& Y_\perp} \bmat{ (R^*)^{-1}SR^{-1} & K_M^* \\ K_M & 0 }\bmat{U^* \\ Y_\perp ^*}$, where $K_M$ is to be determined by the additional equation \eqref{eqn:gbw_equation2}. With  $ B = Y\eta_Y^* + \eta_Y Y^*$ we have,  
    \[
    \frac{1}{2}\ip{M,B}_{\mathbb{C}^{n\times n}} = \frac{1}{2}\ip{M,Y\eta_Y^* }_{\mathbb{C}^{n\times n}} +\frac{1}{2}\ip{M,  \eta_Y Y^*}_{\mathbb{C}^{n\times n}}=\ip{M Y,  \eta_Y}_{\mathbb{C}^{n\times p}}. 
    \] 
Thus in order for \eqref{eqn:gbw_equation2} to hold, $M$ needs to satisfy $MY = \xi_Y$. Recall that $ \xi_Y = Y (Y^*Y)^{-1} S + Y_\perp K = U(R^*)^{-1} S + Y_\perp K$. Thus $K_M$ needs to satisfy $Y_\perp K_M R = Y_\perp K$, which gives the unique $K_M = K R^{-1}$. 
\end{proof}

\begin{rem}
\label{rem-Lyapunov}
The solution $X$ to the Lyapunov equation $XE +EX = Z$ for a Hermitian $E$ is unique if $E$ is Hermitian positive-definite \cite[Section 2.2]{massart_quotient_2020}. Let $E = U\Lambda U^*$ be the SVD, then the Lyapunov equation  $XE +EX = Z$ becomes 
\begin{equation*}
     (U^*XU)\Lambda +\Lambda (U^*XU)  = U^*ZU, 
\end{equation*}
which gives the solution 
\begin{equation*}
    (U^*XU)_{i,j} = (U^*ZU)_{i,j}/(\Lambda_{i,i}+\Lambda_{j,j}).
\end{equation*}
\end{rem}

Now we can show  that Definition~\ref{def:bw_metric} generalizes  the Definition~\ref{def:pBW_metric} and defines the Bures-Wasserstein metric in the low-rank case $\mathcal{H}^{n,p}_+$ in the following theorem.
\begin{thm}[Equivalence of the two Bures-Wasserstein metrics] \label{lem:pBW_metric}
If $p=n$, then the Definition~\ref{def:pBW_metric} reduces to the Definition~\ref{def:bw_metric}.  
\end{thm}
\begin{proof} 
For the case $p=n$, $Y$ is invertible, thus 
\eqref{low-rank-lyapunov} is equivalent to the Lyapunov equation \eqref{eqn:lyapunov_M}. Therefore, the Definition \ref{def:pBW_metric}
indeed reduces to the Definition \ref{def:bw_metric} when $p=n$.
\end{proof}

The next proposition characterizes the horizontal space for metric $g^1$.  
\begin{prop}\label{prop:horizontal-space-1}
Under metric $g^1$, the horizontal space at $Y$ satisfies 
\begin{eqnarray*}
    \mathcal{H}^1_Y &=& \left\{Z\in \mathbb{C}^{n\times p}: Y^*Z =Z^*Y \right\} \\
                    &=& \left\{Y(Y^*Y)^{-1}S +Y_\perp K \vert S^* = S, S \in \mathbb{C}^{p\times p}, K \in \mathbb{C}^{(n-p)\times p} \right\},
\end{eqnarray*}
where $Y_\perp$ has orthonormal columns. 
\end{prop}
\begin{proof}
    The result of real case can be found in \cite{massart2019curvature} but the proof was omitted. For completeness, we outline the proof here.  $Z \in \mathbb C^ {n\times p}$ belongs to $\mathcal H ^1_Y$ if and only if  $Z$ is orthogonal to $\mathcal{V}_Y$ under the metric $g^1_Y$, i.e., $g^1_Y(Z,Y\Omega) = \ip{Z,Y\Omega}_{\mathbb{C}^{n\times p}}= \ip{Y^*Z,\Omega}_{\mathbb C^{n\times p}}=0, \forall \Omega 
 = - \Omega^*$. This is equivalent to $Y^*Z = Z^*Y$. The second equality can be obtained by writing any $Z\in \mathcal{H}_Y^1$ as $Z = Y(Y^*Y)^{-1} S + Y_\perp K$ as $Y(Y^*Y)^{-1}$ and $Y_\perp$ forms a basis for the column space of $\mathbb C^{n\times p}$, and verify that $S=S^*$
\end{proof}

\subsubsection{The second quotient metric}
Another Riemannian metric used in \cite{huang_solving_2017, huang2013optimization} is defined by 
\[
    g_Y^2(A,B) := \ip{AY^*,BY^*}_{\mathbb{C}^{n\times n}} = \Re(tr((Y^*Y)A^*B)),\quad \forall A,B\in T_Y\mathbb{C}^{n\times p}_* = \mathbb{C}^{n\times p}. 
\]

\begin{prop}
Under metric $g^2$, the horizontal space at $Y$ satisfies 
\begin{eqnarray*}
    \mathcal{H}^2_Y &=& \left\{Z\in \mathbb{C}^{n\times p}: (Y^*Y)^{-1}Y^* Z = Z^*Y(Y^*Y)^{-1}\right\} \\
                    &=& \left\{YS +Y_\perp K \vert S^* = S, S \in \mathbb{C}^{p\times p}, K \in \mathbb{C}^{(n-p)\times p} \right\}.
\end{eqnarray*}
\end{prop}
\begin{proof}
    The proof follows
the same idea used in proving Proposition \ref{prop:horizontal-space-1}. 
\end{proof}

\subsubsection{The third quotient metric}
The third Riemannian metric for $\mathbb{C}^{n\times p}_*$ is motivated by the Riemannian metric of $\mathcal{H}^{n,p}_+$ and the diffeomorphism between $\mathbb{C}^{n\times p}_*/\mathcal{O}_p$ and $\mathcal{H}^{n,p}_+$. We know that $\beta$ is a submersion. Every tangent vector of $\mathcal{H}^{n,p}_+$  corresponds to a tangent vector of $\mathbb{C}^{n\times p}_*$. We can use the Riemannian metric of $\mathcal{H}^{n,p}_+$ and the correspondence of tangent vectors between $\mathcal{H}^{n,p}_+$ and $\mathbb{C}^{n\times p}_*$ to  define a Riemannian metric for $\mathbb{C}^{n\times p}_*$.
A natural first attempt would be to use
\[
    g_Y(A,B) := \ip{\D\beta(Y)[A],\D\beta(Y)[B]}_{\mathbb{C}^{n\times n}}=  \ip{YA^*+AY^*,YB^*+BY^*}_{\mathbb{C}^{n\times n}}, 
\]
which is however not a Riemannian metric because it is not positive-definite. To see this, notice that $\ker(\D\beta(Y)[\cdot]) = \mathcal{V}_Y$. Consider $C \neq 0 \in \mathcal{V}_Y$, then $g_Y^3(C,C) = 0$. 
To modify this definition for $g^3$, we can use the Riemannian metric $g^2$ and the decomposition $T_Y\mathbb{C}^{n\times p}_*= \mathcal{H}^2_Y \oplus \mathcal{V}_Y$, by which $A \in T_Y\mathbb{C}^{n\times p}_*$ can be uniquely decomposed as 
\[
    A = A^{\mathcal{V}} +A^{\mathcal{H}^2},
\]
where $A^{\mathcal{V}} \in \mathcal{V}_Y$ and $ A^{\mathcal{H}^2} \in \mathcal{H}^2_Y$. Now define $g^3$ as
\begin{eqnarray*}
    g^3_Y(A,B) &:=& \ip{\D\beta(Y)[A^{\mathcal{H}^2}],\D\beta(Y)[B^{\mathcal{H}^2}]}_{\mathbb{C}^{n\times n}} + g^2_Y\left(A^\mathcal{V},B^\mathcal{V}\right) \\ 
    &=& \ip{\D\beta(Y)[A],\D\beta(Y)[B]}_{\mathbb{C}^{n\times n}} + \ip{P^\mathcal{V}_Y(A)Y^*,P^\mathcal{V}_Y(B)Y^*}_{\mathbb{C}^{n\times n}},\\
    &=&\ip{YA^*+AY^*,YB^*+BY^*}_{\mathbb{C}^{n\times n}}+ \ip{P^\mathcal{V}_Y(A)Y^*,P^\mathcal{V}_Y(B)Y^*}_{\mathbb{C}^{n\times n}}
\end{eqnarray*}
where $P_Y^\mathcal{V}$ is the projection of any tangent vector of $\mathbb{C}^{n\times p}_*$ to the vertical space $\mathcal{V}_Y$. It is straightforward to verify that $g^3$ defined above is now a Riemannian metric.
 With the definition \eqref{real-inner-product}, the properties $tr(UV)=tr(VU)$ for two matrices $U,V$ and $\Re(tr(C+C^*))=2\Re(tr(C))$, we have 
\begin{equation}
    \label{metric-3-horizontal}
\forall A, B\in {\mathcal{H}_Y^2}, \quad 
g^3_Y(A,B)=\ip{YA^*+AY^*,YB^*+BY^*}_{\mathbb{C}^{n\times n}}=2\ip{AY^*Y+YA^*Y, B}_{\mathbb{C}^{n\times p}}.
\end{equation}

\begin{prop} 
Under metric $g^3$, the horizontal space at $Y$ is the same set as $\mathcal{H}_Y^2$. That is,
\begin{eqnarray*}
    \mathcal{H}^3_Y &=& \left\{Z\in \mathbb{C}^{n\times p}: (Y^*Y)^{-1}Y^* Z = Z^*Y(Y^*Y)^{-1}\right\} \\
                    &=& \left\{YS +Y_\perp K \vert S^* = S, S \in \mathbb{C}^{p\times p}, K \in \mathbb{C}^{(n-p)\times p} \right\}.
\end{eqnarray*}
\end{prop}
\begin{proof}
$Z \in \mathcal{H}^3_Y$ if and only if $g^3_Y(Z, Y\Omega) = 0$ for all $\Omega = \Omega^*$. That is, $\forall \Omega = \Omega^*$, 
\[
\ip{YZ^*+ZY^*, 2Y\Omega Y^*}_{\mathbb{C}^{n\times n}} + \ip{P^\mathcal{V}_Y(Z)Y^*,Y\Omega Y^*}_{\mathbb{C}^{n\times n}} = 0. 
\]
Hence we must have 
\begin{subequations}
\begin{equation}
\ip{YZ^*+ZY^*, 2Y\Omega Y^*}_{\mathbb{C}^{n\times n}} =0 \label{eqn:projection_horizontal_g3a}
\end{equation}
and
\begin{equation}\label{eqn:projection_horizontal_g3b}
\ip{P^\mathcal{V}_Y(Z)Y^*,Y\Omega Y^*}_{\mathbb{C}^{n\times n}} = 0. 
\end{equation}
\end{subequations}
\eqref{eqn:projection_horizontal_g3a}  is equivalent to 
\[
\ip{YZY^*Y, \Omega}_{\mathbb{C}^{n\times n}} = 0 \quad \forall \Omega = \Omega^*. 
\]
Hence $YZY^*Y$ must be Hermitian since $\Omega$ is an arbitrary skew Hermitian matrix.  Therefore $Z$ is  in $\mathcal{H}^2_Y$ as well and hence \eqref{eqn:projection_horizontal_g3b} is satisfied. 

Thus we have shown that 
\[
\mathcal{H}^3_Y = \mathcal{H}^2_Y = \left\{Z\in \mathbb{C}^{n\times p}: (Y^*Y)^{-1}Y^* Z = Z^*Y(Y^*Y)^{-1}\right\}.
\]
\end{proof}

\subsection{Projections onto vertical space and horizontal space}\label{section:quotient_projection}
Due to the direct sum property
\eqref{tangentspace-decomp}, for our choices of $\mathcal{H}_Y^i$, there exist projection operators for any $A \in T_Y\mathbb{C}^{n\times p}_*$ to $\mathcal{H}_Y^i$ as
\begin{equation*}
    A = P_Y^{\mathcal{V}}(A) + P_Y^{\mathcal{H}^i}(A).
\end{equation*}

It is straightforward to verify the following formulae for projection operators $P_Y^\mathcal{V}$ and $P_Y^{\mathcal{H}^i}$.

\begin{prop}
If we use $g^1$ as our Riemannian metric on $\mathbb{C}^{n\times p}_*$, then the orthogonal projections  of any $A \in \mathbb{C}^{n \times p}$ to $\mathcal{V}_Y$ and $\mathcal{H}^1_Y$ are 
\begin{equation*}
    P^\mathcal{V}_Y(A) = Y\Omega,\quad  P^{\mathcal{H}^1}_Y (A)= A - Y\Omega,
\end{equation*}
where $\Omega$ is the skew-symmetric matrix that solves the Lyapunov equation
\[
    \Omega Y^*Y + Y^*Y \Omega = Y^* A - A^* Y.
\]
\end{prop}
\begin{prop}
If we use $g^2$ as our Riemannian metric on $\mathbb{C}^{n\times p}_*$, then the orthogonal projection of any $A \in \mathbb{C}^{n\times p}$ to vertical space $\mathcal{V}_Y$ satisfies 
\begin{equation*}
    P_Y^\mathcal{V} (A) = Y \left(\frac{(Y^*Y)^{-1}Y^*A - A^* Y(Y^*Y)^{-1}}{2} \right) = YSkew\left((Y^*Y)^{-1}Y^* A\right),
\end{equation*}
and the orthogonal projection of any $A\in \mathbb{C}^{n\times p }$ to the horizontal space $\mathcal{H}_Y^2$ is 
\begin{eqnarray*}
    P_Y^{\mathcal{H}^2} (A ) &=& A - P_Y^\mathcal{V}(A)\\
    &=& Y \left(\frac{(Y^*Y)^{-1}Y^*A + A^* Y(Y^*Y)^{-1}}{2}\right) + Y_\perp Y_\perp^*A \\
    &=& YHerm\left((Y^*Y)^{-1}Y^*A\right) + Y_\perp Y_\perp^*A.
\end{eqnarray*}
\end{prop}

\begin{prop}
If we use $g^3$ as our Riemannian metric on $\mathbb{C}^{n\times p}_*$, then the orthogonal projection of any $A \in \mathbb{C}^{n\times p}$ to vertical space $\mathcal{V}_Y$ satifies 
\begin{equation*}
    P_Y^\mathcal{V} (A) = Y \left(\frac{(Y^*Y)^{-1}Y^*A - A^* Y(Y^*Y)^{-1}}{2} \right) = Yskew((Y^*Y)^{-1}Y^* A),
\end{equation*}
and the orthogonal projection of any $A \in \mathbb{C}^{n\times p}$ to the horizontal space $\mathcal{H}_Y^3$ is 
\begin{eqnarray*}
    P_Y^{\mathcal{H}^3} (A) &=& A - P_Y^\mathcal{V}(A)\\
    &=& Y \left(\frac{(Y^*Y)^{-1}Y^*A + A^* Y(Y^*Y)^{-1}}{2}\right) + Y_\perp Y_\perp^*A \\
    &=& YHerm\left((Y^*Y)^{-1}Y^*A\right) + Y_\perp Y_\perp^*A.
\end{eqnarray*}
\end{prop}

\subsection{$\mathbb{C}^{n\times p}_*/\mathcal{O}_p$ as  Riemannian quotient manifold}
First we show in the following lemma the relationship between the horizontal lifts of the  quotient tangent vector $\xi_{\pi(Y)}$ lifted at different representatives in $[Y]$. 
\begin{lem}\label{lem:horizontal_lift}
Let $\eta$ be a vector field on $\mathbb{C}^{n\times p}_*/\mathcal{O}_p$, and let $\bar{\eta}$ be the horizontal lift of $\eta$. Then for each $Y\in \mathbb{C}^{n\times p}_*$, we have $$\bar{\eta}_{YO} = \bar{\eta}_YO$$ for all $O\in \mathcal{O}_p$. 
\end{lem}
\begin{proof}
 \cite[Prop.~A.8]{massart_quotient_2020} gives a proof based on metric $g^1$ for real case, and \cite[Lemma 5.1]{huang_solving_2017} proves the result for metric $g^2$. For completeness we will provide a proof applying to all three metrics $g^i$. 
 
By the definition of horizontal lift we have
\[
\eta_{\pi(Y)} = \eta_{\pi(YO)} = \D\pi(YO)[\overline{\eta}_{YO}]. 
\]
On the other hand, notice that $\pi(Y) = \pi(YO)$. Take the differential w.r.t. $Y$ we have 
\[
\D\pi(Y)[A] = \D \pi(YO)[A O]\quad \forall A \in \mathbb{C}^{n\times p}. 
\]
In particular, let $A = \overline{\eta}_Y \in \mathcal{H}_Y$ we have 
\[
\eta_{\pi(Y)} = \D\pi(Y)[\overline{\eta}_Y] = \D \pi(YO)[\overline{\eta}_Y O].
\]

Thus we have
\[
\D \pi(YO)[\overline{\eta}_Y O] = \D\pi(YO)[\overline{\eta}_{YO}]
\]
So 
\[
\overline{\eta}_Y O - \overline{\eta}_{YO} \in \ker(\D \pi(YO)[\cdot]) = \mathcal{V}_{YO}. 
\]

Now one can verify that for each $g^i$ and $YO\Omega \in \mathcal{V}_{YO}$, $g_{YO}^i(\overline{\eta}_YO, YO\Omega) = 0$. So $\overline{\eta}_YO$ is orthogonal to  $\mathcal{V}_Y$ and  hence $\overline{\eta}_YO \in \mathcal{H}^i_{YO}$. So we have
\[
\overline{\eta}_Y O - \overline{\eta}_{YO} \in  \mathcal{H}_{YO}^i.
\]
Therefore $\overline{\eta}_Y O - \overline{\eta}_{YO} \in \mathcal{V}_{YO} \cap \mathcal{H}^i_{YO} = \{ 0 \}$ and we completes the proof.

\end{proof}

Recall from \cite[Section 3.6.2]{absil_optimization_2008} that if the expression $g_{Y}(\bar{\xi}_Y,\bar{\zeta}_Y)$ does not depend on the choice of $Y\in [Y]$ for every $\pi(Y) \in \mathbb{C}^{n\times p}_*/\mathcal{O}_p$ and every $\xi_{\pi(Y)}, \zeta_{\pi(Y)}\in T_{\pi(Y)}\mathbb{C}^{n\times p}_*/\mathcal{O}_p$, then 
\begin{equation}\label{eqn:riemannian_metric_quotient}
    g_{\pi(Y)}(\xi_{\pi(Y)},\zeta_{\pi(Y)}) := g_Y(\bar{\xi}_Y,\bar{\zeta}_Y)
\end{equation}
defines a Riemannian metric on the quotient manifold $\mathbb{C}^{n\times p}_*/\mathcal{O}_p$. By Lemma \ref{lem:horizontal_lift}, it is straightforward to verify that each Riemannian metric $g^i$ on $\mathbb{C}^{n\times p}_*$ induces a Riemannian metric on $\mathbb{C}^{n\times p}_*/\mathcal{O}_p$. The quotient manifold $\mathbb{C}^{n\times p}_*/\mathcal{O}_p$ endowed with a Riemannian metric defined in (\ref{eqn:riemannian_metric_quotient}) is called a \textit{Riemannian quotient manifold}. By abuse of notation, we use $g^i$ for denoting Riemannian metrics on both total space $\mathbb{C}^{n\times p}_*$ and quotient space $\mathbb{C}^{n\times p}_*/\mathcal{O}_p$.

\subsection{Riemannian gradient}
Recall that  $\mathbb{C}^{n\times p}_*/\mathcal{O}_p$ is diffeomorphic to $\mathcal{H}^{n,p}_+$ under $\tilde{\beta}$. Given a smooth real-valued function $f$  on $\mathcal{H}^{n,p}_+$, the corresponding cost function $h$ on $\mathbb{C}^{n\times p}_*/\mathcal{O}_p$ satisfies
\begin{equation}\label{eqn:cost_function_quotient}
\begin{aligned}
    h: \mathbb{C}^{n\times p}_*/\mathcal{O}_p & \rightarrow \mathbb{R}\\
     \pi(Y) & \mapsto f(\tilde{\beta}(\pi(Y))) = f(\beta(Y)) = f(YY^*). 
\end{aligned}
\end{equation}

Observe that the function $F(Y):=f(YY^*)$ satisfies  $F(Y)= h\circ \pi(Y) = f \circ \beta(Y)$. 

The Riemannian gradient of $h$ at $\pi(Y)$ is a tangent vector in $T_{\pi(Y)}\mathbb{C}^{n\times p}_*/\mathcal{O}_p$ . The next theorem shows that the horizontal lift of $\grad h(\pi(Y))$ can be obtained from the Riemannian gradient of $F$ defined on $\mathbb{C}^{n\times p}_*$. 
\begin{thm}\label{thm:lift_gradient_quotient} The horizontal lift of the gradient of $h$ at $\pi(Y)$ is the Riemannian gradient of $F$ at $Y$. That is, 
\[
\overline{\grad h(\pi(Y)) }_Y  = \grad F(Y).
\]
Therefore, $\grad F(Y)$ is automatically in $\mathcal{H}_Y$. 
\end{thm}
\begin{proof}
See \cite[Section 3.6.2]{absil_optimization_2008}. 
\end{proof}

The next proposition summarizes the expression of $\grad F(Y)$ under different metrics.
\begin{prop}\label{thm:gradF(Y)}
Let $f$ be a smooth real-valued function defined on $\mathcal{H}^{n,p}_+$ and let $F: \mathbb{C}^{n\times p}_* \rightarrow \mathbb{R}: Y \mapsto f(YY^*)$. Assume $YY^* = X$. Then the Riemannian gradient of $F$ is given by
\[
    \grad F(Y) =\left\{\begin{aligned} 
     2\nabla f(YY^*)Y, \quad \text{if using metric } g^1\\
     2\nabla f(YY^*)Y(Y^*Y)^{-1}, \quad  \text{if using metric } g^2 \\ 
     \left(I - \frac{1}{2}P_Y\right)\nabla f(YY^*) Y(Y^*Y)^{-1} \quad \text{if using metric } g^3\\ 
\end{aligned}\right.
\]
where $\nabla f$ denotes Fr\'{e}chet gradient \eqref{frechet-derivative} and $P_Y = Y(Y^*Y)^{-1}Y^*$.
\end{prop}

\begin{proof}
Let  $A \in T_Y\mathbb{C}^{n\times p}_* = \mathbb{C}^{n\times p}$. By chain rule, we have
\[
    \D F(Y)[A] = \D f(YY^*)[YA^* + AY^*].
\]
This yields to 
\[
    g^i_Y(\grad F(Y), A) = g_X\left({\grad f(YY^*), YA^*+AY^*}\right),
\]
where $g_X$ is the metric \eqref{metric-embedded-z}.
Since $YA^*+AY^* \in T_{YY^*}\mathcal{H}^{n,p}_+$, we have 
\begin{equation*}
    g_X\left(\grad f(YY^*), YA^*+AY^*\right) =  \ip{P^t_{YY^*}(\nabla f(YY^*)), YA^*+YA^*}_{\mathbb{C}^{n\times n}}=  \ip{\nabla f(YY^*), YA^*+AY^*}_{\mathbb{C}^{n\times n}}.  
\end{equation*}

It is straightforward to verify that
\begin{align*}
   \ip{\nabla f(YY^*), YA^*+AY^*}_{\mathbb{C}^{n\times n}} &= g^1_Y\left(2 \nabla f(YY^*)Y,A\right) \\
  & = g^2_Y\left(2\nabla f(YY^*)Y(Y^*Y)^{-1},A \right),
\end{align*} 
which yields the expression of $\grad F(Y)$ under $g^1$ and $g^2$. 

The Riemannian gradient for $g^3$ is due to
\begin{eqnarray*}
     \ip{P^t_{YY^*}(\nabla f(YY^*)), YA^*+YA^*}_{\mathbb{C}^{n\times n}} &=& g^3_Y\left(\left(I-\frac{1}{2}P_Y\right)P^t_X(\nabla f(YY^*))Y(Y^*Y)^{-1},A\right) \\
     &=& g^3_Y\left(\left(I - \frac{1}{2}P_Y\right)\nabla f(YY^*)Y(Y^*Y)^{-1}, A\right).
\end{eqnarray*}
\end{proof}

\subsection{Retraction}\label{sec:retraction_quotient}
The retraction on the quotient manifold $\mathbb{C}^{n\times p}_*/\mathcal{O}_p$ can be defined using the retraction on the total space $\mathbb{C}^{n\times p}_*$.  
For any $A \in T_Y \mathbb{C}^{n\times p}_*$ and a step size $\tau>0$,
\[
\overline{R}_Y(\tau A):= Y + \tau A, 
\]
is a retraction on $\mathbb{C}^{n\times p}_*$ if  $Y+\tau A$ remains full rank, which is ensured for small enough $\tau$. 
Then Lemma \ref{lem:horizontal_lift} indicates that $\overline{R}$ satisfies the conditions of \cite[Prop.~4.1.3]{absil_optimization_2008}, which implies that 
\begin{equation}\label{eqn:retraction_quotient}
    R_{\pi(Y)}(\tau \eta_{\pi(Y)}) := \pi(\overline{R}_Y(\tau  \overline{\eta}_Y)) = \pi(Y+\tau \overline{\eta}_Y)
\end{equation}
defines a retraction on the quotient manifold $\mathbb{C}^{n\times p}_*/\mathcal{O}_p$
for a small enough step size $\tau >0.$

\subsection{Vector transport}\label{sec:vector_transport_quotient}
A vector transport on $\mathbb{C}^{n\times p}_*/\mathcal{O}_p$ introduced in \cite[Section 8.1.4]{absil_optimization_2008}  is  projection to horizontal space.
\begin{equation}\label{eqn:vector_transport_quotient}
    \overline{\left(\mathcal{T}_{\eta_{\pi(Y)}} \xi_{\pi(Y)}\right)}_{ Y + \overline{\eta}_Y} := P^\mathcal{H}_{Y + \overline{\eta}_Y}(\overline{\xi}_Y).
\end{equation}
It can be shown that this vector transport is actually the differential of the retraction $R$ defined in (\ref{eqn:retraction_quotient}) (see \cite[Section 8.1.2]{absil_optimization_2008}) since 
\begin{eqnarray*}
    \D R_{\pi(Y)}(\eta_{\pi(Y)})[\xi_{\pi(Y)}] &=& \D\pi\left(\overline{R}_{Y}(\overline{\eta}_Y)\right)\left[ \D \overline{R}_Y(\overline{\eta}_Y)[\overline{\xi}_Y]\right] \\
    &=& \D\pi(Y+\overline{\eta}_Y) \left[ \dt{0}(Y + \overline{\eta}_Y + t\overline{\xi}_Y)\right] \\
    &=& \D\pi(Y+\overline{\eta}_Y) \left[\overline{\xi}_Y\right] \\
    &=& \D\pi(Y+\overline{\eta}_Y) \left[P^\mathcal{H}_{Y + \overline{\eta}_Y}(\overline{\xi}_Y)\right].
\end{eqnarray*}

Based on the projection formulae in Section \ref{section:quotient_projection}, we can obtain formulae of vector transports using different Riemannian metrics. Denote $Y_2 = Y_1 + \overline{\eta}_{Y_1}$. If we use metric $g^1$, then 
\[
    \overline{\left( \mathcal{T}_{\eta_{\pi(Y_1)}}\xi_{\pi(Y_1)}\right)}_{ Y_1 + \overline{\eta}_{Y_1}} = \overline{\xi}_{Y_1} - Y_2 \Omega,
\]
where $\Omega$ solves the Lyapunov equation 
\[Y_2^*Y_2\Omega+\Omega Y_2^*Y_2=Y_2^* \overline{\xi}_{Y_1} - \overline{\xi}_{Y_1}^* Y_2.\]

% See Remark \ref{rem-Lyapunov} for the expression of $\Omega$. 

If we use metric $g^2$ or $g^3$, then 
\begin{eqnarray*}
    \overline{\left( \mathcal{T}_{\eta_{\pi(Y_1)}}\xi_{\pi(Y_1)}\right)}_{ Y_1+ \overline{\eta}_{Y_1}} &=& \overline{\xi}_{Y_1} - P_{Y_2}^\mathcal{V}(\overline{\xi}_{Y_1})\\
    &=& \overline{\xi}_{Y_1} - Skew\left((Y_2^*Y_2)^{-1}Y_2^* \overline{\xi}_{Y_1}\right)\\
    &=& Y_2 \left(\frac{(Y_2^*Y_2)^{-1}Y_2^*\overline{\xi}_{Y_1} + \overline{\xi}_{Y_1}^* Y_2(Y_2^*Y_2)^{-1}}{2}\right) + {Y_2}_\perp {Y_2}_\perp^*\overline{\xi}_{Y_1}.
\end{eqnarray*}

\subsection{Riemannian Hessian operator}
Recall that the cost function $h$ on $\mathbb{C}^{n\times p}_*/\mathcal{O}_p$ is defined in (\ref{eqn:cost_function_quotient}). In this section, we summarize the Riemannian Hessian of $h$ under the three different metrics $g^i$. The proofs are tedious calculations and given in Appendix \ref{Appen:riemannian_hessian_quotient}. 

\begin{prop} Using $g^1$, the Riemannian Hession of $h$ is given by
\[
    \overline{\left( \Hess h(\pi(Y))[\xi_{\pi(Y)}]\right)}_{ Y} = P^{\mathcal{H}^1}_Y \left(2\nabla^2 f(YY^*)[Y\overline{\xi}_Y^*+\overline{\xi}_YY^*]Y + 2\nabla f(YY^*)\overline{\xi}_Y \right).
\]

\end{prop}

\begin{prop}
Using $g^2$, the Riemannian Hession of $h$ is given by
\begin{eqnarray*}
    \overline{\left( \Hess h(\pi(Y))[\xi_{\pi(Y)}]\right)}_{ Y} &=& P^{\mathcal{H}^2}_Y \left\{2\nabla^2f(YY^*)[Y\overline{\xi}_Y^*+\overline{\xi}_YY^*]Y(Y^*Y)^{-1} \right. \\
    &&+ \nabla f(YY^*)P_Y^\perp\overline{\xi}_Y(Y^*Y)^{-1}+P_Y^\perp \nabla f(YY^*)\overline{\xi}_Y(Y^*Y)^{-1}  \\
    &&  +2skew(\overline{\xi}_YY^*)\nabla f(YY^*)Y(Y^*Y)^{-2}\\
    &&+ \left. 2skew\{\overline{\xi}_Y(Y^*Y)^{-1}Y^*\nabla f(YY^*)\}Y(Y^*Y)^{-1} \right\}.
\end{eqnarray*}

\end{prop}

\begin{prop}
Using $g^3$, the Riemannian Hession of $h$ is given by
\begin{eqnarray*}
    \overline{\left( \Hess h(\pi(Y))[\xi_{\pi(Y)}]\right)}_{ Y}  &=& \left(I-\frac{1}{2}P_Y\right)\nabla^2f(YY^*)[Y\overline{\xi}_Y^*+\overline{\xi}_YY^*]Y(Y^*Y)^{-1}\\
    && + (I-P_Y)\nabla f(YY^*)(I-P_Y)\overline{\xi}_Y(Y^*Y)^{-1}.
\end{eqnarray*}

\end{prop}

\section{The Riemannian conjugate gradient method}\label{sec:RCG}
For simplicity, in this paper we only consider the Riemannian conjugate gradient (RCG) method described as Algorithm 1 in \cite{vandereycken_low-rank_2013} with the geometric variant of Polak–Ribi\'{e}re (PR+) for computing the conjugate direction. It is possible to explore other methods such as the limited-memory version of the Riemannian BFGS method (LRBFGS) as in \cite{huang_broyden_2015}.  
However, RCG performs very well on a wide variety of problems and is easier to implement for our numerical examples.

In this section, we focus on establishing two equivalences in algorithms.
First, we show that the Burer--Monteiro CG method, which is simply applying CG method for the unconstrainted problem \eqref{min-BR}, is equivalent to RCG on the Riemannian quotient  manifold $(\mathbb{C}^{n\times p}_*/\mathcal{O}_p, g^1)$ with  our retraction and vector transport defined in the previous sections. 
Second, we  show that RCG on the embedded manifold $\mathcal{H}^{n,p}_+$ is equivalent to RCG on the quotient manifold $(\mathbb{C}^{n\times p}_*/\mathcal{O}_p, g^3)$ with a specific retraction and vector transport. 
 
% For convenience, let $\mathcal{T}_{X_{k-1}\rightarrow X_k}$  denote a vector transport that maps tangent vectors from $T_{X_{k-1}}\mathcal{H}^{n,p}_+$ to $T_{X_k}\mathcal{H}^{n,p}_+$, defined as
% \[
%     \mathcal{T}_{{X_{k-1}} \rightarrow X_k }: T_{X_{k-1}}\mathcal{H}^{n,p}_+ \rightarrow T_{X_k}\mathcal{H}^{n,p}_+,\quad \zeta_{X_{k-1}} \mapsto \mathcal{T}_{R_{X_{k-1}}^{-1}(X_k)}(\zeta_{X_{k-1}}),
% \]
% where $R_{X}^{-1}$ exists locally for every $X\in\mathcal{H}^{n,p}_+$ by the inverse function theorem. Hence $\mathcal{T}_{{X_{k-1}} \rightarrow X_k }$ should be understood locally in the sense that $X_{k-1}$ is sufficiently close to $X_k$. See  \cite[Section 2.4]{vandereycken_low-rank_2012}.

% Similarly, Let $\mathcal{T}_{Y_{k-1}\rightarrow Y_k}$  denote a vector transport that maps tangent vectors from $\mathcal{H}_{Y_{k-1}}$ to $\mathcal{H}_{Y_k}$ as
% \[
%     \mathcal{T}_{Y_{k-1}\rightarrow Y_k}: \mathcal{H}_{Y_{k-1}} \rightarrow \mathcal{H}_{Y_k}, \quad \overline{\xi}_{Y_{k-1}} \mapsto \overline{\left(\mathcal{T}_{R^{-1}_{\pi(Y_{k-1})} \xi_{\pi(Y_k)}} \right)}_{Y_k},
% \]
% where $R^{-1}_{\pi(Y)}$ also exists locally for every $\pi(Y) \in \mathbb{C}^{n\times p}_*/\mathcal{O}_p$. $\mathcal{T}_{{Y_{k-1}} \rightarrow Y_k }$ and should again be understood locally in the sense that $\pi(Y_{k-1})$ is sufficiently close to $\pi(Y_k)$. 

We first summarize two Riemannian CG algorithms in Algorithm  \ref{alg:CG_embedded}  and Algorithm \ref{alg:CG_quotient}  below. Algorithm  \ref{alg:CG_embedded} is the RCG on the embedded manifold for solving \eqref{lowrank_prob} and Algorithm \ref{alg:CG_quotient} is the RCG on the quotient manifold $(\mathbb{C}^{n\times p}_*/\mathcal{O}_p, g^i)$ for solving \eqref{quotient_prob}. We remark that the explicit constants $0.0001$ and $0.5$ in the Armijo backtracking are chosen for convenience.
\begin{algorithm}[htbp]
\caption{Riemannian Conjugate Gradient on the embedded manifold $\mathcal{H}^{n,p}_+$}
\label{alg:CG_embedded}
\begin{algorithmic}[1]
\Require initial iterate $X_0 \in \mathcal{H}^{n,p}_+$, initial gradient $\xi_0 = \grad f(X_0)$, initial conjugate direction $\eta_0 = -\grad f(X_0)$, tolerance $\varepsilon>0$
\For{ $k =1,2,\dots$}
    \State Compute an initial step $t_k$. For special cost functions, it is possible to compute: 
    \par\hskip\algorithmicindent $t_k = \argmin_t{f(X_{k-1} + t \eta_{k-1})}$

    \State Perform Armijo backtracking to find the smallest integer  $m \geq 0$ such that 
    \[
        f(X_{k-1}) - f(R_{X_{k-1}}(0.5^mt_k\eta_{k-1})) \geq -0.0001\times 0.5^m t_k g_{X_{k-1}}(\xi_{k-1},\eta_{k-1})
    \]
    \[
    \zeta_k : = 0.5^m t_k \eta_{k-1}
    \]
    \State Obtain the new iterate by retraction
    \par\hskip\algorithmicindent $X_{k} = {R}_{X_{k-1}}(\zeta_k)$ \Comment{ See Algorithm \ref{alg:retraction_embedded}}
    
    \State{Compute gradient}
    \par\hskip\algorithmicindent $\xi_k := \grad f(X_k)$ 
    \Comment{ See Algorithm \ref{alg:grad_embedded}} 
    \State{Check convergence}
     \par\hskip\algorithmicindent if $\norm{\xi_k}:=\sqrt{g_{X_k}(\xi_k, \xi_k)} < \varepsilon $ or $f(X_k) < \varepsilon$, then break 
    \State Compute a conjugate direction by $\mbox{PR}_+$ and vector transport
    \par\hskip\algorithmicindent $\eta_k = -\xi_k + \beta_k \mathcal{T}_{\zeta_k}(\eta_{k-1}),$   \Comment{ See Algorithm \ref{alg:vector_transport_embedded},  \ref{alg:vector_transport_embedded_2}
    }
    
    \[\text{with }    \beta_k := \frac{g_{X_k}\left(\xi_k, \xi_k-\mathcal{T}_{\zeta_k}(\xi_{k-1})\right)}{g_{X_{k-1}}\left( \xi_{k-1}, \xi_{k-1} \right)}. \]
\EndFor
\end{algorithmic}
\end{algorithm}

\begin{algorithm}[htbp]
\caption{Riemannian Conjugate Gradient on the quotient manifold $\mathbb{C}^{n\times p}_*/\mathcal{O}_p$ with metric $g^i$
}
\label{alg:CG_quotient}
\begin{algorithmic}[1]
\Require initial iterate $Y_0 \in \pi^{-1}(\pi(Y_0))$, initial horizontal lift of gradient $\overline{\xi}_0 = \grad F(Y_0)$, initial conjugate direction $\overline{\eta}_0 = - \overline{\xi}_0$, tolerance $\varepsilon>0$
\For{ $k =1,2,\dots$}
    \State Compute an initial step $t_k$. For special cost functions, it is possible to compute:
    \par\hskip\algorithmicindent $t_k = \argmin_t{F(Y_{k-1} + t \overline{\eta}_{k-1})}$
    \State Perform Armijo backtracking to find the smallest integer  $m \geq 0$ such that 
    \begin{equation*}
        F(Y_{k-1}) - F(\overline{R}_{Y_{k-1}}(0.5^mt_k\overline{\eta}_{k-1})) \geq -0.0001\times 0.5^m t_k g^i_{Y_{k-1}}(\overline{\xi}_{k-1},\overline{\eta}_{k-1})
    \end{equation*}
    \[
    \overline{\zeta}_k := 0.5^m t_k \overline{\eta}_k
    \]
    \State Obtain the new iterate by retraction
    \par\hskip\algorithmicindent $Y_{k} = \overline{R}_{Y_{k-1}}(\overline{\zeta}_k)$
    \State{Compute the horizontal lift of gradient}
    \par\hskip\algorithmicindent $\overline{\xi}_k := \overline{(\grad h(\pi(Y_k)))}_{ Y_k} = \grad F(Y_k)$   \Comment{See Algorithm \ref{alg:grad_quotient}}
    \State{Check convergence}
     \par\hskip\algorithmicindent if $\norm{\overline{\xi}_k}:=\sqrt{g_{Y_k}^i(\overline{\xi}_k, \overline{\xi}_k)} < \varepsilon $ or $F(Y_k) < \varepsilon$, then break 
    \State Compute a conjugate direction by $\mbox{PR}_+$ and vector transport
    \par\hskip\algorithmicindent $\overline{\eta}_k = -\overline{\xi}_k + \beta_k \overline{(\mathcal{T}_{\zeta_k}\eta_{k-1})}_{Y_k},$ \Comment{See Algorithm \ref{alg:vector_transport_quotient1}}
    \[\text{with }    \beta_k := \frac{g_{Y_k}^i\left(\grad F(Y_k), \grad F(Y_k)-\overline{(\mathcal{T}_{\zeta_k}\xi_{k-1})}_{Y_k}\right)}{g^i_{Y_{k-1}}\left( \grad F(Y_{k-1}), \grad F(Y_{k-1}) \right)}. \]
\EndFor
\end{algorithmic}
\end{algorithm}

\subsection{Equivalence between Burer--Monteiro CG and RCG on the Riemannian quotient manifold with the Bures-Wasserstein metric $(\mathbb{C}^{n\times p}_*/\mathcal{O}_p,g^1)$}\label{sec:BMCG_equal_RCG}

\begin{thm}
 Using retraction (\ref{eqn:retraction_quotient}), vector transport (\ref{eqn:vector_transport_quotient}) and metric $g^1$, Algorithm \ref{alg:CG_quotient}  is equivalent to the conjugate gradient method solving \eqref{min-BR} in the sense that they produce exactly the same iterates if started from the same initial point. 
\end{thm}
\begin{proof}
First of all,
for $g^1$, the Riemannian gradient of $F$ at $Y$ is $\grad F(Y) = 2 \nabla f(YY^*)Y$, which is equal to the Fr\'{e}chet gradient of $F(Y)=f(YY^*)$ at $Y$. Since vector transport is the orthogonal projection to the horizontal space, the $\mbox{PR}_+$ $\beta_k$ used in Riemannian CG becomes 
 \begin{equation}\label{eqn:PR_quotient}
    \beta_k = \frac{g_{Y_k}^1\left(\grad F(Y_k), \grad F(Y_k)-{P}^{\mathcal{H}^1}_{Y_k}(\grad F(Y_{k-1}))\right)}{g^1_{Y_{k-1}}\left( \grad F(Y_{k-1}), \grad F(Y_{k-1}) \right)}.
 \end{equation}
Now observe that 
\[
    P^{\mathcal{H}^1}_{Y_k} (\grad F(Y_{k-1})) = \grad F(Y_{k-1}) -  P^{\mathcal{V}}_{Y_k} (\grad F(Y_{k-1}))
\]
 and $g^1$ is equivalent to the classical inner product for $\mathbb{C}^{n\times p}$. Hence $\beta_k$ computed by (\ref{eqn:PR_quotient}) is equal to $\mbox{PR}_+$ $\beta_k$ in conjugate gradient for \eqref{min-BR}. 
 
 The first conjugate direction is $\eta_1 = - \grad F(Y_1) = - \nabla F(Y_1)$, so  Burer--Monteiro CG coincides with Riemannian CG for the first iteration. It remains to show that $\eta_k$ generated in Riemannian CG by 
 \begin{equation*}
     \eta_k = - \xi_k + \beta_k P^{\mathcal{H}^1}_{Y_k}(\eta_{k-1})
 \end{equation*}
is equal to $\eta_k$ generated in Burer--Monteiro CG for each $k\geq 2$. It suffices to show that
\begin{equation*}
    P^{\mathcal{H}^1}_{Y_k}(\eta_{k-1}) = \eta_{k-1}, \quad \forall k \geq 2.
\end{equation*}
Equivalently we need to show that for all $k\geq 2$, the Lyapunov equation
\begin{equation}\label{eqn:Lyapunov}
    (Y_k^*Y_k)\Omega + \Omega (Y_k^*Y_k) = Y_k^*\eta_{k-1} - \eta_{k-1}^* Y_k
\end{equation}
only has trivial solution $\Omega = 0$. By invertibility of the equation, this means that we only need to show the right hand side is zero. We prove it by induction. 

For $k=2$, $\eta_{k-1} = \eta_1 = - \xi_1 = - \grad F(Y_1)$. The following computations show that the RHS of (\ref{eqn:Lyapunov}) satisfies
\begin{eqnarray*}
    Y_2^*\eta_1 - \eta_1^*Y_2 &=& -Y_2^*\xi_1 + \xi_1^* Y_2 \\
    &=& -(Y_1 - c\xi_1)^* \xi_1 +\xi_1^*(Y_1-c \xi_1)\\
    &=& \xi_1^*Y_1 - Y_1^*\xi_1 \\
    &=& Y_1^*(2\nabla f(Y_1Y_1^*))Y_1 - Y_1^*(2 \nabla f(Y_1Y_1^*))Y_1 \\
    &=& 0. 
\end{eqnarray*}
Hence $\Omega = 0$ and $P^{\mathcal{H}^1}_{Y_k}(\eta_{k-1}) = \eta_{k-1}$ for $k=2$. 

Now suppose for $k\geq 2$, the RHS of (\ref{eqn:Lyapunov}) is 0 and hence $P^{\mathcal{H}^1}_{Y_k}(\eta_{k-1}) = \eta_{k-1}$ holds. Then the RHS of the Lyapunov equation of step $k+1$ is
\begin{eqnarray*}
    Y_{k+1}^*\eta_k - \eta_k^*Y_{k+1} &=&  (Y_{k}+c\eta_k)^*\eta_k - \eta_k^*(Y_{k} +c \eta_k) \\
    &=& Y_k^*\eta_k - \eta_k^* Y_k \\
    &=& Y_k^*\left(-\xi_k + \beta_k P^{\mathcal{H}^1}_{Y_k}(\eta_{k-1}) \right) - \left(-\xi_k+\beta_k P^{\mathcal{H}^1}_{Y_k}(\eta_{k-1}) \right)^*Y_k \\
    &=& Y_k^*(-\xi_k + \beta_k \eta_{k-1}) - (-\xi_k+\beta_k \eta_{k-1})^*Y_k \\
    &=& -Y_k^*\xi_k + \xi_k^* Y_k\\
    &=& -Y_k^*(2\nabla f(Y_kY_k^*))Y_k + Y_k^*(2\nabla f(Y_kY_k^*))Y_k\\
    &=& 0.
\end{eqnarray*}
Hence $P^{\mathcal{H}^1}_{Y_{k+1}}(\eta_{k}) = \eta_{k}$ also holds. We have thus proven that Riemannian CG is equivalent to Burer--Monteiro CG.
\end{proof}

Since the gradient descent corresponds to $\beta_k\equiv 0$, the same discussion also implies the following

\begin{coro}
 Using retraction (\ref{eqn:retraction_quotient}) and metric $g^1$, the Riemannian gradient descent on the quotient manifold is equivalent to the Burer--Monteiro gradient descent method with suitable step size \eqref{br-gd} in the sense that they produce exactly the same iterates. 
\end{coro}

\subsection{Equivalence between RCG on embedded manifold and RCG on the quotient manifold $(\mathbb{C}^{n\times p}_*/\mathcal{O}_p,g^3)$}
\label{sec:embedded_equal_quotient}
In this subsection we show that Algorithm \ref{alg:CG_embedded} is equivalent to Algorithm \ref{alg:CG_quotient} with  Riemannian metric $g^3$, a specific initial line-search in step 5, a specific retraction and a specific vector transport. The idea is to take the advantage of the diffeomorphism $\tilde{\beta}$ between $\mathbb{C}^{n\times p}_*/\mathcal{O}_p$ and $\mathcal{H}^{n,p}_+$,
as well as the fact that the metric $g^3$ of $\mathbb{C}^{n\times p}_*/\mathcal{O}_p$ is induced from the metric of $\mathcal{H}^{n,p}_+$. 

The Lemma below shows that there is a one-to-one correspondence between $\grad f$ and $\grad h$. 
\begin{lem} 
If we use $g^3$ as the Riemannian metric for $\mathbb{C}^{n\times p}_*/\mathcal{O}_p$, then the Riemannian gradient of $f$ and $h$ is related by the diffeomorphism $\tilde{\beta}$ in the following way:
\[
(\D \tilde{\beta}) ( \pi(Y) )[\grad h(\pi(Y))]  = \grad f(YY^*).
\]
\end{lem}
\begin{proof}
Recall that $\beta = \tilde{\beta}\circ \pi$ and we have Theorem \ref{thm:lift_gradient_quotient}. By chain rule and the definition of horizontal lift we have
\begin{eqnarray*}
 LHS = (\D \tilde{\beta}) (\pi(Y) )[\grad h(\pi(Y))]  &=&  (\D\tilde{\beta})(\pi(Y))\left[\D\pi(Y)\left[\overline{\grad h(\pi(Y))}_{Y}\right]\right] \\
 &=& \D \beta(Y)\left[\overline{\grad h(\pi(Y))}_{Y} \right] \quad \text{(Inverse direction of chain rule)}\\
 &=& \D \beta(Y)\left[ \grad F(Y) \right].
\end{eqnarray*}

Now recall that $F = f \circ \beta$. Let $A \in \mathbb{C}^{n\times p}$ then 
\begin{equation*}
    \D F(Y)[A] = \D f(YY^*)[YA^*+YA^*].
\end{equation*}
Let $X = YY^*$. Then we have
\begin{equation*}
    g^3_Y(\grad F(Y), A) = g_X(\grad f(YY^*), YA^*+AY^*).
\end{equation*}
Apply the definition of $g^3$, we  have
\begin{equation*}
    g_X\left(\D \beta(Y)[\grad F(Y)], YA^*+AY^*\right) = g_X\left(\grad f(YY^*), YA^*+AY^*\right),
\end{equation*}
or 
\begin{equation*}
    g_X\left(LHS, YA^*+AY^*\right) = g_X\left( RHS, YA^*+AY^*\right).
\end{equation*}
Now notice that $A$ is arbitrary and $YA^*+AY^*$ can be any tangent vector in $T_X\mathcal{H}^{n,p}_+$. Hence we must have $LHS = RHS$
\end{proof}

Since $\tilde{\beta}$ is a diffeomorphism bewteen $\mathbb{C}^{n\times p}_*/\mathcal{O}_p$ and $\mathcal{H}^{n,p}_+$, $D\tilde{\beta}(\pi(Y))[\cdot]$ defines an isomorphism between the tangent space $T_{\pi(Y)}\mathbb{C}^{n\times p}_*/\mathcal{O}_p$ and $T_{YY^*} \mathcal{H}^{n,p}_+$. We denote this isomorphism by $L_{\pi(Y)}$. When the tangent space is clear from the context, $\pi(Y)$ is omitted and we only use the notation $L$ for simplicity. The previous lemma then simply reads as
\[
    L_{\pi(Y)}(\grad h(\pi(Y))) = \grad f( \tilde{\beta}(\pi(Y))). 
\]

In Algorithm \ref{alg:CG_embedded}, we have a retraction $R^{E}$ and a vector transport $\mathcal{T}^E$ on the embedded manifold $\mathcal{H}^{n, p}_+$, (with the superscript $E$ for {\it Embedded}), such that $R^E$ is the retraction associated with $\mathcal{T}^E$. Then we claim that there is a retraction $R^Q$ and a vector transport $\mathcal{T}^Q$, (with the superscript $Q$ denoting {\it Quotient}), on the Riemannian quotient manifold $(\mathbb{C}^{n\times p}_*/\mathcal{O}_p,g^3)$, such that Algorithm \ref{alg:CG_quotient} is equivalent to Algorithm \ref{alg:CG_embedded}. The idea is again to use the diffeomorphism $\tilde{\beta}$ and the isomorphism $L_{\pi(Y)}$. We give the desired expression of $R^Q$ and $\mathcal{T}^Q$ as follows.

\begin{equation}\label{eqn:retraction_q}
    R^Q_{\pi(Y)}(\xi_{\pi(Y)}) := \tilde{\beta}^{-1}\left(R^E_{\tilde{\beta}(\pi(Y))}\left(L(\xi_{\pi(Y)})\right) \right). 
\end{equation}
\begin{equation}\label{eqn:vector_transport_q}
    \mathcal{T}^Q_{\eta_{\pi(Y)}}(\xi_{\pi(Y)}) := L_{\pi(Y_2)}^{-1}\left(\mathcal{T}^E_{L(\eta_{\pi(Y)})}\left(L(\xi_{\pi(Y)})\right) \right),
\end{equation}
where $\pi(Y_2)$ is in $\mathbb{C}^{n\times p}_*/\mathcal{O}_p$ such that  $\tilde{\beta}(\pi(Y_2))$ denotes the foot of the tangent vector $\mathcal{T}^E_{L(\eta_{\pi(Y)})}\left(L(\xi_{\pi(Y)})\right)$.

Now it remains to show that 
$R^Q$ defined in (\ref{eqn:retraction_q}) is indeed a retraction and $\mathcal{T}^Q$ defined in (\ref{eqn:vector_transport_q}) is indeed a vector transport.

\begin{lem}
$R^Q$ defined in (\ref{eqn:retraction_q}) is a retraction. 
\end{lem}
\begin{proof}
First it is easy to see that $R^Q_{\pi(Y)}(0_{\pi(Y)}) = \pi(Y)$. 
Then we also have for all $v_{\pi(Y)} \in T_{\pi(Y)}\mathbb{C}^{n\times p}_*/\mathcal{O}_p$
\begin{eqnarray*}
    \D R^Q_{\pi(Y)}(0_{\pi(Y)})[v_{\pi(Y)}] &=& (\D \tilde{\beta}^{-1}) (\tilde{\beta}(\pi(Y))\left[\D R^E_{\tilde{\beta}(\pi(Y))}(0)\left[\D L(0)\left[v_{\pi(Y)}\right]\right] \right] \\ 
    &=& (\D\tilde{\beta}^{-1}) (\tilde{\beta}(\pi(Y))\left[\D R^E_{\tilde{\beta}(\pi(Y))}(0)\left[L(v_{\pi(Y)})\right] \right] \\ 
    &=& (\D\tilde{\beta}^{-1})(\tilde{\beta}(\pi(Y))\left[ L(v_{\pi(Y)}) \right] \\
    &=& \left( \D\tilde{\beta}(\pi(Y))\right)^{-1} [L(v_{\pi(Y)})] \\
    &=& L^{-1}(L(v_{\pi(Y)}))\\
    &=& v_{\pi(Y)}
\end{eqnarray*}
Hence $\D R^Q_{\pi(Y)}(0_{\pi(Y)})[\cdot]$ is an identity map. 
\end{proof}

\begin{lem}
$\mathcal{T}^E$ defined in (\ref{eqn:vector_transport_q}) is a vector transport and $R^Q$ is the retraction associated with $\mathcal{T}^E$.
\end{lem}
\begin{proof}
Consistency and linearity are straightforward. It thus suffices to verify that the foot of $\mathcal{T}^Q_{\eta_{\pi(Y)}}(\xi_{\pi(Y)})$ is equal to $R^Q_{\pi(Y)}(\eta_{\pi(Y)})$. Since $R^E$ is the associated retraction with $\mathcal{T}^E$, the foot of $\mathcal{T}^E_{L(\eta_{\pi(Y)})}(L(\xi_{\pi(Y)}))$ is equal to $R^E_{\tilde{\beta}(\pi(Y))}\left(L(\eta_{\pi(Y)})\right)$, which we denote by $\tilde{\beta}(\pi(Y_2))$ for some $\pi(Y_2)$.  Hence $R^Q_{\pi(Y)}(\eta_{\pi(Y)}) = \tilde{\beta}^{-1}\left(R^E_{\tilde{\beta}(\pi(Y))}\left(L(\eta_{\pi(Y)})\right)\right) = \pi(Y_2)$.

Furthermore, we have that $\mathcal{T}^Q_{\eta_{\pi(Y)}}(\xi_{\pi(Y)}) = L^{-1}_{\pi(Y_2)}\left(\mathcal{T}^E_{L(\eta_{\pi(Y)})}\left(L(\xi_{\pi(Y)})\right) \right)$ is a tangent vector in $T_{\pi(Y_2)}\mathbb{C}^{n\times p}_*/\mathcal{O}_p$. Hence, the foot of $\mathcal{T}^Q_{\eta_{\pi(Y)}}(\xi_{\pi(Y)})$ is also $\pi(Y_2)$.
\end{proof}

Finally,  in order to reach an equivalence, we also need the initial step size to match the one in step 5 of Algorithm \ref{alg:CG_quotient}. 
We simply replace the original initial step size $t_k$ by 
\[
     t_k = \argmin_t{f(Y_kY_k^* + t (Y_k\eta_k^*+\eta_k Y_k^*))}. 
\]
This value of $t_k$ now is equivalent to the initial step size in step 5 of Algorithm \ref{alg:CG_embedded}. This gives us the following result:

\begin{thm}
With the newly constructed initial step size, retraction and vector transport in this subsection, Algorithm \ref{alg:CG_quotient} for solving \eqref{quotient_prob} is equivalent to Algorithm \ref{alg:CG_embedded} solving \eqref{lowrank_prob} in the sense that they produce exactly the same iterates. 
\end{thm}

\section{Implementation details}
\label{sec:details}
The algorithms in this paper can be applied for minimizing 
any smooth function $f(X)$ in \eqref{lowrank_prob}. For problems with large $n$, however, it is advisable to avoid constructing and storing the Fr\'{e}chet derivative $\nabla f(X)\in \mathbb C^{n\times n}$ explicitly. Instead, one directly computes the matrix-vector multiplications $\nabla f(X) U$. In the PhaseLift problem \cite{candes2013phaselift}, for example,  these matrix-vector multiplications can be implemented via the FFT at a cost of $\mathcal O(p n\log n )$ when $U\in \mathbb C^{n\times p}$; see \cite{huang_solving_2017}.

Below, we detail the calculations needed in Algorithms~\ref{alg:CG_embedded} and~\ref{alg:CG_quotient}. When giving flop counts, we assume that $\nabla f(X) U \in \mathbb{C}^{n\times p}$ can be computed  in $spn\log n$ flops with $s$ small. For $g^2$ and $g^3$ in Algorithms \ref{alg:grad_quotient} and \ref{alg:vector_transport_quotient1}, we use forwardslash "/" and backslash "\textbackslash" in Matlab command to compute the inverse of $Y^*Y$.

\subsection{Embedded manifold}
\begin{algorithm}[H]
\caption{Calculate the Riemannian gradient $\grad f(X)$}\label{alg:grad_embedded}
\begin{algorithmic}[1]
\Require $X = U\Sigma U^* \in \mathcal{H}^{n,p}_+$
\Ensure $\grad f(X) = UHU^* + U_p U^* + U U_p^* \in T_X\mathcal{H}^{n,p}_+$
\par $T \leftarrow \nabla f(X)U$ \Comment{\# $spn\log n$ flops}
\par $H \leftarrow U^*T$ \Comment{\# $p^2(2n-1)$ flops}
\par $U_p \leftarrow T - UH$ \Comment{\# $np+np(2p-1)$ flops}
\end{algorithmic}
\end{algorithm}

\begin{algorithm}[h]
\caption{Calculate the vector transport by projection to tangent space $P^t_{X_2}(\nu)$} %for the embedded manifold
\label{alg:vector_transport_embedded}
\begin{algorithmic}[1]
\Require $X_1 = U_1\Sigma_1U_1^*$, $X_2 = U_2\Sigma_2U_2^*$  and tangent vector $\nu = U_1H_1U_1^* + {U_p}_1 U_1^* + U_1 {U_p}_1^* \in T_{X_1} \mathcal{H}^{n,p}_+$. 
\Ensure $P^t_{X_2}(\nu) = U_2 H_2 U_2^* + {U_p}_2 U_2^* + U_2 {U_p}_2^*$ 
\par $A \leftarrow U_1^*U_2$ \Comment{\# $p^2(2n-1)$ flops}
\par $H_2^{(1)} \leftarrow A^*H_1 A$,\quad $U_p^{(1)}\leftarrow U_1(H_1A)$ \Comment{\# $3p^2(2p-1)+ np(2p-1)$ flops}
\par $H_2^{(2)} \leftarrow U_2^* {U_p}_1A$,\quad $U_p^{(2)}\leftarrow {U_p}_1A$ \Comment{\# $p^2(2n-1) + 2np(2p-1)$ flops}
\par $H_2^{(3)} \leftarrow {H_2^{(2)}}^*$, \quad $U_p^{(3)} \leftarrow U_1({U_1}_p^*U_2)$ \Comment{\# $np(2p-1) + p^2(2n-1)$ flops}
\par $H_2 \leftarrow H_2^{(1)} + H_2^{(2)} + H_2^{(3)}$ \Comment{\# $2p^2$ flops}
\par ${U_p}_2 \leftarrow U_p^{(1)} + U_p^{(2)} + U_p^{(3)}$, \quad ${U_p}_2 \leftarrow {U_p}_2 - U_2(U_2^* {U_p}_2)$  \Comment{\# $3np + np(2p-1)+p^2(2n-1)$ flops} 
\end{algorithmic}
\end{algorithm}

In implementation, we observe a vector transport that has better numerical performance if we only keep the first term in the above sum of $H_2$ and the second term of ${U_2}_p$ in Algorithm \ref{alg:vector_transport_embedded}, which is outlined in Algorithm \ref{alg:vector_transport_embedded_2}. 

\begin{algorithm}[h]
\caption{Calculate the simpler form of vector transport used in implementation that has a better performance $P^t_{X_2}(\nu)$} %for the embedded manifold
\label{alg:vector_transport_embedded_2}
\begin{algorithmic}[1]
\Require $X_1 = U_1\Sigma_1U_1^*$, $X_2 = U_2\Sigma_2U_2^*$  and tangent vector $\nu = U_1H_1U_1^* + {U_p}_1 U_1^* + U_1 {U_p}_1^* \in T_{X_1} \mathcal{H}^{n,p}_+$. 
\Ensure $P^t_{X_2}(\nu) = U_2 H_2 U_2^* + {U_p}_2 U_2^* + U_2 {U_p}_2^*$ 
\par $A \leftarrow U_1^*U_2$ \Comment{\# $p^2(2n-1)$ flops}
\par $H_2 \leftarrow A^*H_1 A$ \Comment{\# $2p^2(2p-1)$ flops}
\par $U_p\leftarrow {U_p}_1A$ \Comment{\# $np(2p-1)$ flops}
\par ${U_p}_2 \leftarrow {U_p} - U_2(U_2^* {U_p})$  \Comment{\# $np+p^2(2n-1) + np(2p-1)$ flops} 
\end{algorithmic}
\end{algorithm}

\begin{algorithm}[H]
\caption{Calculate the retraction $R_X(Z) = P_{\mathcal{H}^{n,p}_+}(X+Z)$}\label{alg:retraction_embedded}
\begin{algorithmic}[1]
\Require $X = U\Sigma U^* \in \mathcal{H}^{n,p}_+$, tangent vector $Z = UHU^* + U_pU^*+UU_p^*$. 
\Ensure $P_{\mathcal{H}^{n,p}_+}(X+Z) = U_+\Sigma_+U_+^*$. 
\par $(Q,R)\leftarrow \text{qr}(U_p,0)$ \quad $M \leftarrow \bmat{\Sigma + H & R^* \\ R & 0 }$ \Comment{\# $20np^2$ flops}
\par $[V,S] \leftarrow \text{eig}(M)$ \Comment{$O(p^3)$ flops}
\par $\Sigma+ \leftarrow S(1:p,1:p)$, \quad $U_+ \leftarrow\bmat{U&Q} V(:,1:p)$ \Comment{\# $np(4p-1)$ flops} 
\end{algorithmic}
\end{algorithm}

\subsection{Quotient manifold}
\begin{algorithm}[H]
\caption{Calculate the Riemannian gradient $\grad F(Y)$}\label{alg:grad_quotient}
\begin{algorithmic}[1]
\Require $Y\in \mathbb{C}^{n\times p}_*$
\Ensure $T = \grad F(Y) $ 
\If {metric is $g^1$}
\par $T \leftarrow 2\nabla f(YY^*)Y$. \Comment{\# $2spn\log n$ flops}
\ElsIf {metric is $g^2$}
\par $Z \leftarrow Y(Y^*Y)^{-1}$  \Comment{\# $np(2p-1)+p^2(2n-1) + O(p^3)$ flops}
\par $T\leftarrow 2\nabla f(YY^*)Z$ \Comment{\# $2spn\log n$ flops}
\ElsIf {metric is $g^3$}
\par $Z \leftarrow Y(Y^*Y)^{-1}$  \Comment{\# $np(2p-1)+p^2(2n-1) + O(p^3)$ flops}
\par $T \leftarrow 2 \nabla f(YY^*) Z$ \Comment{\# $2spn\log n$ flops}
\par $M  \leftarrow Y^*T$, \quad $T \leftarrow T - \tfrac{1}{2} ZM$ \Comment{\# $p^2(2n-1) + np + 2np^2$ flops}
\EndIf 
\end{algorithmic}
\end{algorithm}

\begin{algorithm}[H]
\caption{Calculate the quotient vector transport  $P^\mathcal{H}_{Y_2}(h_1)$}\label{alg:vector_transport_quotient1}
\begin{algorithmic}[1]
\Require $Y_1 \in \mathbb{C}^{n\times p}_*$, $Y_2 \in \mathbb{C}^{n\times p}_*$ and horizontal vector $h_1 \in \mathcal{H}_{Y_1}$.
\Ensure $h_2 = P^\mathcal{H}_{Y_2}(h_1) \in \mathcal{H}_{Y_2}$. 
\If {metric is $g^1$}
\par $E \leftarrow Y_2^*Y_2$ \Comment{\# $p^2(2n-1)$ flops}
\par $(Q,S) \leftarrow \text{eig}(E)$, \quad $d \leftarrow \diag(S)$ \Comment{\# $O(p^3)$ flops}
\par $\lambda \leftarrow d  \bmat{1,1,\cdots,1} + \bmat{1,1,\cdots,1}^T d^T$ \Comment{\# $ 2p^2$ flops}
\par $A \leftarrow Q^*(Y_2^*h_1 - h_1^*Y_2)Q$ \Comment{\# $p^2(2n-1)+np+2p^2(2p-1)$ flops}
\par $\Omega \leftarrow Q(A./\lambda )Q^*$ \Comment{\# $p^2 + 2p^2(2p-1)$ flops}
\par $h_2 \leftarrow h_1 - Y_2 \Omega $ \Comment{\# $np+np(2p-1)$ flops}
\ElsIf{metric is $g^2$ or $g^3$}
\par $\tilde \Omega \leftarrow (Y^*Y)^{-1}(Y_2^*h_1)$   \Comment{\# $2p^2(2p-1)+p^2(2n-1) + O(p^3)$ flops}
\par $\Omega \leftarrow \tfrac{1}{2} (\tilde \Omega - \tilde \Omega^*)$ \Comment{\#  $2p^2$ flops}
\par $h_2 \leftarrow h_1 - Y_2 \Omega$ \Comment{\# $np+np(2p-1)$ flops}
\EndIf 
\end{algorithmic}
\end{algorithm}

\subsection{Initial guess for the line search}

The initial guess for the line search generally depends on the expression of the cost function $f(X)$. For the important case of $f(X) = \frac{1}{2}\norm{\mathcal{A}(X) - b}_F^2$ where  $\mathcal{A}$ is a linear operator and $b$ is a matrix, the initial guess for embedded CG requires solving a linear equation and for quotient CG it requires solving a cubic equation. Below this calculation is detailed for $b$ of size $m n$ for some $m$ and assuming that $\mathcal{A}(X), \mathcal{A}(T)$ and $\mathcal{A}(Y\eta^*)$ can be evaluated in $sp^\alpha n\log n $ flops for $X\in  \mathcal{H}^{n,p}_+$, 
 $T \in T_X\mathcal{H}^{n,p}_+$ and
$Y, \eta\in  \mathbb{C}^{n\times p}_*$. 

\begin{algorithm}[H]
\caption{Calculate the initial guess $t_* = \argmin_t f(X+tT)$}\label{alg:initial_guess_embedded}
\begin{algorithmic}[1]
\Require $X\in \mathcal{H}^{n,p}_+$ and a descend direction $T \in T_X\mathcal{H}^{n,p}_+$
\Ensure $t_*= \argmin_t f(X+tT) = \argmin_t \frac{1}{2}\norm{\mathcal{A}(X+tT)-b}_F^2$
\par $R \leftarrow \mathcal{A}(X) - b$ \Comment{\# $sp^\alpha n\log n + mn$ flops}
\par $S \leftarrow \mathcal{A}(T)$ \Comment{\# $sp^\alpha n\log n$ flops}
\par $t_* \leftarrow - \frac{\ip{R,S}}{\ip{S,S}}$ \Comment{\# $4mn-1$ flops}
\end{algorithmic}
\end{algorithm}

\begin{algorithm}[H]
\caption{Calculate the initial guess $t_* = \argmin_t F(Y+t\eta)$}\label{alg:initial_guess_quotient}
\begin{algorithmic}[1]
\Require $Y \in \mathbb{C}^{n\times p}_*$, a descend direction  $\eta \in \mathcal{H}_Y$, 
\Ensure $t_*= \argmin_t F(Y+t\eta) = \argmin_t \frac{1}{2}\norm{\mathcal{A}((Y+t\eta)(Y+t\eta)^*)-b}_F^2$
\par $c_0 \leftarrow \mathcal{A}(YY^*) - b$ \Comment{\# $sp^\alpha n\log n  + mn$ flops}
\par $c_1^{(1)} \leftarrow \mathcal{A}(Y\eta^*),\quad c_1^{(2)} \leftarrow \mathcal{A}(\eta Y^*),\quad  c_1\leftarrow c_1^{(1)} + c_1^{(2)}$\Comment{\# $2sp^\alpha n\log n  + mn$ flops}
\par $c_2 \leftarrow \mathcal{A}(\eta \eta^*)$ \Comment{\# $sp^\alpha n\log n $ flops}
\par $d_4 \leftarrow \ip{c_2,c_2}$, \quad $d_3 \leftarrow 2\ip{c_2,c_1}$ \Comment{\# $4mn-1$ flops}
\par $d_2 \leftarrow 2\ip{c_2,c_0} + \ip{c_1,c_1}$, \quad $d_1 \leftarrow 2\ip{c_1,c_0}$ \Comment{\# $6mn-1$ flops}
\par $C \leftarrow \bmat{4d_4 & 3d_3 & 2d_2 & d_1}$
\par $S \leftarrow roots(C), \quad t_* \leftarrow \text{the smallest real positive root in } S$
\end{algorithmic}
\end{algorithm}

\section{Estimates of Rayleigh quotient for Riemannian Hessians}\label{section:rayleigh_quotient}

In many applications, 
\eqref{lowrank_prob} or \eqref{quotient_prob} is often used 
for solving \eqref{min-psd}.
In \cite{boumal2020deterministic}, it was proven that first-order and second-order optimality conditions for the nonconvex Burer--Monteiro approach are sufficient to find the global minimizer of certain convex semi-definite programs under certain assumptions. 
In practice, even if the global minimizer of \eqref{min-psd} has a known rank $r$, one might consider solving \eqref{lowrank_prob}
or \eqref{quotient_prob} for Hermitian PSD matrices with fixed rank $p> r$.
For instance, in PhaseLift \cite{candes2013phaselift} and interferometry recovery \cite{demanet2017convex}, the minimizer to \eqref{min-psd} is rank one, but in practice optimization over the set of PSD Hermitian matrices of rank $p$ with $p\geq 2$ is often used because of a larger basin of attraction \cite{demanet2017convex, huang_solving_2017}. If $p>r$, then
an algorithm that solves \eqref{lowrank_prob} or \eqref{quotient_prob} can generate a sequence that goes to the boundary of the manifold. Numerically, the smallest $p-r$ singular values of the iterates $X_k$ will become very small as $k\rightarrow\infty$. 

In this section, we analyze the 
eigenvalues of the Riemannian Hessian near the global minimizer. More specifically, we will obtain upper and lower bounds of the Rayleigh quotient at the point $X=YY^*$ (or $\pi(Y)$) that is close to the global minimizer $\hat{X}=\hat Y \hat Y^*$ (or $\pi(\hat{Y})$). 

We first define the Rayleigh quotient and the condition number of Riemannian Hessian. 

\begin{defn}[Rayleigh quotient of Riemannian Hessian]
The Rayleigh quotient of the Riemannian Hessian of $(\mathcal{H}^{n,p}_+, g)$ is defined by 
\[
    \rho^E(X,\zeta_X)  = \frac{g_X(\Hess f(X)[\zeta_X],\zeta_X)}{g_X(\zeta_X,\zeta_X)}
\]
for $\zeta_X \in T_X \mathcal{H}^{n,p}_+$. 

The Rayleigh quotient of the Riemannian Hessian of $(\mathbb{C}^{n\times p}_*/\mathcal{O}_p, g^i)$ is defined by 
\[
\rho^i(\pi(Y),\xi_{\pi(Y)}) = \frac{g_{\pi(Y)}^i(\Hess h(\pi(Y))[\xi_{\pi(Y)}],\xi_{\pi(Y)})}{g_{\pi(Y)}^i(\xi_{\pi(Y)},\xi_{\pi(Y)})}
\]
for $\xi_{\pi(Y)} \in T_{\pi(Y)}\mathbb{C}^{n\times p}_*/\mathcal{O}_p$.
If the Rayleigh quotient has lower bound $a$ and upper bound $b$, then we define $\frac{b}{a}$ as the upper bound on the condition number of the Riemannian Hessian. 
\end{defn}

\subsection{The Rayleigh quotient estimates}
We assume that the Fr\'{e}chet Hessian $\nabla^2 f$ is well conditioned when restricted to the tangent space. Formally, our bounds will be stated in terms of the constants $A,B$ defined in the following assumption:
\begin{assm}\label{assm:rip_embedded}
For a fixed $\epsilon>0$, there exists constants $A>0$ and $B>0$ such that for all $X$ with $\norm{X-\hat{X}}_F < \epsilon$, the following inequality  holds for all $\zeta_X \in T_X \mathcal{H}^{n,p}_+$.
\[
    A\norm{\zeta_X}^2_F \leq \ip{\nabla^2 f(X)[\zeta_X],\zeta_X}_{\mathbb{C}^{n\times n}} \leq B\norm{\zeta_X}^2_F. 
\]
\end{assm}
Observe that this assumption is always satisfied for sufficiently small $\epsilon$ when $f$ is smooth. However the condition number $B/A$ might be large in general. An important case for which this assumption holds everywhere is ${f}(X) = \frac{1}{2}\norm{X-H}^2_F$ with $H$ a given Hermitian PSD matrix. In this case, $\nabla^2 f(X)$ is  the identity operator thus $A = B = 1$.

% With both Lemma \ref{lemma: rayleigh_quotient_first_order_term} and Lemma \ref{lemma: rayleigh_quotient_second_order_term}, we summarize the main result in the following theorem. 

Our main result is given in the following theorem.

\begin{thm}\label{thm:RQ}
Let $\hat X=\hat Y\hat Y^*$ be the global minimizer of \eqref{min-psd} with rank $r\leq p$. 
For $X=YY^*$ near $\hat{X}$ where $Y\in \mathbb C_*^{n\times p}$,  let $\zeta_X \in T_X\mathcal{H}^{n,p}_+$ be any tangent vector at $X$,    $\xi_{\pi(Y)} \in T_{\pi(Y)}\mathbb{C}^{n\times p}_*/\mathcal{O}_p$ be any tangent vector at $\pi(Y)$, and $\overline{\xi}_Y \in \mathcal{H}^i_Y$ be its horizontal lift at $Y$ w.r.t.\ the metric $g^i$. Let $X=U\Sigma U^*$ denote the compact SVD of $X$ and denote the $i$th diagonal entry of $\Sigma$ to be $\sigma_i$ with $\sigma_1 \geq \cdots \geq \sigma_p >0$. 
Under the Assumption \ref{assm:rip_embedded}, for any arbitrary tangent vectors $\zeta_X$ and $\xi_{\pi(Y)}$, the following bounds hold: 
\begin{enumerate}
\item For the embedded manifold,
\[
    A - \frac{2 }{\sigma_p} \norm{\nabla f(X)}  \leq \rho^E(X,\zeta_X) \leq B +\frac{2 }{\sigma_p}\norm{\nabla f(X)}.
\]
\item For the quotient manifold metric $g^i$,
\[
 2A\sigma_p - 2\norm{\nabla f(YY^*)} \leq \rho^1(\pi(Y),\xi_{\pi(Y)}) \leq B \cdot D^1_{\pi(Y)} +2\norm{\nabla f(YY^*)},
\] 
\[ 
 2A - \frac{4(\sqrt{p}+1)}{\sigma_p} \norm{\nabla f(YY^*)}
 \leq \rho^2(\pi(Y),\xi_{\pi(Y)}) \leq 4B +  \frac{4(\sqrt{p}+1)}{\sigma_p} \norm{\nabla f(YY^*)},
 \]
 \[
 A -\frac{1}{\sigma_p}\norm{\nabla f(YY^*)} \leq \rho^3(\pi(Y),\xi_{\pi(Y)}) \leq B + \frac{1}{\sigma_p}\norm{\nabla f(YY^*)},
\]
where $D^1_{\pi(Y)}$ satisfies $ 2\sigma_1 \leq D^1_{\pi(Y)}  \leq 2\left(\frac{\sigma_1^2}{\sigma_p}+\sigma_1\right)$.
\end{enumerate}
 
In particular, if $\hat X=\hat Y\hat Y^*$ has rank $p$, e.g., $\hat{X}$ has singular values $\hat{\sigma_1} \geq \cdots \geq  \hat{\sigma_p} >0$, then under the Assumption \ref{assm:rip_embedded}, we have the following limits, where the limits $X\to \hat{X}$ and $\pi(Y) \to \pi(\hat{Y})$ are taken in the sense of $\norm{X-\hat{X}}_F\to 0$ and $\norm{YY^* - \hat{Y}\hat{Y}^*}_F \to 0$:
 \begin{enumerate}
\item For the  embedded manifold 
\[
    A-\frac{2 }{\hat \sigma_p} \norm{\nabla f(\hat X)} \leq  \lim_{X \to \hat{X}} \rho^E(X,\xi_X) \leq B+\frac{2 }{\hat \sigma_p} \norm{\nabla f(\hat X)} .
\]
\item For the quotient manifold metric $g^i$,
\[
 2A\hat \sigma_p - 2\norm{\nabla f(\hat X)} \leq \lim_{\pi(Y) \to \pi(\hat{Y})}\rho^1(\pi(Y),\xi_{\pi(Y)}) \leq B \cdot D^1_{\pi(\hat Y)} +2\norm{\nabla f(\hat X)},
\] 
\[ 
 2A - \frac{4(\sqrt{p}+1)}{\hat \sigma_p} \norm{\nabla f(\hat X)}
 \leq \lim_{\pi(Y) \to \pi(\hat{Y})}\rho^2(\pi(Y),\xi_{\pi(Y)}) \leq 4B +  \frac{4(\sqrt{p}+1)}{\hat \sigma_p} \norm{\nabla f(\hat X)},
 \]
 \[
 A -\frac{1}{\hat \sigma_p}\norm{\nabla f(\hat X)} \leq\lim_{\pi(Y) \to \pi(\hat{Y})} \rho^3(\pi(Y),\xi_{\pi(Y)}) \leq B + \frac{1}{\hat \sigma_p}\norm{\nabla f(\hat X)},
\]
where $D^1_{\pi(\hat{Y})}$ satisfies $2\hat{\sigma}_1 \leq D^1_{\pi(\hat{Y})} \leq 2 \left( \frac{\hat{\sigma}_1^2}{\hat{\sigma}_p}+ \hat{\sigma}_1\right).$
\end{enumerate}
\end{thm}

\begin{rem}
If we further assume that $\nabla f(\hat X)=0$, then the limits above can be further simplified. Such an assumption $\nabla f(\hat X)=0$ may not be true in general, but it holds for all numerical examples considered in this paper, where the cost function takes the form $f(X) = \frac{1}{2}\norm{\mathcal{A}(X) - b}_F^2$ for some matrix-valued linear operator $\mathcal{A}$, and  the minimizer $\hat X$ for constrained minimization \eqref{lowrank_prob} or \eqref{min-psd} satisfies $f(\hat X)=0$. Thus $\hat X$ is also the global minimizer for minimizing $f(X)$ over all $X\in \mathbb C$, which implies  $\nabla f(\hat X) = 0$.
\end{rem}
\begin{rem}
We can define the ratio of the upper and lower bounds of the Rayleigh quotient as the  upper bound on the condition number of the Riemannian Hessian.
Then under the assumption $\nabla f(\hat X)=0$,
 the limit of the condition number of the Riemannian Hessian for the Bures-Wasserstein metric $g^1$ depends on the condition number of the minimizer $\hat X$. This 
 reflects a significant difference between $g^1$ and the other two metrics. \end{rem}
 \begin{rem}
For the case $\nabla f(\hat{X}) \neq 0$, if $\| \nabla f(\hat{X}) \| $ is sufficiently small in the sense that
\begin{equation} \label{eq:bound}
\| \nabla f(\hat{X}) \| < a,
\end{equation}
where $a$ is equal to $\hat{\sigma}_p A / 4$,
$\hat{\sigma}_p A / 8 / (\sqrt{p}+1)$, and $\hat{\sigma}_p A / 2$ for the embedded metric,
the condition numbers of the embedded metric, quotient metric $g^2$ and $g^3$  
are on the 
order of $B / A$.  The quotient manifold with $g^1$ is still different from the other metrics since the condition number of its Riemannian Hessian additionally depends on the ratio $\hat{\sigma}_1/\hat{\sigma}_p$.
\end{rem}

The rest of this subsection is the proof of Theorem \ref{thm:RQ}. 
By the expressions of Riemannian Hessian, we have
\begin{equation*}
    \rho^E(X,\zeta_X)   =
    \frac{\ip{\nabla^2 f(X)[\zeta_X],\zeta_X}_{\mathbb{C}^{n\times n}}}{g_X(\zeta_X,\zeta_X)} + \frac{g_X\left(P^p_X\left(\nabla f(X)(X^\dagger\zeta_X^p)^* + (\zeta_X^p X^\dagger)^*\nabla f(X)\right), \zeta_X\right)}{g_X(\zeta_X,\zeta_X)}.
\end{equation*}
\begin{equation*}
    \rho^1(\pi(Y),\xi_{\pi(Y)}) = \frac{\ip{\nabla^2 f(YY^*)[Y\overline{\xi}_Y^*+\overline{\xi}_YY^*], Y\overline{\xi}_Y^*+\overline{\xi}_YY^*}_{\mathbb{C}^{n\times n} }}{g^1_Y(\overline{\xi}_Y,\overline{\xi}_Y)}
     + \frac{g^1_Y(2\nabla f(YY^*)\overline{\xi}_Y ,\overline{\xi}_Y)}{g^1_Y(\overline{\xi}_Y,\overline{\xi}_Y)}.
\end{equation*}
\begin{eqnarray*}
\rho^2(\pi(Y),\xi_{\pi(Y)}) &=& \frac{\ip{\nabla^2 f(YY^*)[Y\overline{\xi}_Y^*+\overline{\xi}_YY^*], Y\overline{\xi}_Y^*+\overline{\xi}_YY^*}_{\mathbb{C}^{n\times n} }}{g^2_Y(\overline{\xi}_Y,\overline{\xi}_Y)}  \\
&& + \frac{\ip{\nabla f(YY^*)P_Y^\perp \overline{\xi}_Y, \overline{\xi}_Y}_{\mathbb{C}^{n\times p}}}{g^2_Y(\overline{\xi}_Y,\overline{\xi}_Y)} + \frac{\ip{ P_Y^\perp \nabla f(YY^*) \overline{\xi}_Y, \overline{\xi}_Y}_{\mathbb{C}^{n\times p}}}{g^2_Y(\overline{\xi}_Y,\overline{\xi}_Y)}  \\ 
&& + \frac{\ip{Y\overline{\xi}^*_Y\overline{\xi}_Y ,2 \nabla f(YY^*)Y(Y^*Y)^{-1}}_{\mathbb{C}^{n\times p}}}{g_Y^2(\overline{\xi}_Y, \overline{\xi}_Y)} \\
&& - \frac{\ip{ \overline{\xi}_YY^*\overline{\xi}_Y,2 \nabla f(YY^*)Y(Y^*Y)^{-1}}_{\mathbb{C}^{n\times p}}}{g_Y^2(\overline{\xi}_Y, \overline{\xi}_Y)}.
\end{eqnarray*}

\begin{eqnarray*}
    \rho^3(\pi(Y),\xi_{\pi(Y)})  
    &=& \frac{\ip{\nabla^2 f(YY^*)[Y\overline{\xi}_Y^*+\overline{\xi}_YY^*], Y\overline{\xi}_Y^*+\overline{\xi}_YY^*}_{\mathbb{C}^{n\times n} }}{g^3_Y(\overline{\xi}_Y,\overline{\xi}_Y)} \\
    &&+ \frac{g^3_Y((I-P_Y)\nabla f(YY^*)(I-P_Y)\overline{\xi}_Y(Y^*Y)^{-1},\overline{\xi}_Y)}{g^3_Y(\overline{\xi}_Y,\overline{\xi}_Y)}. 
\end{eqnarray*}

Observe that the leading terms in the above Rayleigh quotients take similar form:  the numerator involves the Fr\'{e}chet Hessian $\nabla^2 f$, and the denominator is the induced norm of  tangent vector from the respective Riemannian metric. We call the leading term  \textit{second order term} (SOT) as it involves  Fr\'{e}chet Hessian of $f$ as the second order information of $f$ and we call the other terms that follow the leading term \textit{first order terms} (FOTs) as they only contain the first order Fr\'{e}chet gradient.

Under the Assumption \ref{assm:rip_embedded}, we get bounds of the SOT in $\rho^E(X,\zeta_X)$ as:
\[
 A = A \frac{\norm{\zeta_X}^2_F}{g_X(\zeta_X,\zeta_X)}\leq \frac{\ip{\nabla^2 f(X)[\zeta_X],\zeta_X}_{\mathbb{C}^{n\times n}}}{g_X(\zeta_X,\zeta_X)}  \leq B\frac{\norm{\zeta_X}^2_F}{g_X(\zeta_X,\zeta_X)} = B.
\]

For the quotient manifold, observe that $Y\overline{\xi}_Y^* + \overline{\xi}_YY^* \in T_{YY^*}\mathcal{H}^{n,p}_+$. Hence Assumption \ref{assm:rip_embedded} also applies and we get
\[
A\frac{\norm{Y\overline{\xi}_Y^*+\overline{\xi}_YY^*}^2_F}{g^i_Y\left(\overline{\xi}_Y,\overline{\xi}_Y\right)} \leq \frac{\ip{\nabla^2 f(YY^*)[Y\overline{\xi}_Y^*+\overline{\xi}_YY^*], Y\overline{\xi}_Y^*+\overline{\xi}_YY^*}_{\mathbb{C}^{n\times n} }}{g^i_Y(\overline{\xi}_Y,\overline{\xi}_Y)}  \leq B \frac{\norm{Y\overline{\xi}_Y^*+\overline{\xi}_YY^*}^2_F}{g^i_Y\left(\overline{\xi}_Y,\overline{\xi}_Y\right)}.
\]

Hence the analysis of SOT for the quotient manifold now turns to analyzing 
$\frac{\norm{Y\overline{\xi}_Y^*+\overline{\xi}_YY^*}^2_F}{g^i_Y\left(\overline{\xi}_Y,\overline{\xi}_Y\right)}$. We denote its infimum and supremum by 
\begin{eqnarray*}
C^i_{\pi(Y)} &:=& \inf_{\xi_{\pi(Y)} \in T_{\pi(Y)}\mathbb{C}_*^{n\times p}/\mathcal{O}_p}\frac{\norm{Y\overline{\xi}_Y^*+\overline{\xi}_YY^*}^2_F}{g^i_Y(\overline{\xi}_Y,\overline{\xi}_Y)},\\
D^i_{\pi(Y)} &:=& \sup_{\xi_{\pi(Y)} \in T_{\pi(Y)}\mathbb{C}_*^{n\times p}/\mathcal{O}_p}\frac{\norm{Y\overline{\xi}_Y^*+\overline{\xi}_YY^*}^2_F}{g^i_Y(\overline{\xi}_Y,\overline{\xi}_Y)}.
\end{eqnarray*}
The subscript is used to emphasize that the infimum and supremum are dependent on $\pi(Y)$. The next lemma characterizes these infimum and supremum.

\begin{lem}\label{lemma: rayleigh_quotient_second_order_term} For any $Y \in \pi^{-1}(\pi(Y))$, let $YY^* = U\Sigma U^*$ denote the compact SVD of $YY^*$ and denote the $i$th diagonal entry of $\Sigma$ by $\sigma_i$ with $\sigma_1 \geq \cdots \geq \sigma_p >0$. Then the following estimates for the infimum $C^i_{\pi(Y)}$ and the supremum $D^i_{\pi(Y)}$ of $\frac{\norm{Y\overline{\xi}_Y^*+\overline{\xi}_YY^*}^2_F}{g^i_Y\left(\overline{\xi}_Y,\overline{\xi}_Y\right)}$ hold:
\begin{eqnarray*}
&C^1_{\pi(Y)}& = 2\sigma_p, \quad  2\sigma_1 \leq D^1_{\pi(Y)}\leq 2\left(\frac{\sigma_1^2}{\sigma_p}+\sigma_1\right). \\
&C^2_{\pi(Y)}& = 2, \quad D^2_{\pi(Y)} = 4.\\
&C^3_{\pi(Y)}& = D^3_{\pi(Y)} = 1.
\end{eqnarray*}
\end{lem}

Next we estimate the FOTs in the  Rayleigh quotient. The result is given in the next lemma.
\begin{lem}\label{lemma: rayleigh_quotient_first_order_term}
Let $X = YY^*$ for any $Y \in \pi^{-1}(\pi(Y))$ with $X \in \mathcal{H}^{n,p}_+$ and $\pi(Y) \in \mathbb{C}^{n\times p}_*/\mathcal{O}_p$. Let $U\Sigma U^*$ be the compact SVD of $X$ and denote the $i$th diagonal entry of $\Sigma$ with $\sigma_1 \geq \cdots \geq \sigma_p >0$. Then we have the following bounds for the FOTs in Rayleigh quotient of Riemannian Hessian. 
\begin{enumerate}
    \item For the embedded manifold we have 
    \[
        \abs{\text{FOT} } \leq \frac{2}{\sigma_p} \norm{\nabla f(X)}.
    \]
    \item For the quotient manifold with metric $g^1$ we have 
    \[
        \abs{\text{FOT}} \leq 2 \norm{\nabla f(YY^*)}.
    \]
    \item For the quotient manifold with $g^2$ we have 
    \[
        \abs{\text{FOTs}} \leq \frac{4(\sqrt{p}+1)}{\sigma_p}\norm{\nabla f(YY^*)}.
    \]
    \item For the quotient manifold with $g^3$ we have 
    \[
        \abs{\text{FOTs}} \leq \frac{1}{\sigma_p} \norm{\nabla f(YY^*)}.
    \]
\end{enumerate}
\end{lem}

The proofs for Lemma \ref{lemma: rayleigh_quotient_first_order_term} and Lemma \ref{lemma: rayleigh_quotient_second_order_term} are given in Appendix \ref{sec-appendix-lemma}.
With  Lemma \ref{lemma: rayleigh_quotient_first_order_term} and Lemma \ref{lemma: rayleigh_quotient_second_order_term}, the proof of Theorem \ref{thm:RQ} is concluded.

\subsection{The limit of the Rayleigh quotient for a rank-deficient minimizer $\hat X$}

Next we consider the rank deficient case $p> r$ where $r$ is the rank of the minimizer $\hat X$, i.e., the minimizer $\hat X$ lies on the boundary of the constraint manifold. Under the Assumption $\nabla f(\hat X)=0$,  any convergent algorithm that solves \eqref{lowrank_prob} or \eqref{quotient_prob} will generate a sequence such that both $\sigma_{r+1},\cdots, \sigma_p$ and $\nabla f(X)$ will vanish as $X \to \hat{X}$. We make one more assumption  for a simpler  quantification of the lower and upper bounds of the Rayleigh quotient near the minimizer.
\begin{assm}\label{assm:gradient_vanish_speed}
For a sequence $\{X_k\}$ with $X_k \in \mathcal{H}^{n,p}_+$ (or $\pi(Y_k) \in \mathbb{C}^{n\times p}_*/\mathcal{O}_p$ ) that converges to the minimizer $\hat{X}$ (or $\pi(\hat{Y})$), let ${(\sigma_p)}_k$ be the smallest nonzero singular value of $X_k = Y_kY_k^*$, assume the following limits hold. 
\begin{enumerate}
\item For the embedded manifold, 
\[
\lim_{k \to \infty} \frac{2}{{(\sigma_p)}_k} \norm{\nabla f(X_k)} \leq \frac{A}{2}.
\]
\item For the quotient manifold with metric $g^1$,
\[
\lim_{k \to \infty} \frac{1}{{(\sigma_p)}_k} \norm{\nabla f(Y_kY_k^*)} \leq \frac{A}{2}.
\]
\item For the quotient manifold with metric $g^2$,
\[
\lim_{k \to \infty} \frac{4(\sqrt{p}+1)}{{(\sigma_p)}_k} \norm{\nabla f(Y_kY_k^*)} \leq A. 
\]
\item For the quotient manifold with metric $g^3$, 
\[
\lim_{k \to \infty} \frac{1}{{(\sigma_p)}_k} \norm{\nabla f(Y_k Y_k^*)} \leq \frac{A}{2}.
\]
\end{enumerate}
\end{assm}
We remark that  Assumption \ref{assm:gradient_vanish_speed} may not always hold. In the next section, we will give some numerical evaluation of this assumption for four examples listed in Figure \ref{fig:gradOverSigmap_eigenvalueProblem} (eigenvalue problem),
Figure  \ref{fig:gradOverSigmap_matrixCompletion} (matrix completion),
Figure 
\ref{fig:gradOverSigmap_phaseRetrieval} (phase retrieval),
and 
Figure  \ref{fig:gradOverSigmap_interferometry} (interferometry recovery). Assumption \ref{assm:gradient_vanish_speed} holds numerically in most of these tests.
\begin{rem}In general, there exists a sequence such that the FOT in $\rho^3(\pi(Y),\xi_{\pi(Y)})$ may blow up. Consider the following simple example of eigenvalue problem. 
\begin{equation*}
	\MINone{X}{f(X) = \frac{1}{2}\norm{X-\hat{X}}_F^2}{X \in \mathcal{H}^{3,2}_+},
\end{equation*}
where $\hat{X} = \bmat{1&0&0\\0&0&0\\0&0&0}$ is a rank-1 minimizer. Suppose $X$ takes the simple diagonal form $X = \bmat{\sigma_1 & 0 &0 \\ 0 & \sigma_2 &0 \\ 0&0 &0}$. Then we have 
\[
\nabla f(X) = \bmat{\sigma_1 -1 &0 &0  \\ 0 & \sigma_2 & 0 \\ 0&0&0}. 
\]
Since $\nabla f(X) \to 0$ as $X \to \hat{X}$, we have $\sigma_1 \to 1$ and $\sigma_2 \to 0$.

Recall that the FOT in $\rho^3(\pi(Y),\xi_{\pi(Y)})$ is 
\[
\frac{g^3_Y((I-P_Y)\nabla f(YY^*)(I-P_Y)\overline{\xi}_Y(Y^*Y)^{-1},\overline{\xi}_Y)}{g^3_Y(\overline{\xi}_Y,\overline{\xi}_Y)} = \frac{\ip{
    \nabla f(YY^*)Y_\perp K,Y_\perp K}_{\mathbb{C}^{n\times p}}}{2\norm{YSY^*}_F^2+\norm{Y_\perp KY^*}_F^2}. 
\]
Hence if we choose  $S = 0$ and $Y_\perp K = \bmat{0&1\\0&0\\0&0}$, we have 
\[
\frac{\ip{
    \nabla f(YY^*)Y_\perp K,Y_\perp K}_{\mathbb{C}^{n\times p}}}{2\norm{YSY^*}_F^2+\norm{Y_\perp KY^*}_F^2} = \frac{\sigma_1 - 1}{ \sigma_2},
\]
whose limit is dependent on the path that the tuple $(\sigma_1, \sigma_2)$ goes to $(1,0)$ and hence may blow up. 
\end{rem}
 
If $\hat X$ has rank $r<p$ and $\{X_k\}$ is a sequence that satisfies Assumption \ref{assm:gradient_vanish_speed},
then Theorem \ref{thm:RQ} implies
\begin{enumerate}
    \item For the embedded manifold we have 
    \[
        \frac{A}{2} \leq \lim_{k \to \infty } \rho^E(X_k,\xi_{X_k}) \leq B + \frac{A}{2}.
    \]
    \item For the quotient manifold with  metric $g^i$ we have 
    \[
        A \leq \lim_{k \to \infty } \frac{\rho^1(\pi(Y_k),\xi_{\pi(Y_k)})}{{(\sigma_p)}_k} \leq B \lim_{k\to \infty} \frac{D^1_{\pi(Y_k)}}{(\sigma_p)_k} +2A,
    \]
    \[
        A \leq \lim_{k \to \infty } \rho^2(\pi(Y_k),\xi_{\pi(Y_k)}) \leq 4B +A,
    \]
    \[
        \frac{A}{2} \leq \lim_{k \to \infty } \rho^3(\pi(Y_k),\xi_{\pi(Y_k)}) \leq B+\frac{A}{2},
    \]
    where $\lim\limits_{k\to \infty} \frac{D^1_{\pi(Y_k)}}{(\sigma_p)_k} \geq \lim\limits_{k \to \infty} \frac{2 (\sigma_1)_k}{(\sigma_p)_k} = + \infty$ since $\sigma_p\to \hat \sigma_p=0$. 
\end{enumerate} 
Notice that the condition number in Bures-Wassertein metric $g^1$ is fundamentally different from the other ones since it is the only metric that blows up.

\section{Numerical experiments}\label{section:numerical_experiments}
In this section, we report on the numerical performance of the the conjugate gradient methods on three kinds of cost functions of $f(X)$: eigenvalue problem, matrix completion, phase-retrieval, and interferometry. In particular, we implement and compare the following four algorithms:
\begin{enumerate}[topsep=0pt,itemsep=0.5ex,partopsep=1ex,parsep=1ex]
    \item Riemannian CG on the quotient manifold $(\mathbb{C}^{n\times p}_*/\mathcal{O}_p, g^1)$, i.e.,  Algorithm \ref{alg:CG_quotient} with metric $g^1$. This algorithm is equivalent to Burer--Monteiro CG, that is, CG applied directly to \eqref{min-BR}.
   \item Riemannian CG on the quotient manifold $(\mathbb{C}^{n\times p}_*/\mathcal{O}_p, g^2)$, i.e.,  Algorithm \ref{alg:CG_quotient} with metric $g^2$. The same metric $g^2$ was used in \cite{huang_solving_2017}.
    \item Riemannian CG on the quotient manifold $(\mathbb{C}^{n\times p}_*/\mathcal{O}_p, g^3)$, i.e.,  Algorithm \ref{alg:CG_quotient} with metric $g^3$, and also a specific retraction, vector transport and initial step as described in Section \ref{sec:embedded_equal_quotient}. This special implementation is equivalent to Riemannian CG
    on embedded manifold, i.e., Algorithm \ref{alg:CG_embedded}.
    \item Burer--Monteiro L-BFGS method, that is, using the L-BFGS method directly applied to \eqref{min-BR}. This method was used in \cite{demanet2017convex}.
\end{enumerate}

\subsection{Eigenvalue problem}
For any $n$-by-$n$ Hermitian PSD matrix $A$, its top $p$ eigenvalues and associated eigenvectors can be found by solving the following minimization problem: 
\begin{equation*}
	\MINone{X}{f(X) := \frac{1}{2}\norm{X-A}_F^2}{X \in \mathcal{H}^{n,p}_+},
\end{equation*} 
or equivalently
\begin{equation*}
    \MINone{\pi(Y)}{h(\pi(Y)) :=  \frac{1}{2} \norm{YY^*-A}_F^2}{\pi(Y) \in \mathbb{C}^{n\times p}_*/\mathcal{O}_p}.
\end{equation*}

It is easy to verify that 
\begin{equation*}
    \nabla f(X) = X-A,\quad 
    \nabla^2f(X)[\zeta_X] = \zeta_X, \quad \zeta_X \in \mathbb{C}^{n\times n}.
\end{equation*}

In practice we only need $A$ as an operator $A: v \mapsto Av$.  We consider a numerical test for a random Hermitian PSD matrix $A$ of size 50\,000-by-50\,000 with rank  $10$.
We solve the minimization problem above with $p=15$. Obivously, the minimizer is rank-10 thus rank deficient for $\mathbb{C}^{n\times p}_*/\mathcal{O}_p$ with $p=15$. 
This corresponds to a scenario of finding the eigenvalue decomposition of a low rank Hermitian PSD matrix $A$ with estimated rank at most $15$.
The results are shown in Figure \ref{fig:eigenvalue_problem}.
The initial guess is the same random initial matrix for all four algorithms. We see that the simpler Burer--Monteiro approach, including the L-BFGS method and the CG method with metric $g^1$, is significantly slower.
 
In the second test of Figure \ref{fig:eigenvalue_problem2}, the minimizer has rank $r= 15$, and the fixed rank for the manifold is also set to $p= 15$; i.e., there is no rank deficiency. But the condition number of the the minimizer $A$ causes a difference in the asymptotic convergence rate for CG method with metric $g^1$. In \ref{eigenvalue_problem2-a}, the condition number of $A$ is large and we observe slower asymptotic convergence rate for CG method with metric $g^1$; while in  \ref{eigenvalue_problem2-b}, the condition number of $A$ is smaller and the asymptotic convergence rate becomes much faster. This is consistent with Theorem \ref{thm:RQ}. In the third test of Figure \ref{fig:gradOverSigmap_eigenvalueProblem}, we show the ratio term $\frac{\norm{\nabla f(Y_k Y_k^*)} }{ ({\sigma_p})_k }$ in Assumption 
 \ref{assm:gradient_vanish_speed} versus the iteration number $k$. This ratio does not blow up as $\pi(Y_k)$ converges to $\pi(\hat{Y})$.

\begin{figure}[htpb]
    \centering
    \includegraphics[width=0.5\textwidth]{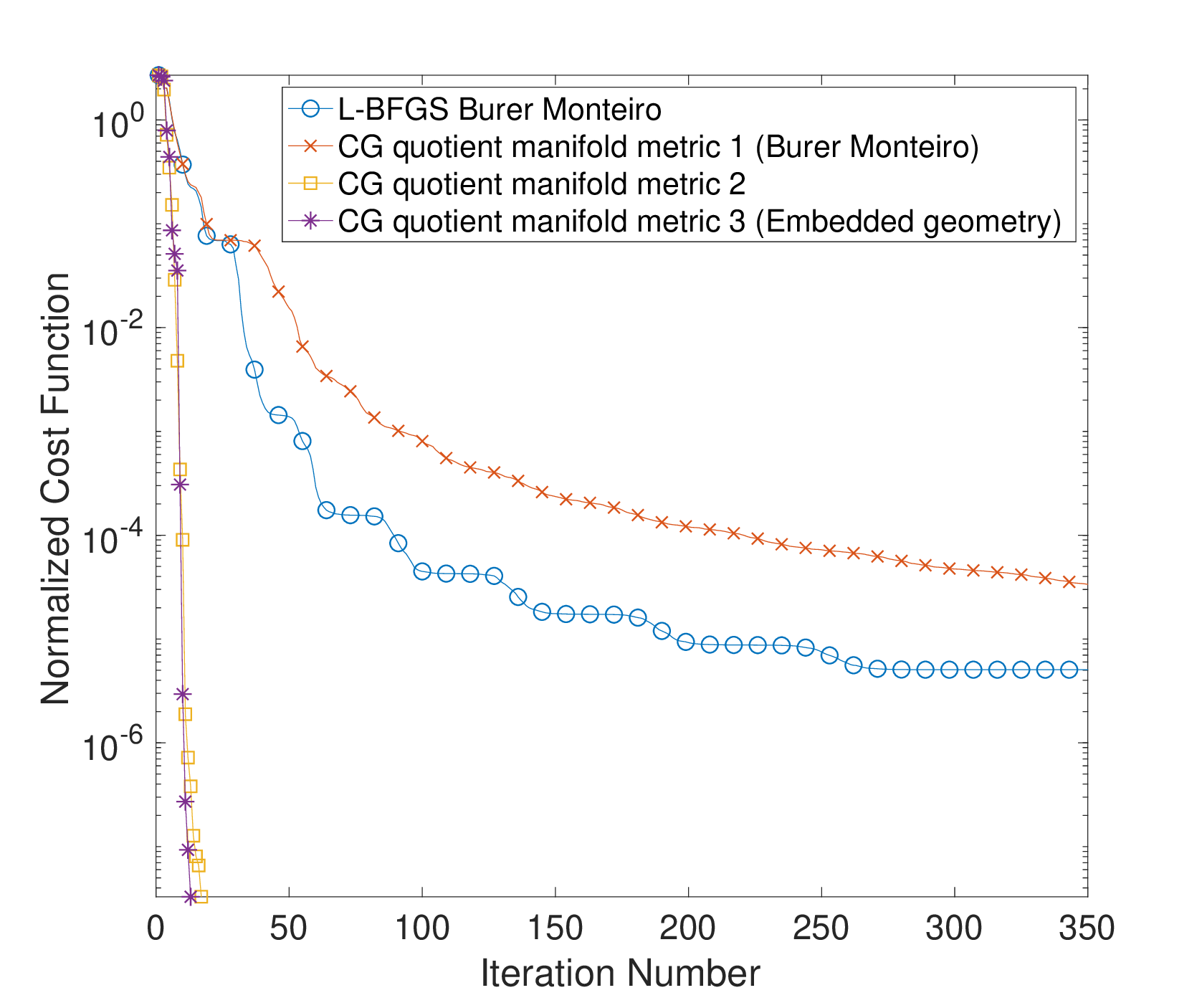}
    \caption{Eigenvalue problem of a random 50\,000-by-50\,000 PSD matrix of rank 10 solved on the rank 15 manifold: a comparison of normalized cost function value $\frac{\norm{Y_kY_k^* - A}_F}{\norm{A}_F}$ decrease versus iteration number $k$ when using L-BFGS approach and CG method with metric $g^i,i=1,2,3$.}
     \label{fig:eigenvalue_problem}
\end{figure}

\begin{figure}[htpb]
\centering
\subfigure[$\frac{\hat{\sigma}_1}{\hat{\sigma}_p} = 10^6$]{
\includegraphics[width=0.4\textwidth]{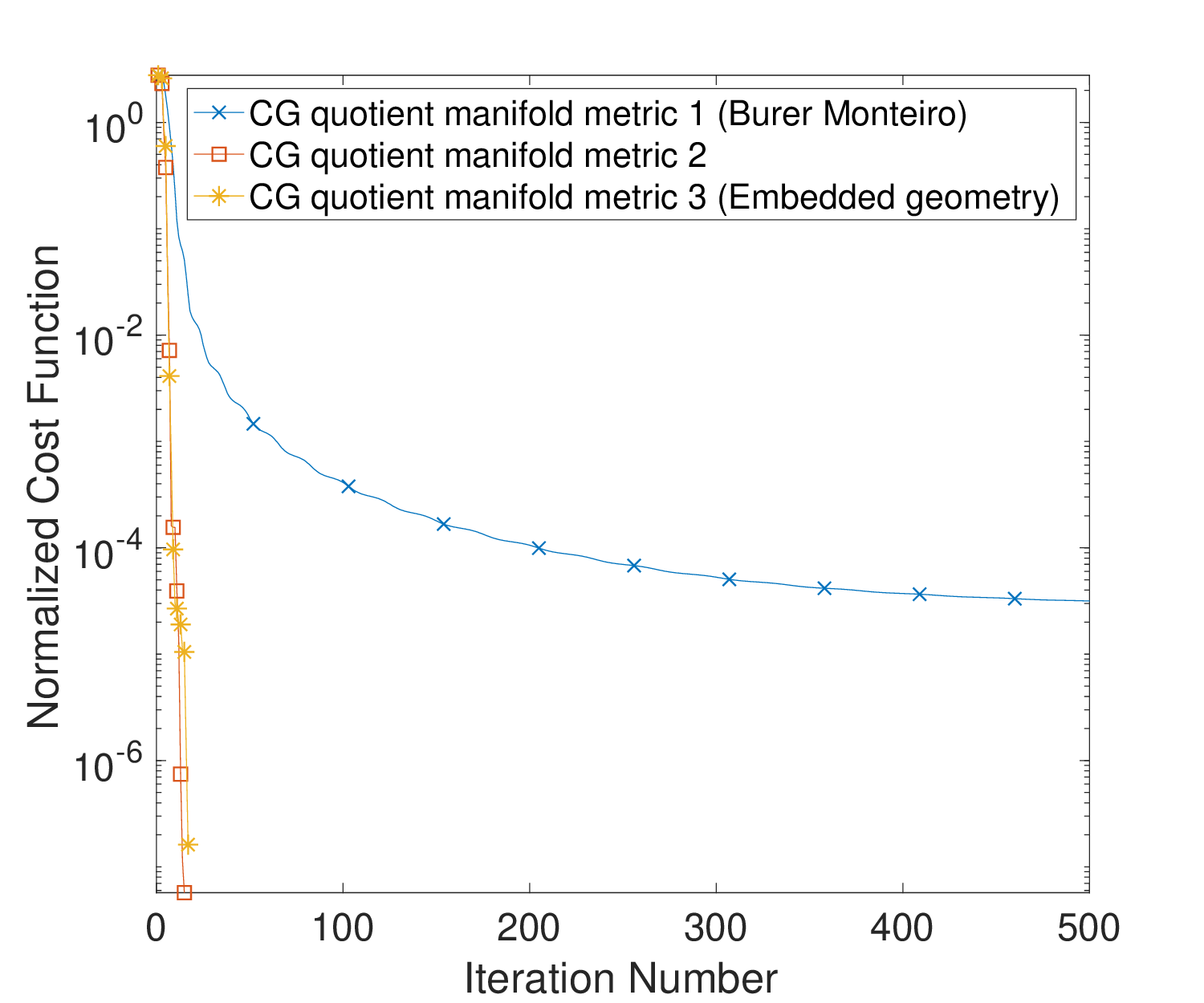}
\label{eigenvalue_problem2-a}
}
\subfigure[$\frac{\hat{\sigma}_1}{\hat{\sigma}_p} = 10^3$]{
\includegraphics[width=0.4\textwidth]{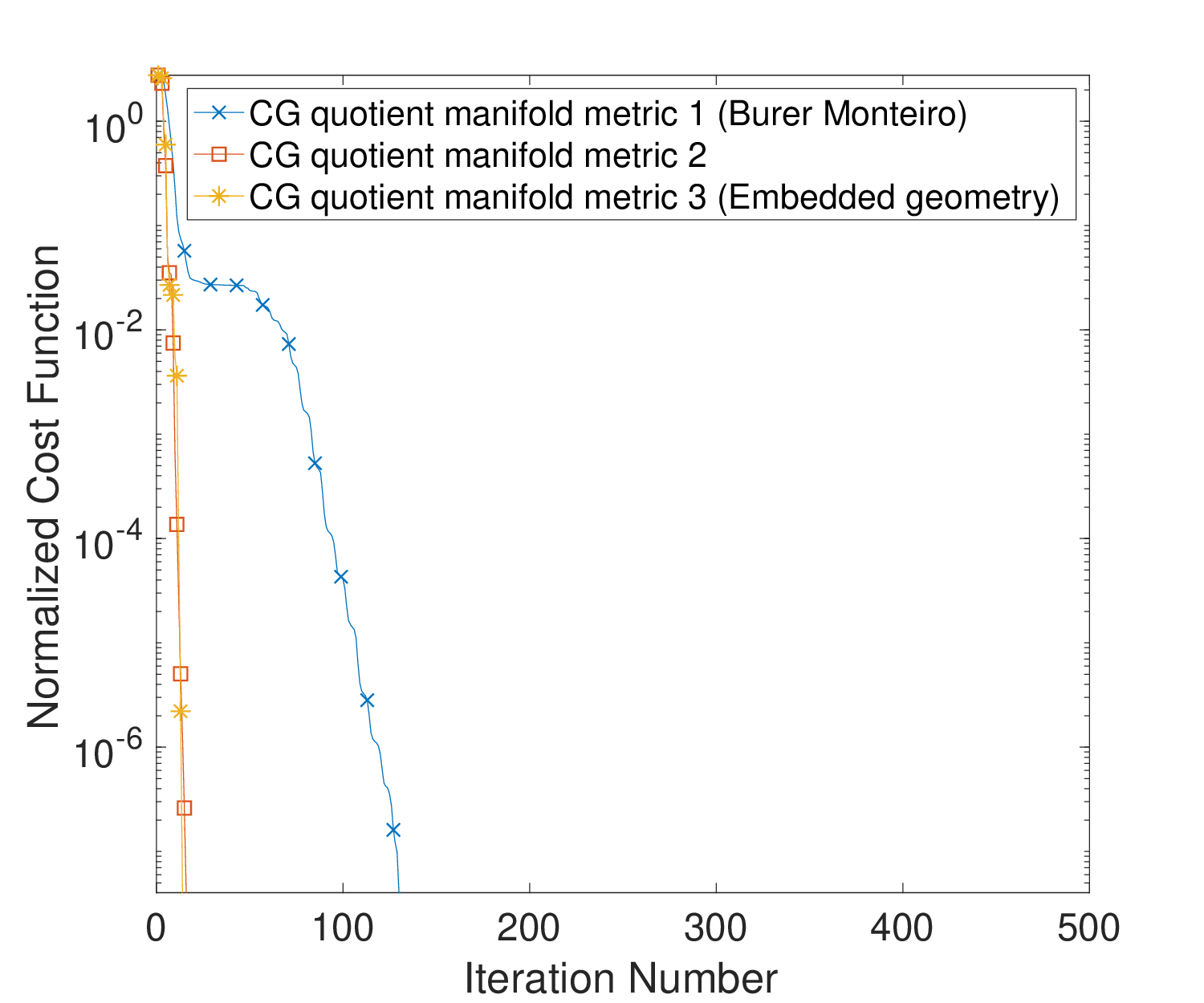}
\label{eigenvalue_problem2-b}
}
\caption{Numerical justification of Theorem \ref{thm:RQ} for the eigenvalue problem of a random 50\,000-by-50\,000 PSD matrix of rank 15 on the rank 15 manifold. Effect of condition number of $A$ on the convergence speed of normalized cost function value $\frac{\norm{Y_kY_k^* - A}_F}{\norm{A}_F}$ versus iteration number $k$. (a): when the condition number of $A$ is large, CG with metric $g^1$ is slower; (b): when the condition number of $A$ is smaller, CG with metric $g^1$ becomes faster.}
\label{fig:eigenvalue_problem2}
\end{figure}

\begin{figure}[htbp!]
\centering
\subfigure[L-BFGS Burer Monteiro]{
\includegraphics[width=0.23\textwidth]{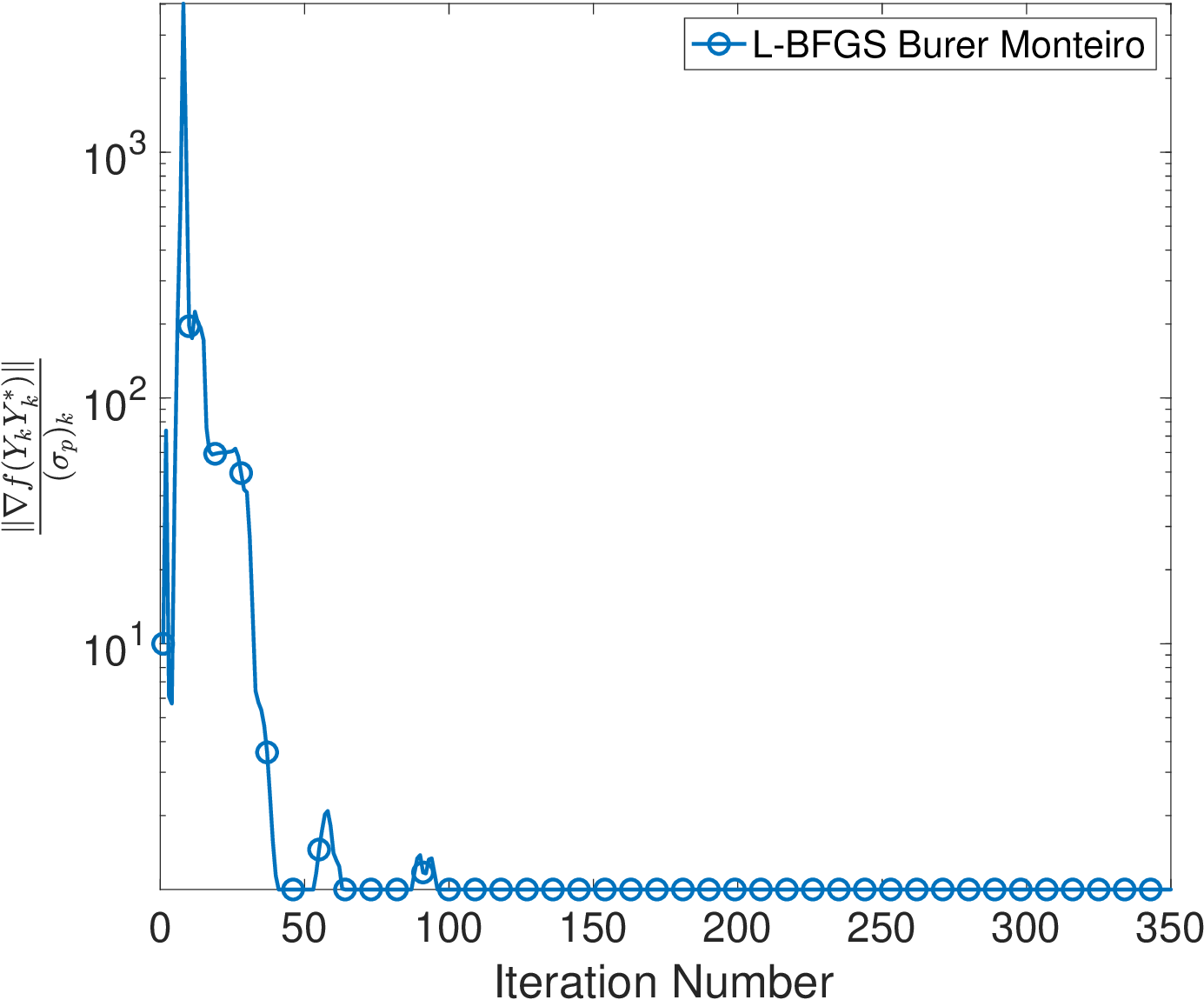}
}
\subfigure[CG quotient manifold metric 1 (Burer Monteiro)]{
\includegraphics[width=0.23\textwidth]{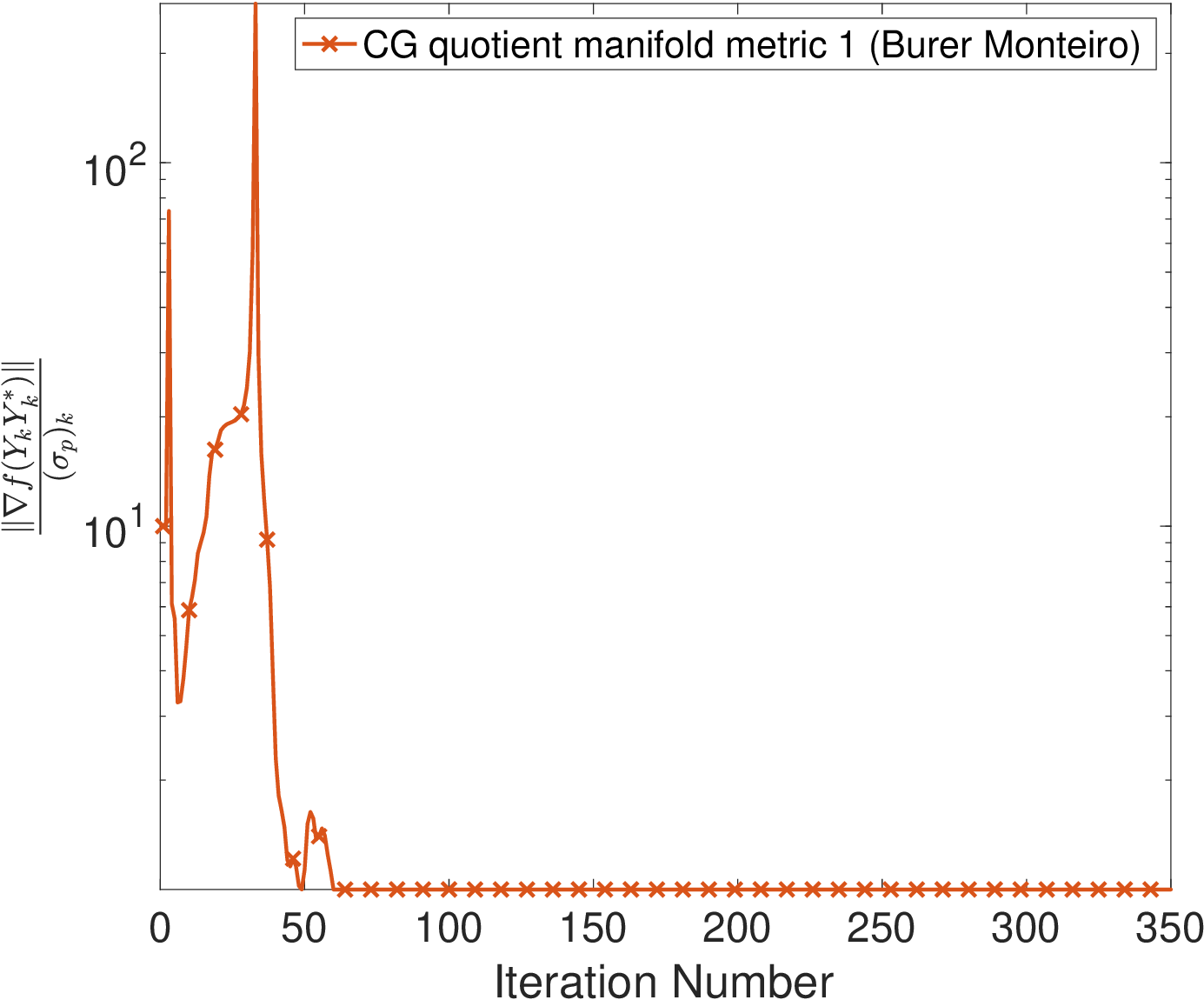}
}
\subfigure[CG quotient manifold metric 2]{
\includegraphics[width=0.23\textwidth]{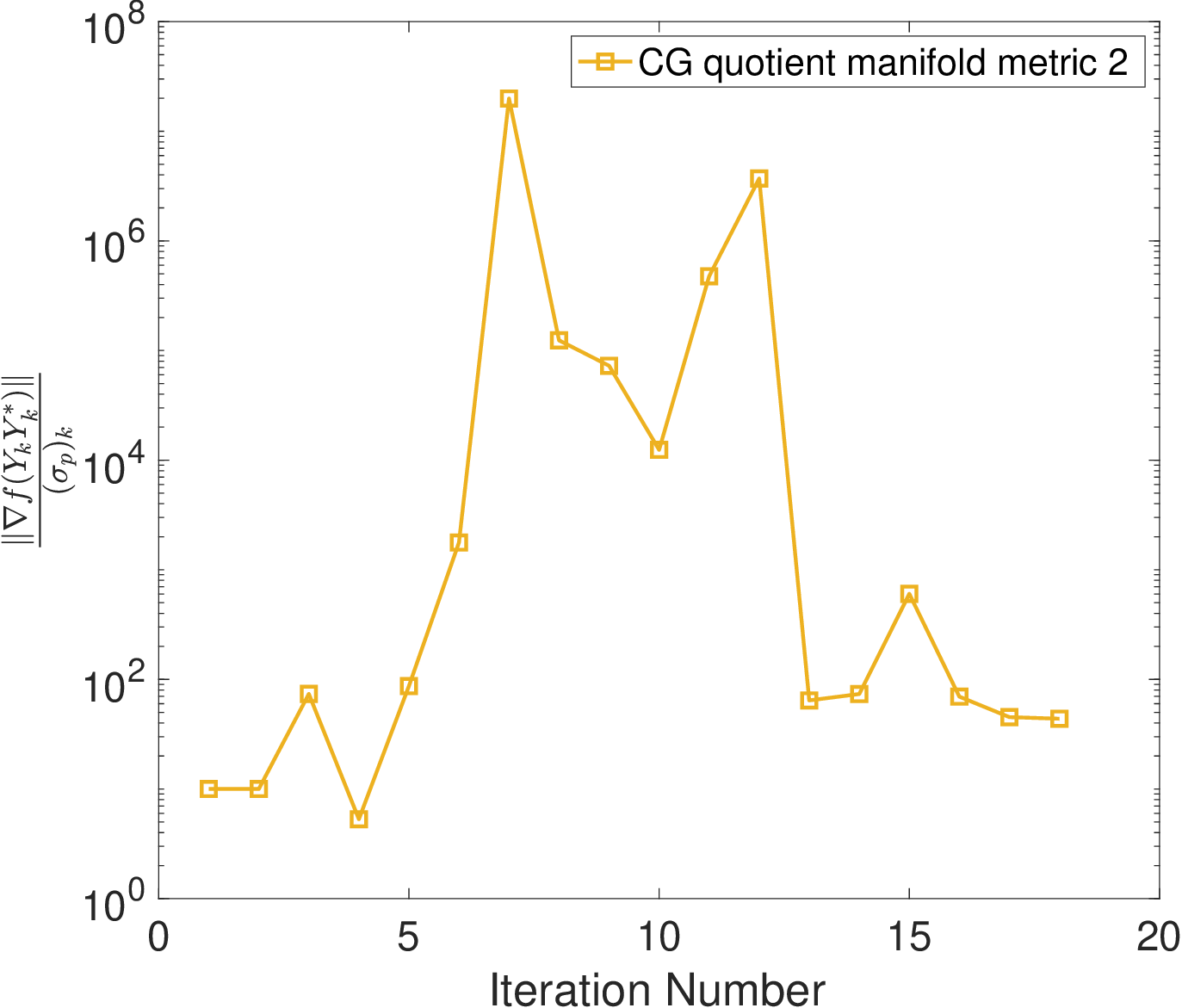}
}
\subfigure[CG quotient manifold metric 3 (Embedded geometry)]{
\includegraphics[width=0.23\textwidth]{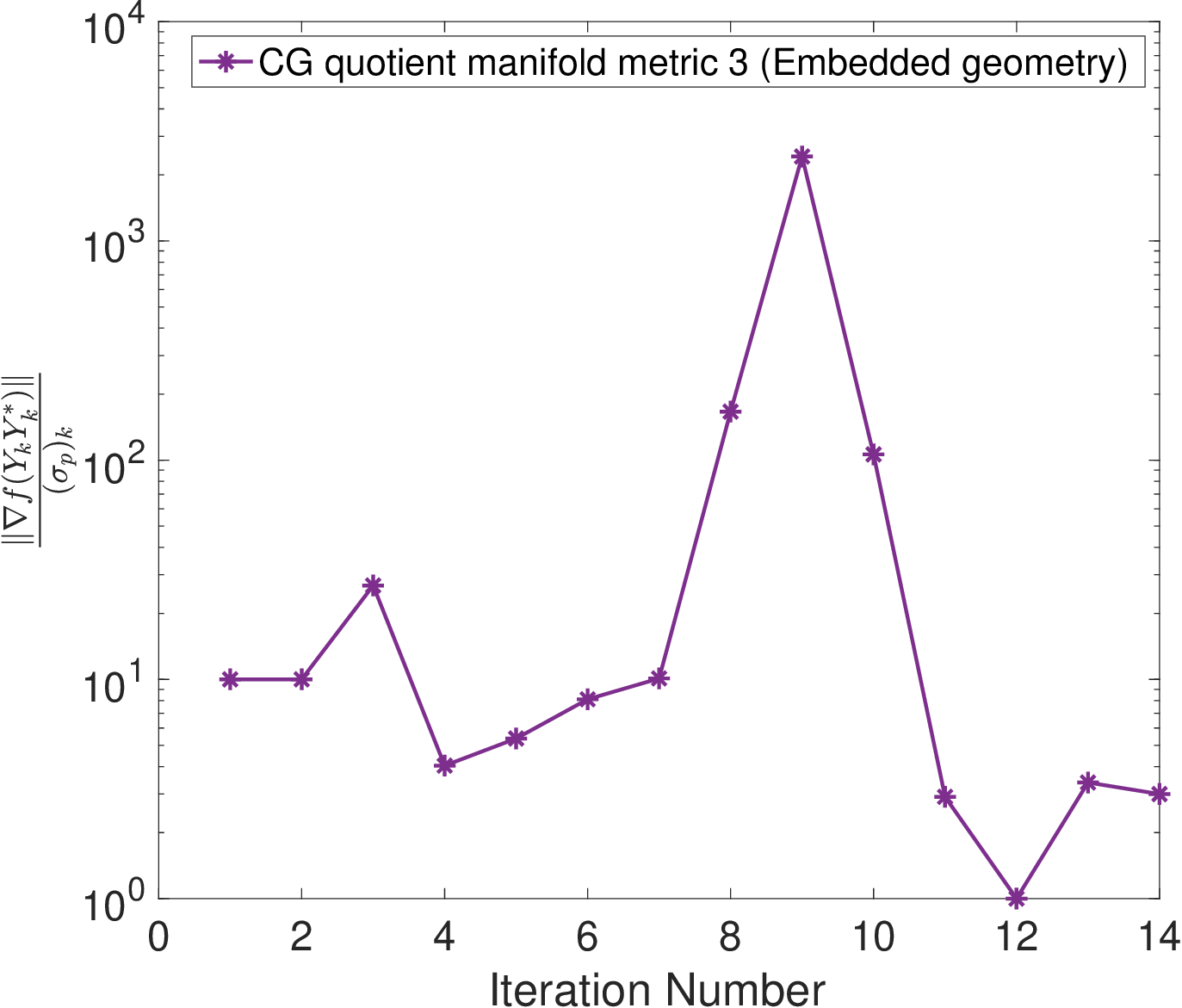}
}
\caption{Numerical justification of Assumption 
 \ref{assm:gradient_vanish_speed} for the eigenvalue problem of a random 50\,000-by-50\,000 PSD matrix of rank 10 on the rank 15 manifold, same setup as the numerical test shown in Fig \ref{fig:eigenvalue_problem}. Plots show the ratio term $\frac{\norm{\nabla f(Y_k Y_k^*)} }{ ({\sigma_p})_k }$ in Assumption 
 \ref{assm:gradient_vanish_speed} versus the iteration number $k$ for L-BFGS approach and CG method with metric $g^i,i=1,2,3$.}
\label{fig:gradOverSigmap_eigenvalueProblem}
\end{figure}

\subsection{Matrix completion}
 Let $\Omega$ be a subset of of the complete set $\{1,\cdots, n\} \times \{1,\cdots,n\}$. 
Then the projection operator onto $\Omega$ is a sampling operator defined as 
\begin{equation*}
    P_\Omega : \mathbb{C}^{n\times n} \rightarrow \mathbb{C}^{n\times n}:  X_{i,j} \mapsto  \left\{\begin{aligned}
        X_{i,j} \quad \text{if } (i,j)\in \Omega, \\
        0 \quad \text{if } (i,j)\notin \Omega.
    \end{aligned}\right.
\end{equation*}
The original matrix completion problem has no symmetry or Hermitian constraint. Here, we just consider an artificial Hermitian matrix completion problem for a given $A\in \mathcal H^{n,p}_+$: 
\[
	\MINone{X}{f(X) := \frac{1}{2}\norm{P_\Omega(X-A)}_F^2}{X \in \mathcal{H}^{n,p}_+},
\] 
or equivalently
\[
    \MINone{\pi(Y)}{h(\pi(Y)) :=  \frac{1}{2} \norm{P_\Omega(YY^*-A)}_F^2}{\pi(Y) \in \mathbb{C}^{n\times p}_*/\mathcal{O}_p}.
\]
 
Straightforward calculation shows
\begin{equation*}
    \nabla f(X) = P_\Omega(X-A),\quad 
    \nabla^2f(X)[\zeta_X] = P_\Omega(\zeta_X), \quad \zeta_X \in \mathbb{C}^{n\times n}.
\end{equation*}

 We consider a Hermitian PSD matrix $A\in \mathbb C^{n\times n}$ with $n=10\,000$ and $P_\Omega$ a random 90\% sampling operator. In the first test of Figure \ref{fig:matrixcompletion-a}, the minimizer has rank $r=25$, and the fixed rank for the manifold is set to $p=30$. 
 In the second test of Figure \ref{fig:matrixcompletion-b}, 
the minimizer has rank $r=25$, and the fixed rank for the manifold is set to $p=25$. The initial guess is the same random matrix for all four algorithms. 
  For both cases, we see that the simpler Burer--Monteiro approach, including the L-BFGS method and the CG method with metric $g^1$, is significantly slower.  
  
  In the third test of Figure \ref{fig:gradOverSigmap_matrixCompletion}, we show that the ratio term $\frac{\norm{\nabla f(Y_k Y_k^*)} }{ ({\sigma_p})_k }$ in Assumption 
 \ref{assm:gradient_vanish_speed} versus the iteration number $k$ does not blow up as $\pi(Y_k)$ converges to $\pi(\hat{Y})$.

\begin{figure}[htpb]
    \centering
    \subfigure[The algorithms are solved on the rank 30 manifold]{
    \includegraphics[scale =0.3]{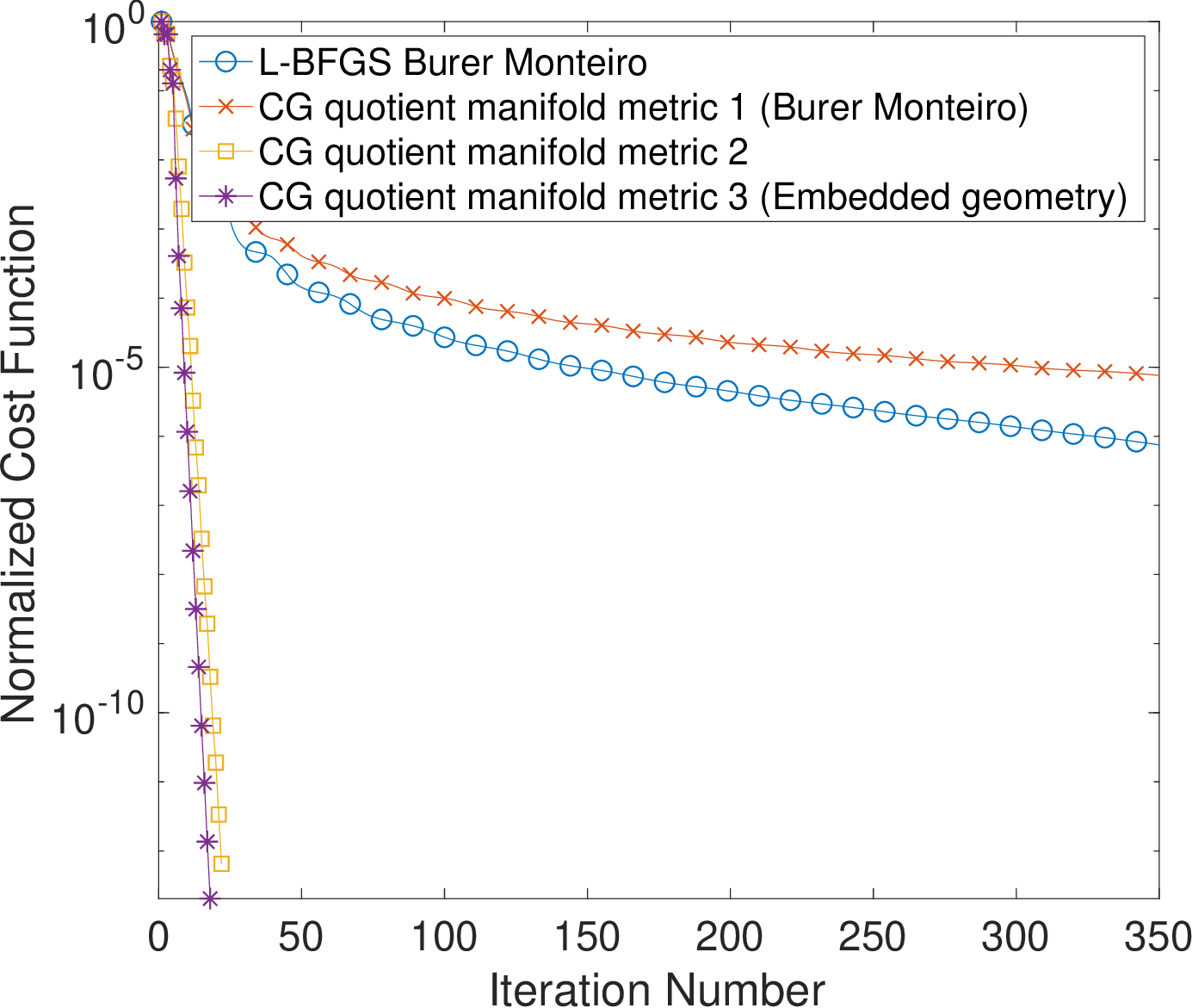}\label{fig:matrixcompletion-a}
    }
    \subfigure[The algorithms are solved on the rank 25 manifold]{
    \includegraphics[scale =0.255]{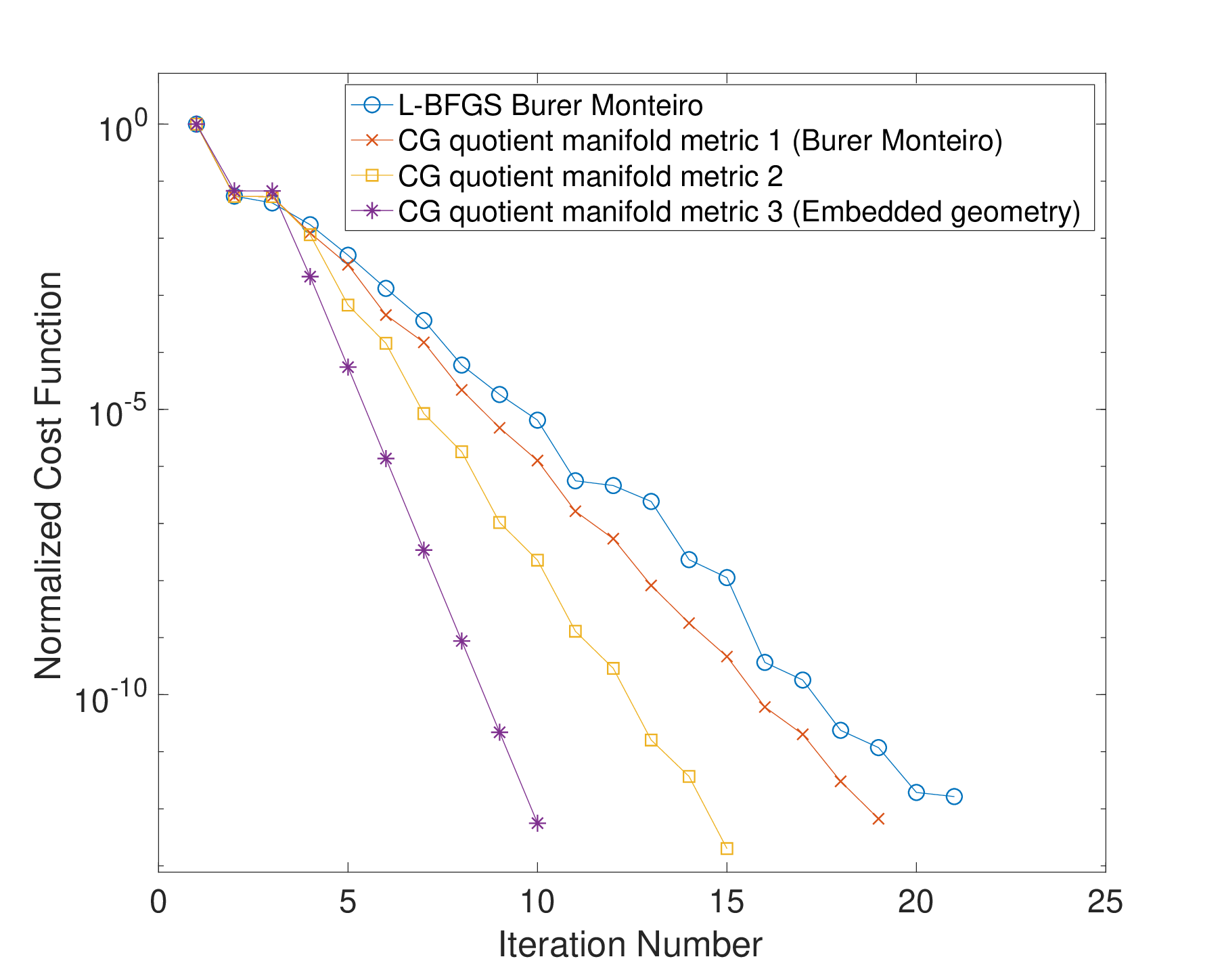}\label{fig:matrixcompletion-b}
    }
    \caption{Matrix completion of a random 10\,000-by-10\,000  PSD matrix of rank 25 observed at random 90\% entries. A comparison of decrease in normalized cost function value $\frac{\norm{P_\Omega(Y_kY_k^* - A)}_F}{\norm{P_\Omega(A)}_F}$  versus iteration number $k$ when using L-BFGS approach and CG method with metric $g^i,i=1,2,3$. When the minimizer is rank deficient (the case in (a)), L-BFGS approach and CG method with metric $g^1$ is significantly slower. }
    \label{fig:matrixcompletion-all}
\end{figure}

\begin{figure}[htbp!]
\centering
\subfigure[L-BFGS Burer Monteiro]{
\includegraphics[width=0.23\textwidth]{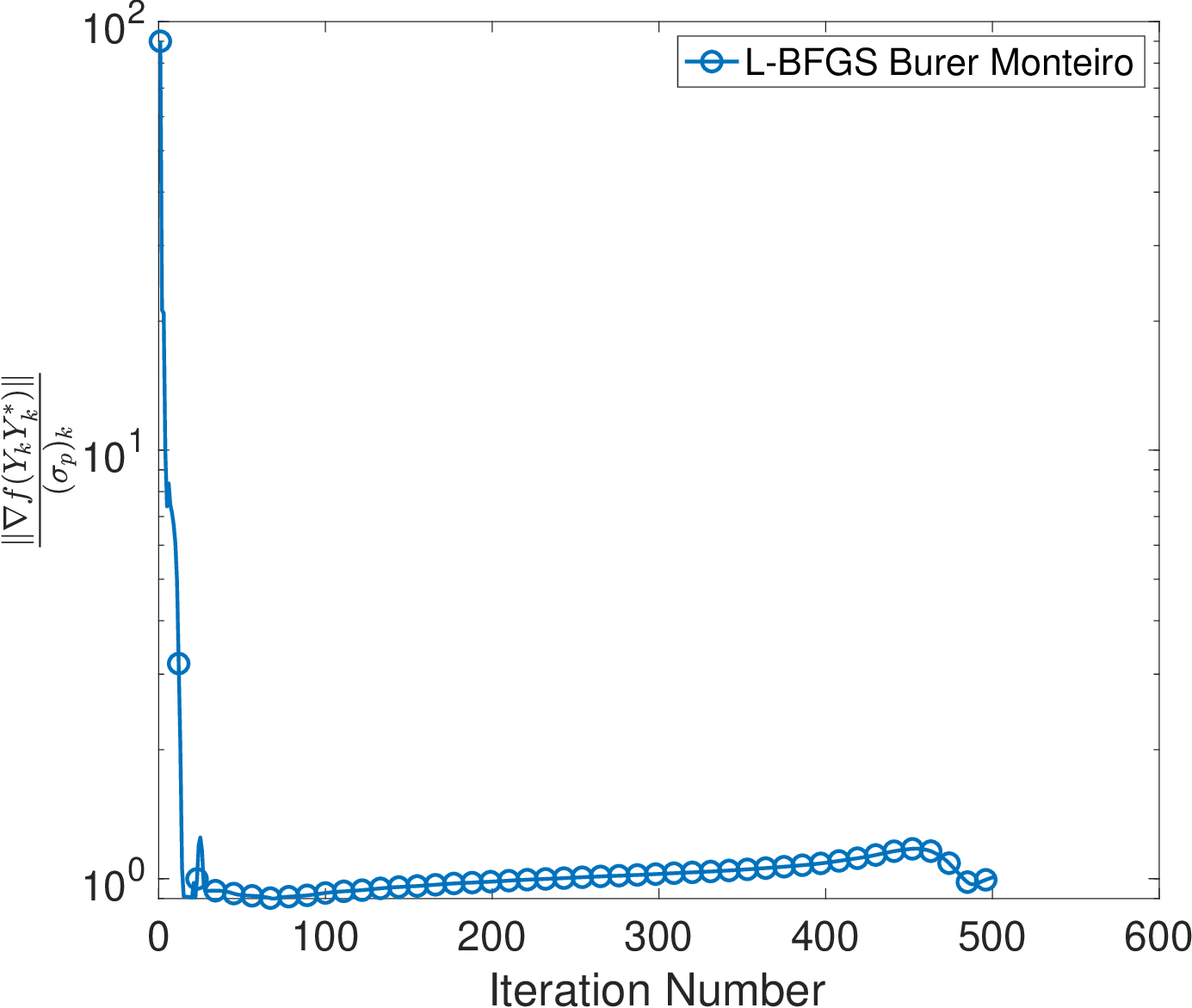}
}
\subfigure[CG quotient manifold metric 1 (Burer Monteiro)]{
\includegraphics[width=0.23\textwidth]{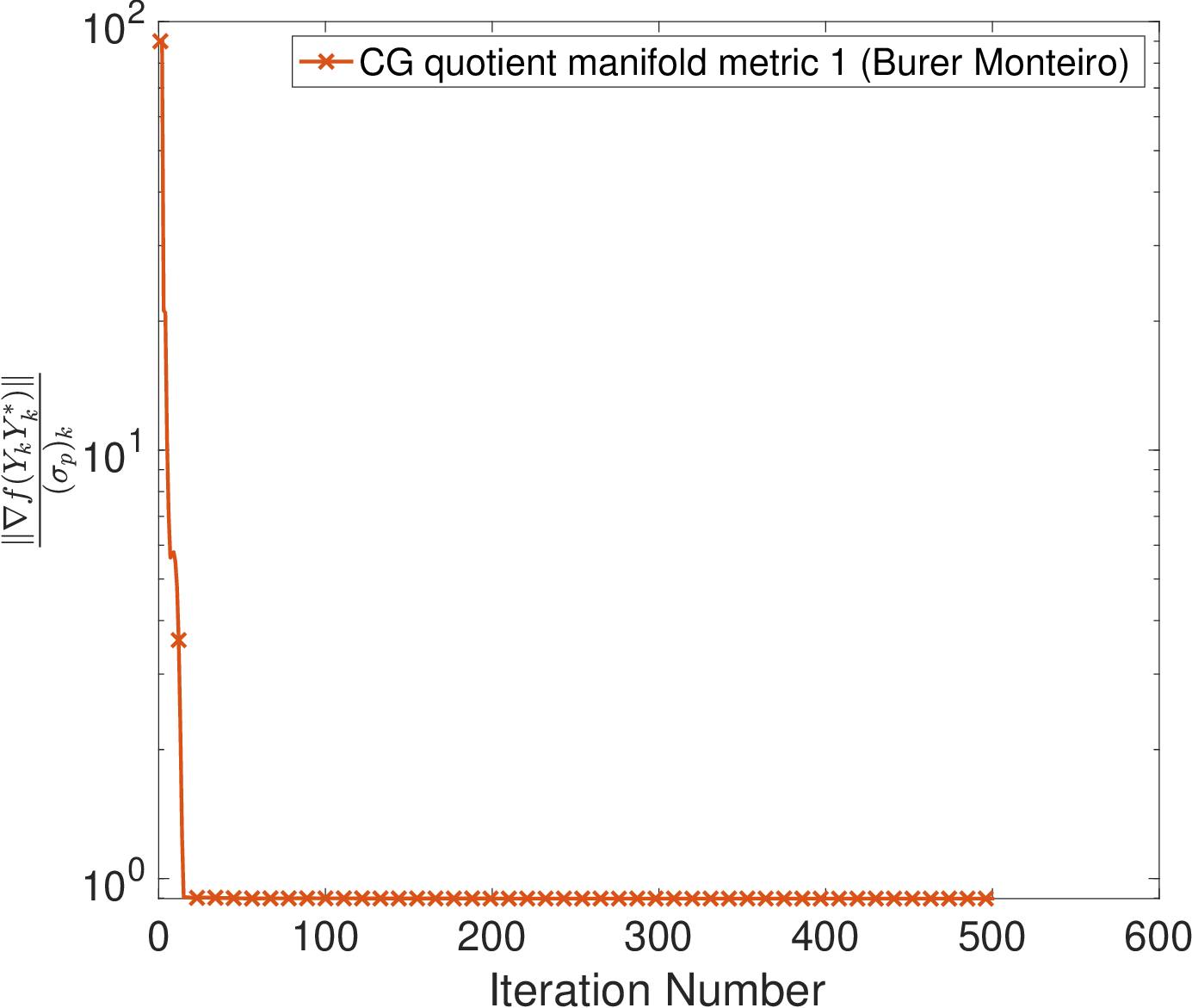}
}
\subfigure[CG quotient manifold metric 2]{
\includegraphics[width=0.23\textwidth]{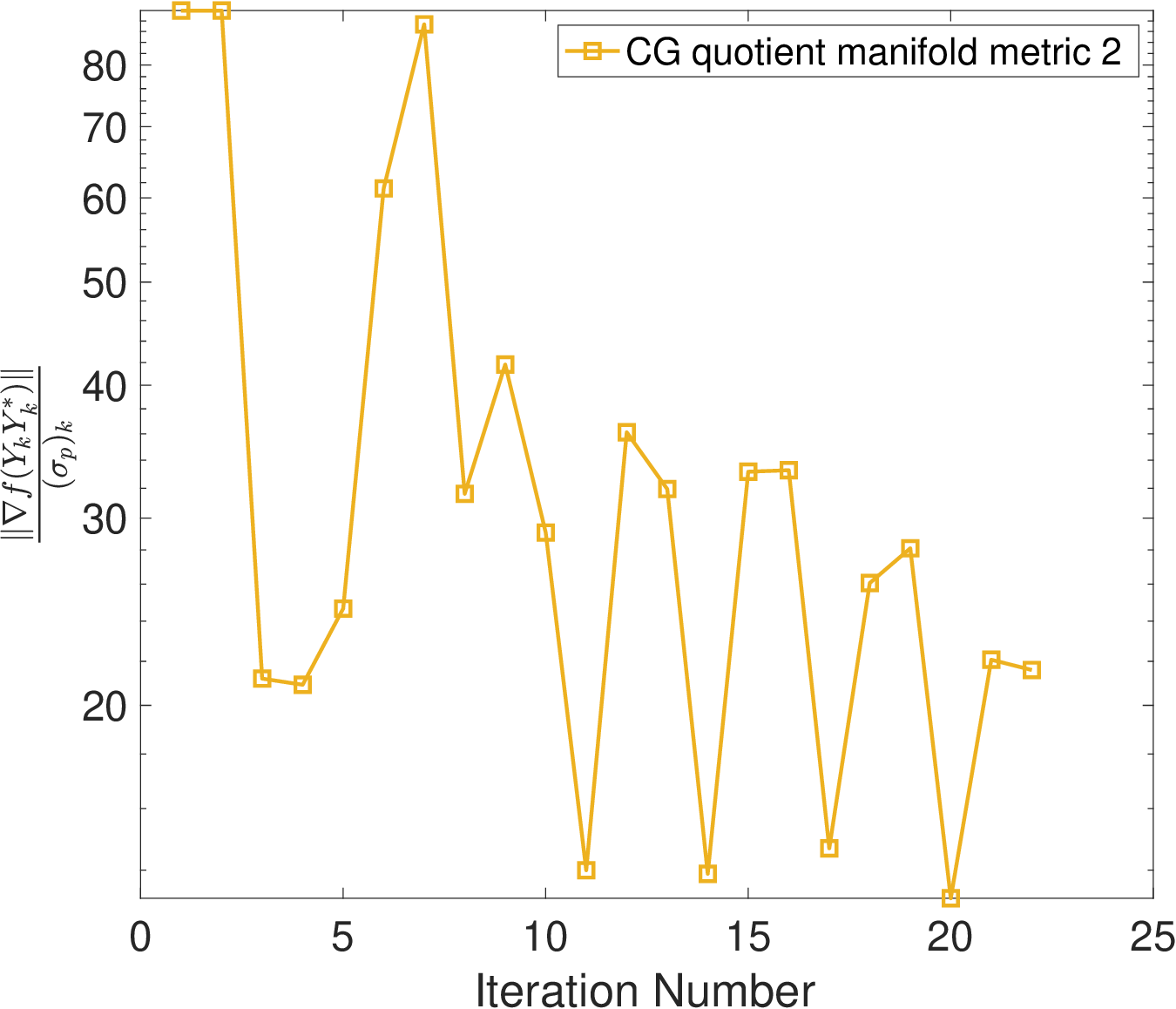}
}
\subfigure[CG quotient manifold metric 3 (Embedded geometry)]{
\includegraphics[width=0.23\textwidth]{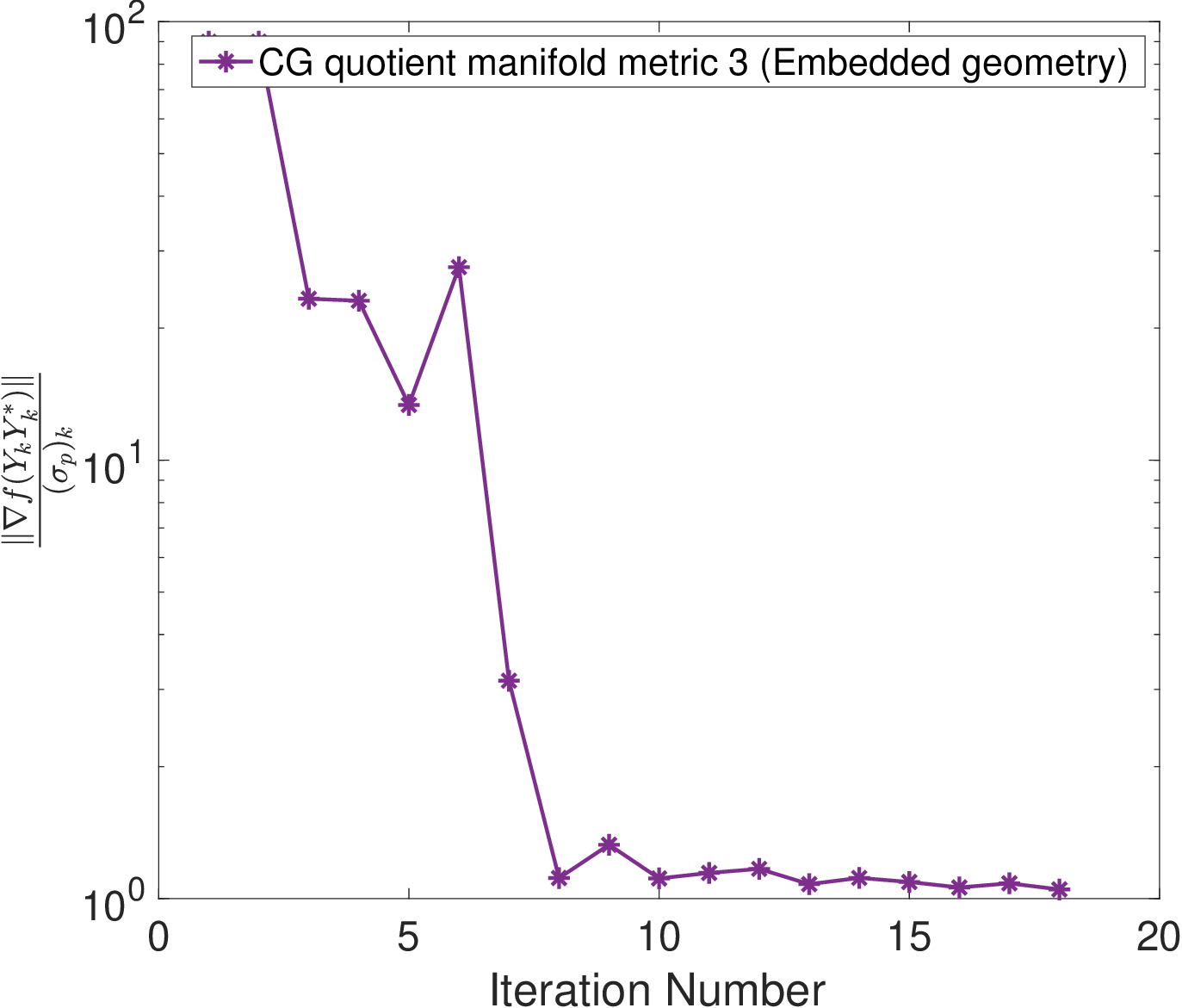}
}
\caption{Numerical justification of Assumption 
 \ref{assm:gradient_vanish_speed} for the matrix completion problem of a random 10\,000-by-10\,000 PSD matrix of rank 25 observed at random 90\% entries solved on the rank 30 manifold (same setup as the numerical test shown in Fig \ref{fig:matrixcompletion-a}). Plots show the ratio term $\frac{\norm{\nabla f(Y_k Y_k^*)} }{ ({\sigma_p})_k }$ in the Assumption 
 \ref{assm:gradient_vanish_speed} versus the iteration number $k$ for L-BFGS approach and CG method with metric $g^i,i=1,2,3$.}
\label{fig:gradOverSigmap_matrixCompletion}
\end{figure}

\subsection{The PhaseLift problem}
We now solve the phase retrieval problem as described in \cite{candes2013phaselift}:  Take an image $x \in \mathbb{C}^{N^2 \times 1} $ and a collection of masks denoted by $\{M_i\}_{i=1}^m$ where $N^2 = n$ is the size of the flattened image. Each $M_i$ is of the same size as $x$ and the elements in each $M_i$ are real or complex numbers with both real and imaginary parts between 0 and 1. We can choose $M_i$ to be random numbers or i.i.d. Gaussian. We have $m$ number of observations for each mask $i = 1,\cdots,m$: 
\begin{equation}\label{eqn:1}
d^i = \mathcal{N}(x) := |(\text{DFT}(\text{Diag}(M_i) * x) |^2,
\end{equation}
where $\mathcal{N}$ denotes the nonlinear operator. The squared power is taken element-wisely. $\text{Diag}(M_i)$ denotes the diagonal matrix whose diagonal is $M_i$. \text{DFT} denotes the $n \times n$ discrete fourier transform matrix. Therefore, $d^i$ is vector of size $  n \times 1$. 

Now we   lift $x$ so that $\mathcal{N}$ can be treated as a linear operator. Let $d_j^i $ denote the $j$th component of $d^i$. Let ${z^i}^* $ denote $\text{DFT}\cdot \text{Diag}(M_i)$ and ${z_j^i}^*$ denote the $j$th row of $\text{DFT} \cdot \text{Diag}(M_i)$.  Then equation (\ref{eqn:1}) can be written as

\[
	|\langle z_j^i , x \rangle|^2  = {z_j^i}^* x x^* z_j^i = d_j^i, \quad j = 1, \ldots n, \quad i = 1,\ldots, m.
\]

Denoting $X := xx^*$, the nonlinear operator $\mathcal{N}$ can be rewritten as the linear operator
\[
	\mathcal{A} : \mathbb{C}^{n \times n} \rightarrow \mathbb{R}^{ m n \times 1}, \quad X \mapsto [  tr(z_1^1{z_1^1}^* X), \cdots, tr(z_{n}^1{z_{n}^1}^* X),\cdots,tr(z_1^m{z_1^m}^* X), \cdots, tr(z_{n}^m{z_{n}^m}^* X) ]^T.
\]
Let $Z^i:= \text{DFT}\cdot\Diag(M_i)= \bmat{-{z^i_1}^*-\\\cdots\\-{z^i_n}^*-}$, then we have alternatively 
\[
    \mathcal{A} : \mathbb{C}^{n \times n} \rightarrow \mathbb{R}^{ m n \times 1}, \quad X \mapsto  [\diag(Z^1X{Z^1}^*),\cdots,\diag(Z^mX{Z^m}^*) ]^T.
\]
Denote $b = [d^1,\cdots, d^m]^T$ . Then the cost function can be written as  
\begin{equation*}
	f(X) = \frac{1}{2}\norm{\mathcal{A}(X) - b}^2
\end{equation*}

The conjugate of operator $\mathcal{A}$, detoted by $\mathcal{A}^*$ is given by 
\begin{equation*}
    A^*(b) = \left\{ \begin{split}
        \sum_{i=1}^m\sum_{j=1}^n b_j^iz_j^i{z_j^i}^*= \sum_{i=1}^m {Z^i}^*\Diag(b^i)Z^i,\quad  \text{if domain of } \mathcal{A} \text{ is } \mathbb{C}^{n\times n} \\
        \Re\left(\sum_{i=1}^m\sum_{j=1}^n b_j^iz_j^i{z_j^i}^* \right)= \Re\left(\sum_{i=1}^m {Z^i}^*\Diag(b^i)Z^i \right), \quad \text{if domain of } \mathcal{A} \text{ is } \mathbb{R}^{n\times n} . 
    \end{split}\right.
\end{equation*}

Straightforward calculation shows
\begin{equation*}
    \nabla f(X) = \mathcal{A}^*(\mathcal{A}(X)-b),\quad    \nabla^2 f(X)[\zeta_X] = \mathcal{A}^*(\mathcal{A}(\zeta_X)) \quad \text{for all $\zeta_X \in \mathbb{C}^{n\times n}$}.
\end{equation*}

For the numerical experiments, we take the phase retrieval problem for a complex gold ball image of size $256 \times 256$ as in \cite{huang_solving_2017}. Thus $n=256^2=65,536$  in \eqref{min-psd} or \eqref{lowrank_prob}. 
We consider the operator $\mathcal A$ that corresponds to $6$ Gaussian random masks. Hence, the size of $b$ is $6n=393,216.$ Remark that problem is easier to solve with more masks.

We first test the algorithms on the rank 3 manifold, and  then on the rank 1 manifolds. The results are visible in Figure \ref{fig:pr-all}. The initial guess is randomly generated. 
First, we observe that solving the PhaseLift problem on the rank $p$ manifold with $p>1$ can accelerate the convergence, compared to solving it on the rank 1 manifold.  
Second, when $p=r=1$, the asymptotic convergence rates of all algorithms are essentially the same, though the algorithms differ in the length of their convergence "plateaus".
When $p=3>r=1$, we can see that the Burer--Monteiro approach has slower asymptotic convergence rates. 
 
In the second test of Figure \ref{fig:gradOverSigmap_phaseRetrieval}, we show that the ratio term $\frac{\norm{\nabla f(Y_k Y_k^*)} }{ ({\sigma_p})_k }$ in Assumption 
 \ref{assm:gradient_vanish_speed} versus the iteration number $k$ does not blow up as $\pi(Y_k)$ converges to $\pi(\hat{Y})$.

\begin{figure}[htbp!]
    \centering
    \subfigure[The algorithms are solved on the rank 3 manifold]{
    \includegraphics[scale =0.25]{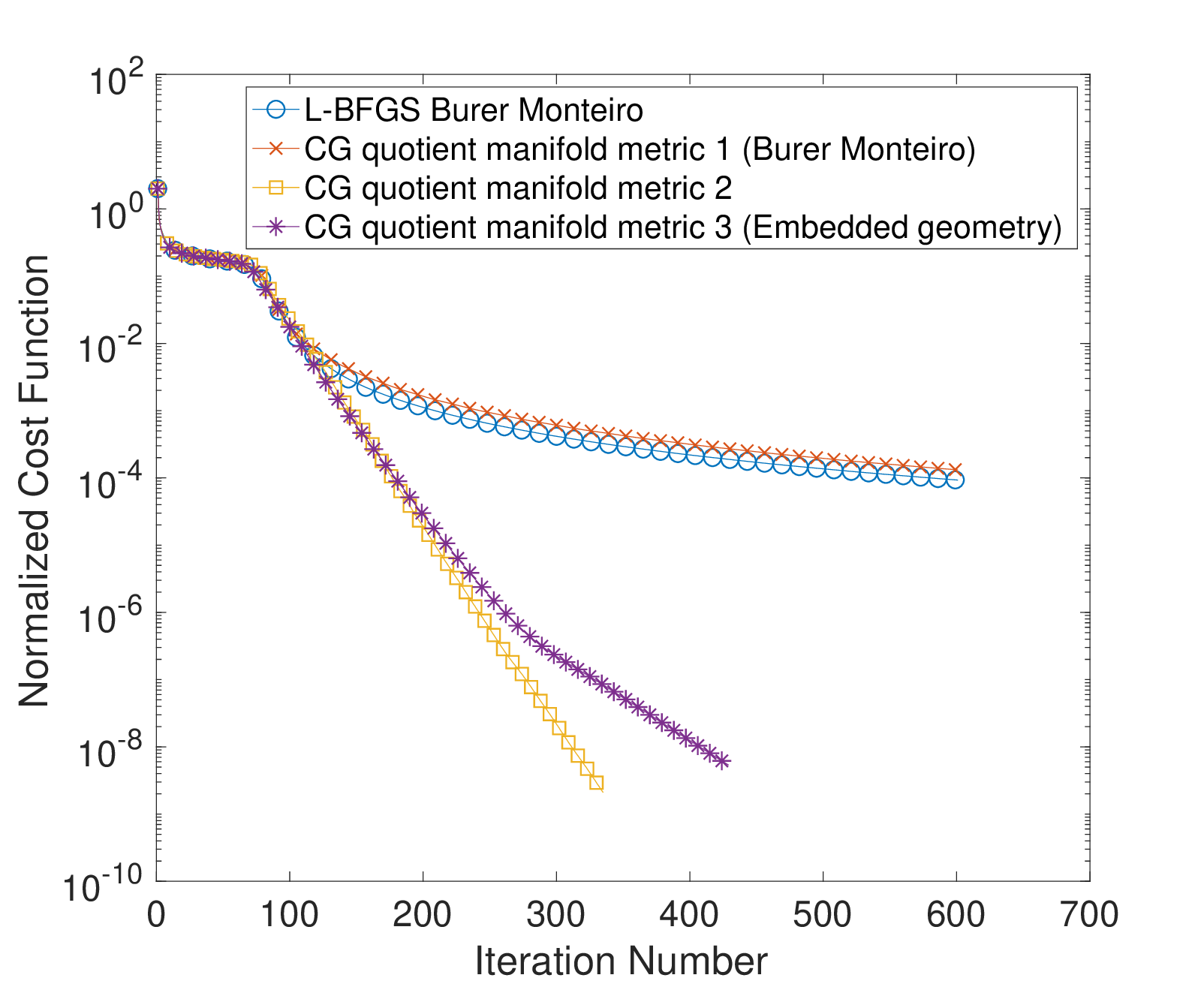}\label{fig:pr}
    }
    \subfigure[The algorithms are solved on the rank 1 manifold]{
    \includegraphics[scale =0.255]{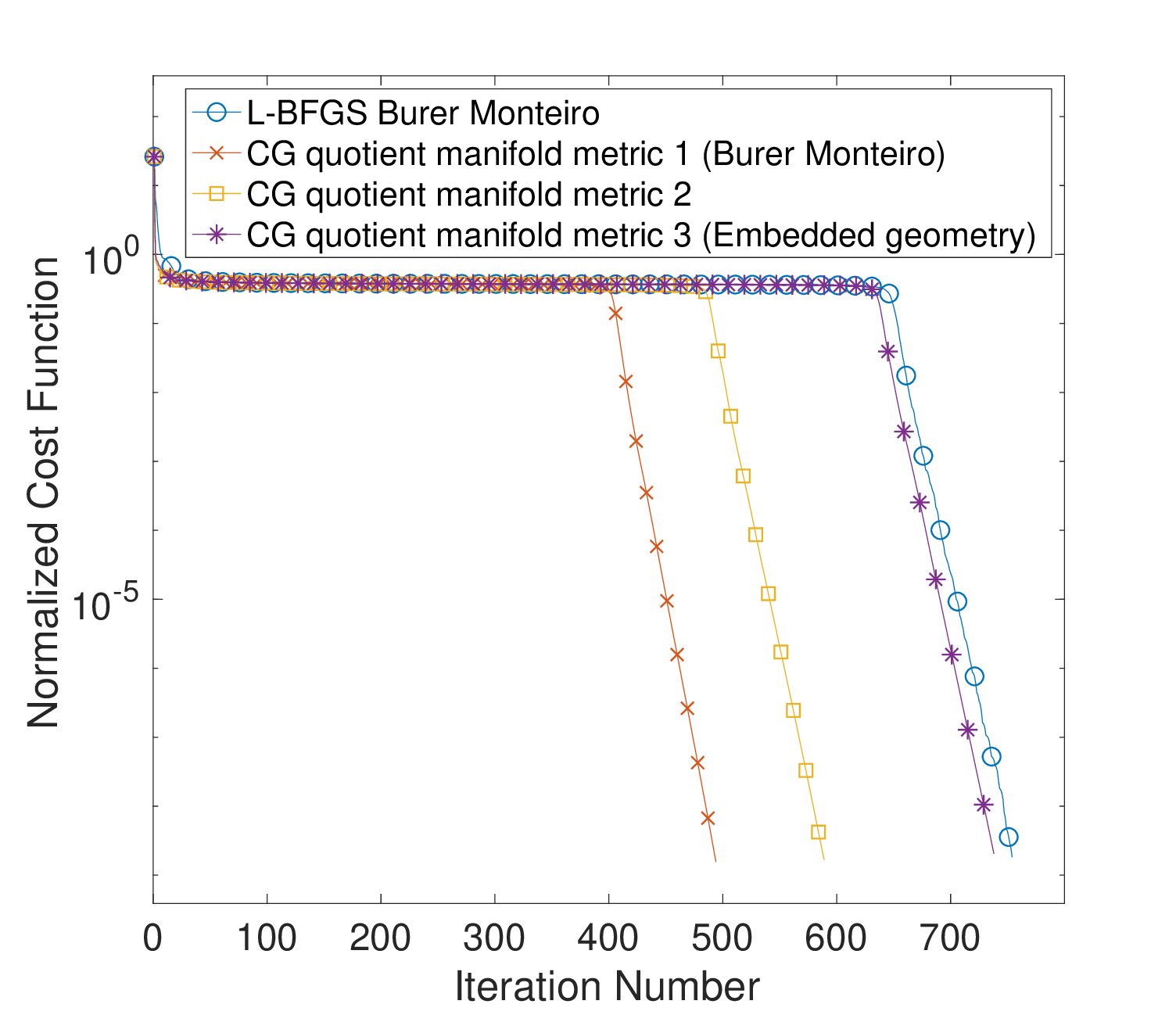}\label{fig:pr2}
    }
    \caption{Phase retrieval of a 256-by-256 image with 6 Gaussian masks. A comparison of normalized cost function value $\frac{\norm{\mathcal{A}(Y_k Y_k^*) - b}}{\norm{b}}$ versus iteration number $k$ when using L-BFGS approach and CG method with metric $g^i, i=1,2,3$.  When the minimizer is rank deficient (the case in \ref{fig:pr}), L-BFGS approach and CG method with metric $g^1$ is significantly slower. }
    \label{fig:pr-all}
\end{figure}

\begin{figure}[H]
\centering
\subfigure[L-BFGS Burer Monteiro]{
\includegraphics[width=0.23\textwidth]{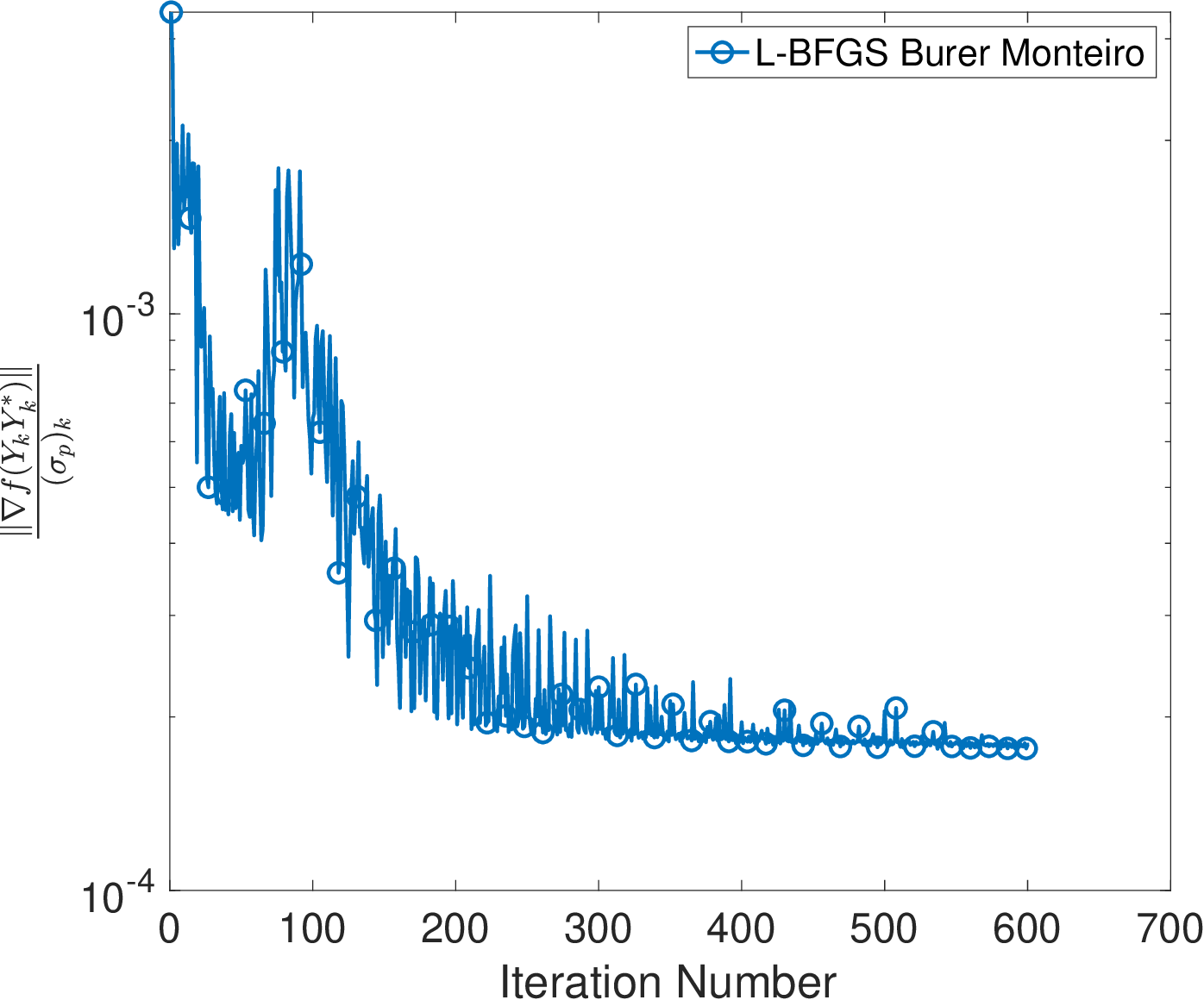}
}
\subfigure[CG quotient manifold metric 1 (Burer Monteiro)]{
\includegraphics[width=0.23\textwidth]{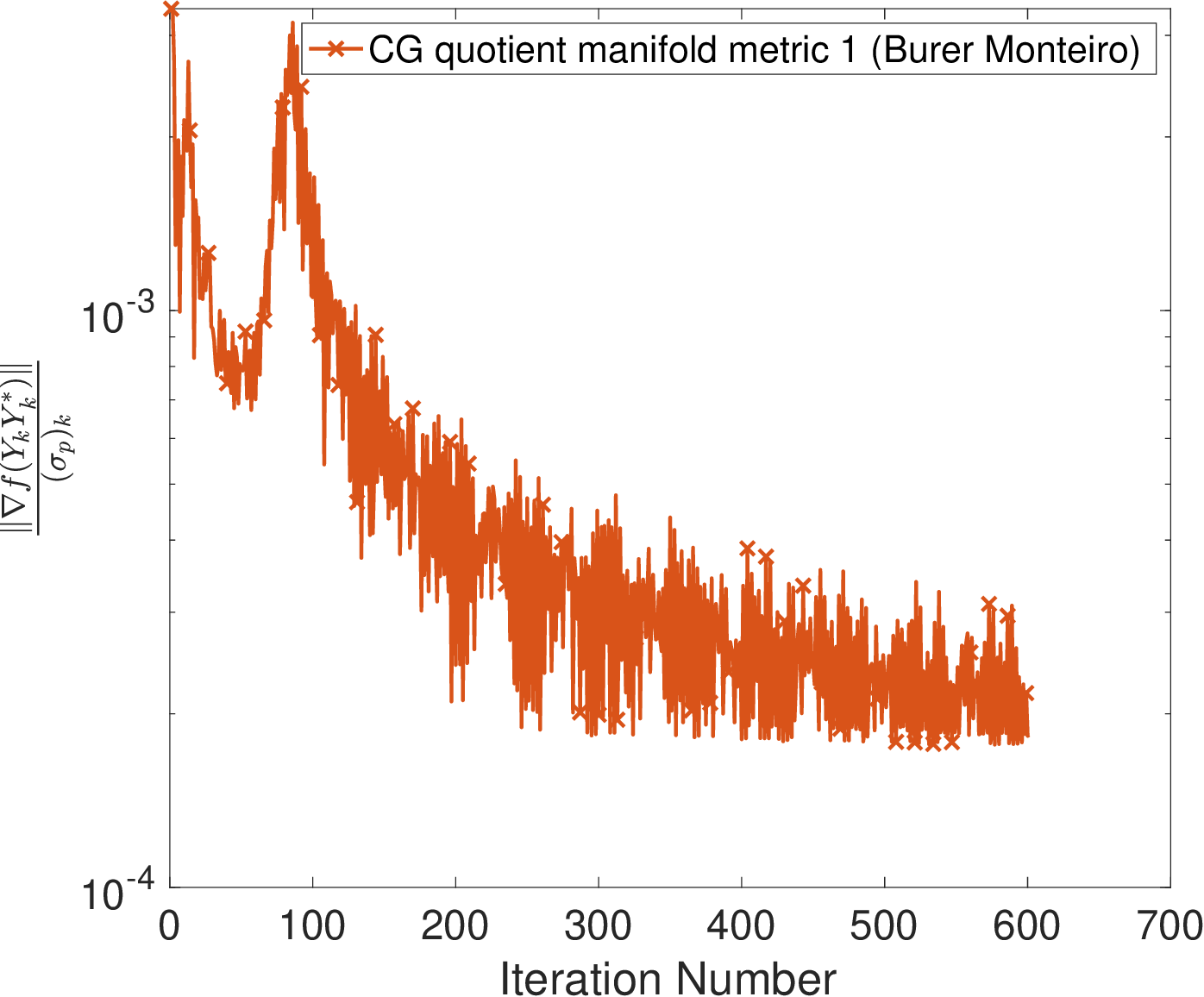}
}
\subfigure[CG quotient manifold metric 2]{
\includegraphics[width=0.23\textwidth]{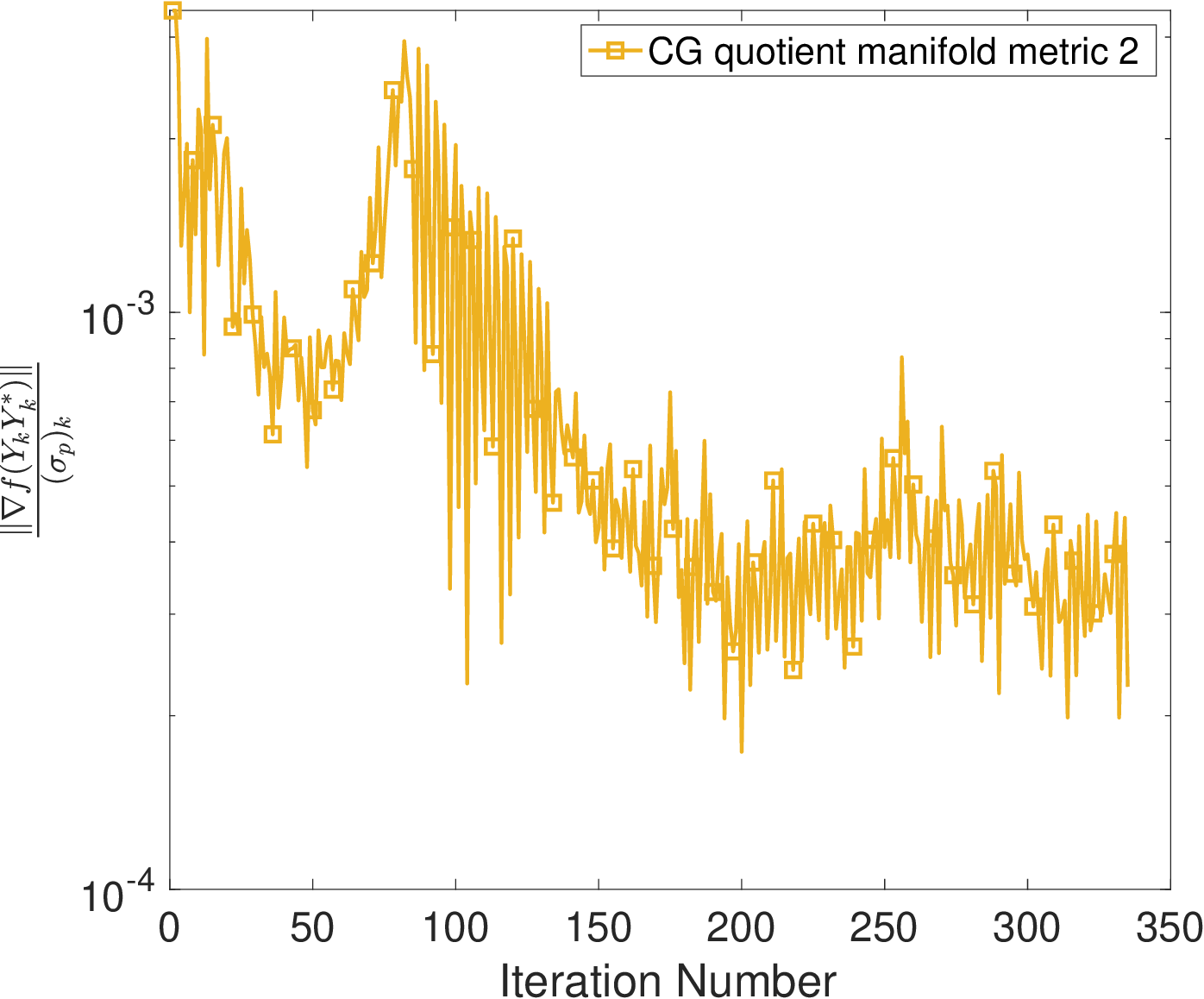}
}
\subfigure[CG quotient manifold metric 3 (Embedded geometry)]{
\includegraphics[width=0.23\textwidth]{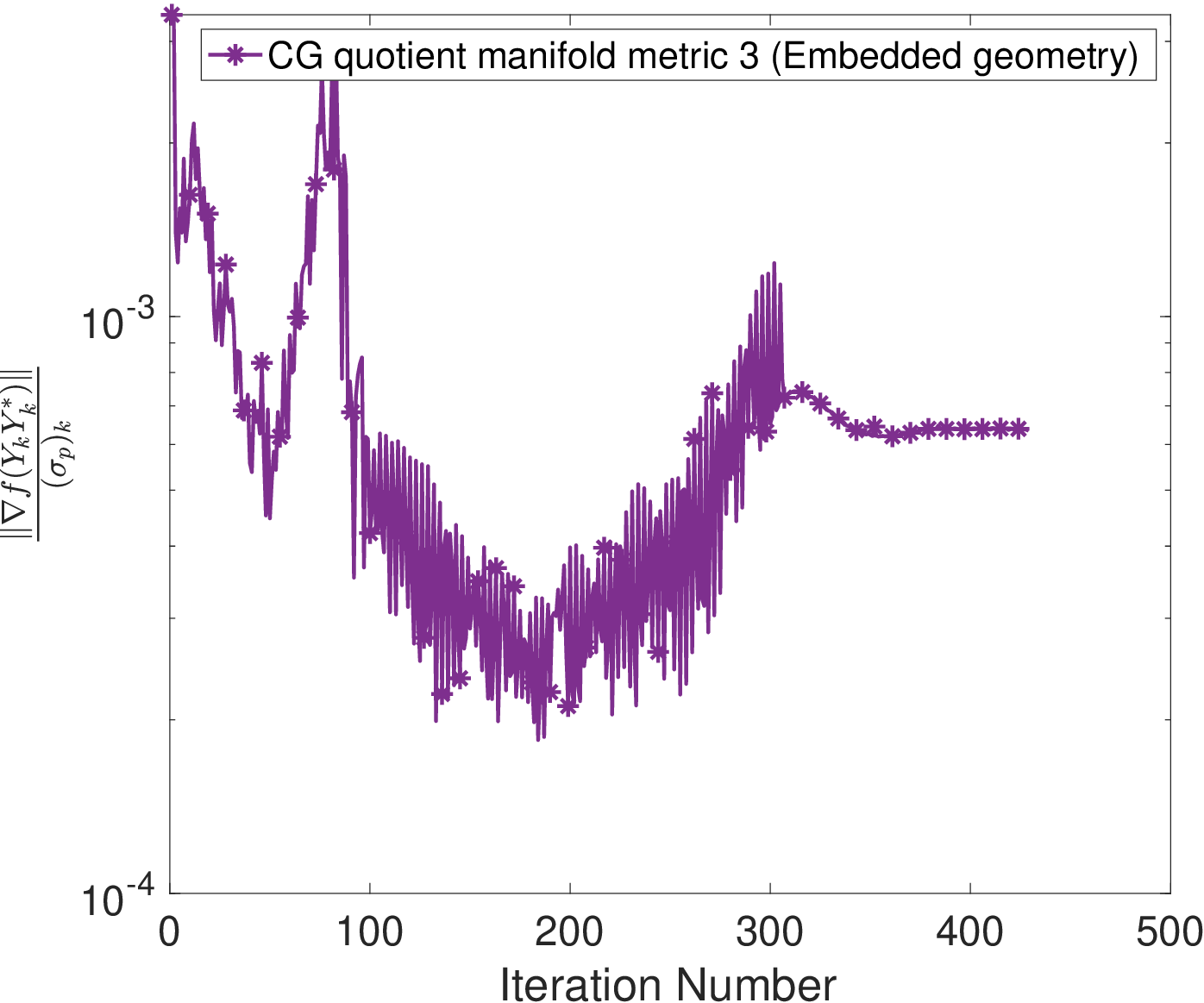}
}
\caption{Numerical justification of Assumption 
 \ref{assm:gradient_vanish_speed} for the phase retrieval problem of a 256-by-256 image with 6 Gaussian masks solved on the rank 3 manifold (same setup as the numerical test shown in Fig \ref{fig:pr}). Plots show the ratio term $\frac{\norm{\nabla f(Y_k Y_k^*)} }{ ({\sigma_p})_k }$ in the Assumption 
 \ref{assm:gradient_vanish_speed} versus the iteration number $k$ for L-BFGS approach and CG method with metric $g^i,i=1,2,3$.}
\label{fig:gradOverSigmap_phaseRetrieval}
\end{figure}

\begin{figure}[H]
    \centering
    \subfigure[The algorithms are solved on the rank 3 manifold]{
    \includegraphics[scale =0.25]{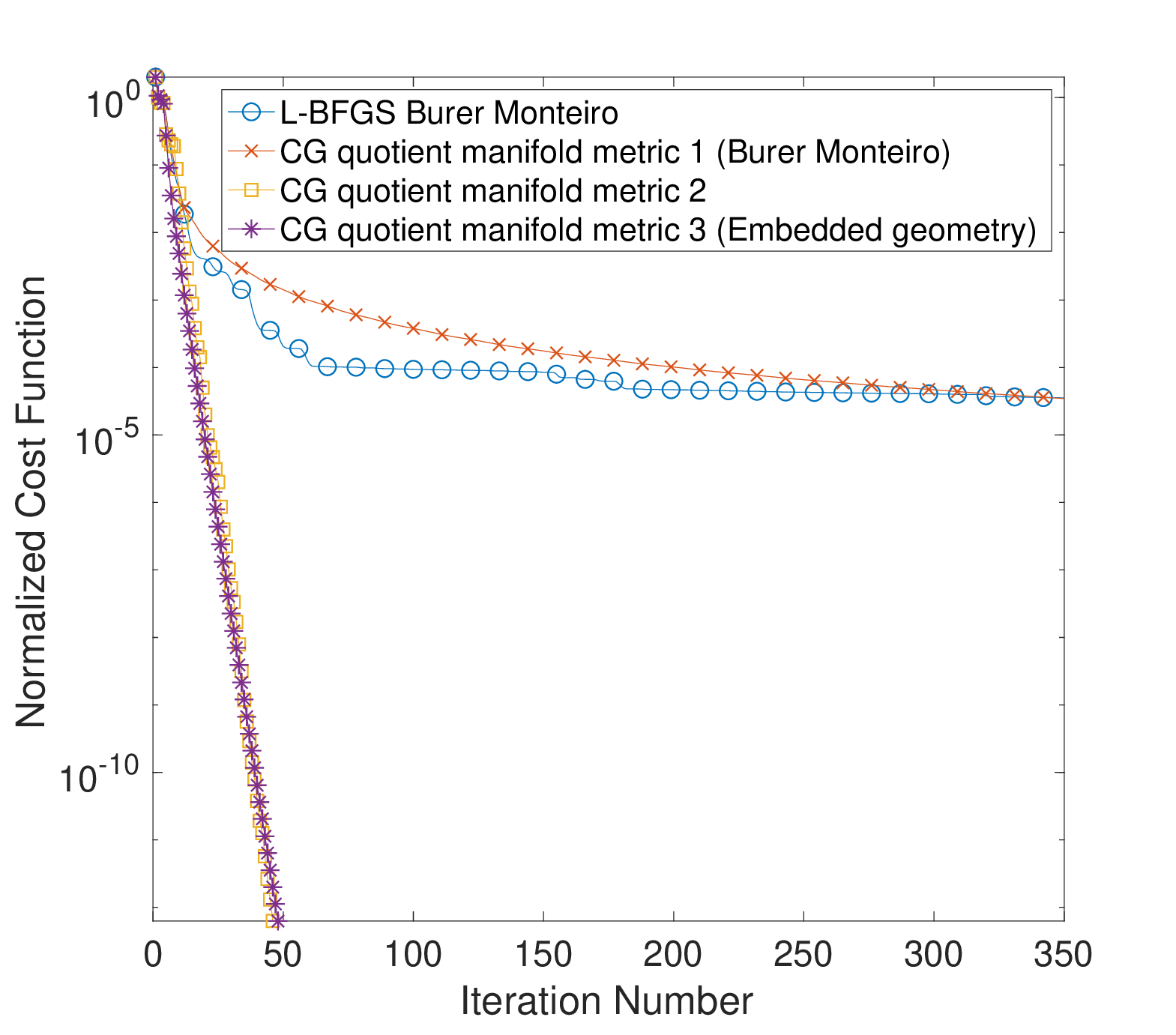}\label{fig:inter}
    }
    \subfigure[The algorithms are solved on the rank 1 manifold]{
    \includegraphics[scale =0.255]{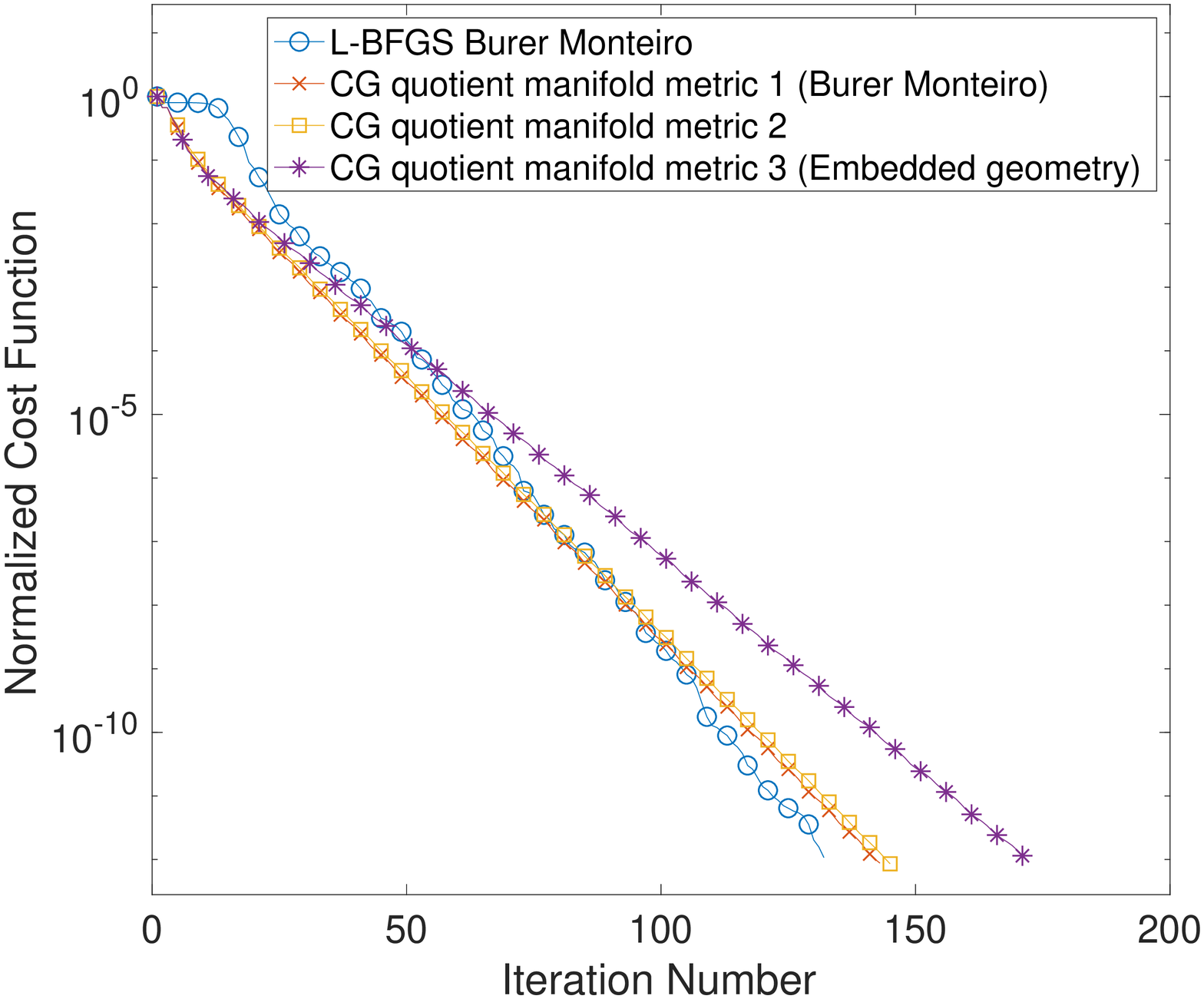}\label{fig:inter2}
    }
    \caption{Interferometry  recovery of a random 10\,000-by-1000 $F$ with 1\% sampling.  A comparison of normalized cost function value $\frac{\norm{P_\Omega(FY_kY_k^*F^* - dd^*)}_F}{\norm{P_\Omega(dd^*)}_F}$ versus iteration number $k$ when using L-BFGS approach and CG method with metric $g^i, i=1,2,3$.  When the minimizer is rank deficient (the case in (a)), L-BFGS approach and CG method with metric $g^1$ is significantly slower. }
    \label{fig:inter-all}
\end{figure}

\begin{figure}[H]
\centering
\subfigure[L-BFGS Burer Monteiro]{
\includegraphics[width=0.23\textwidth]{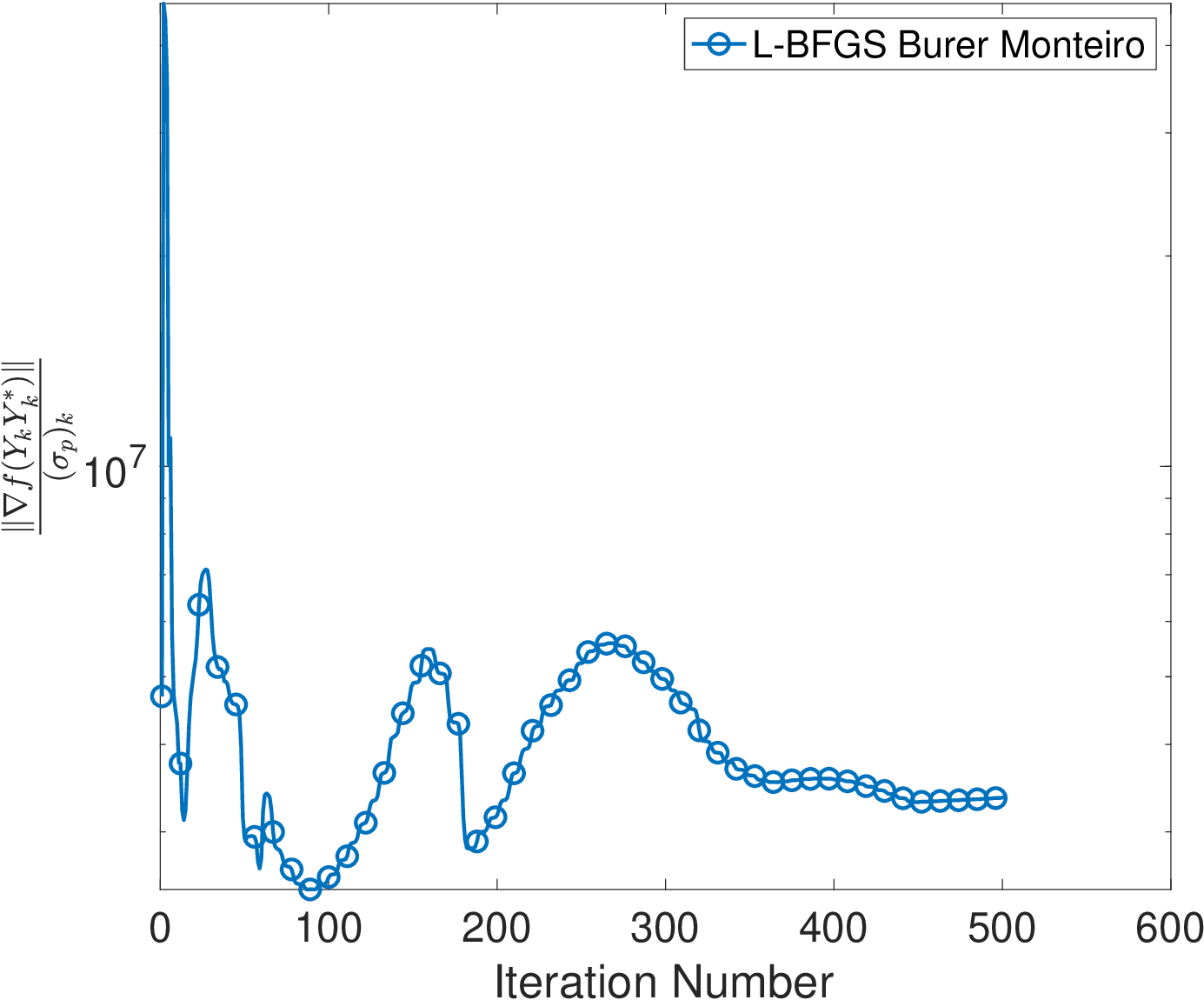}
}
\subfigure[CG quotient manifold metric 1 (Burer Monteiro)]{
\includegraphics[width=0.23\textwidth]{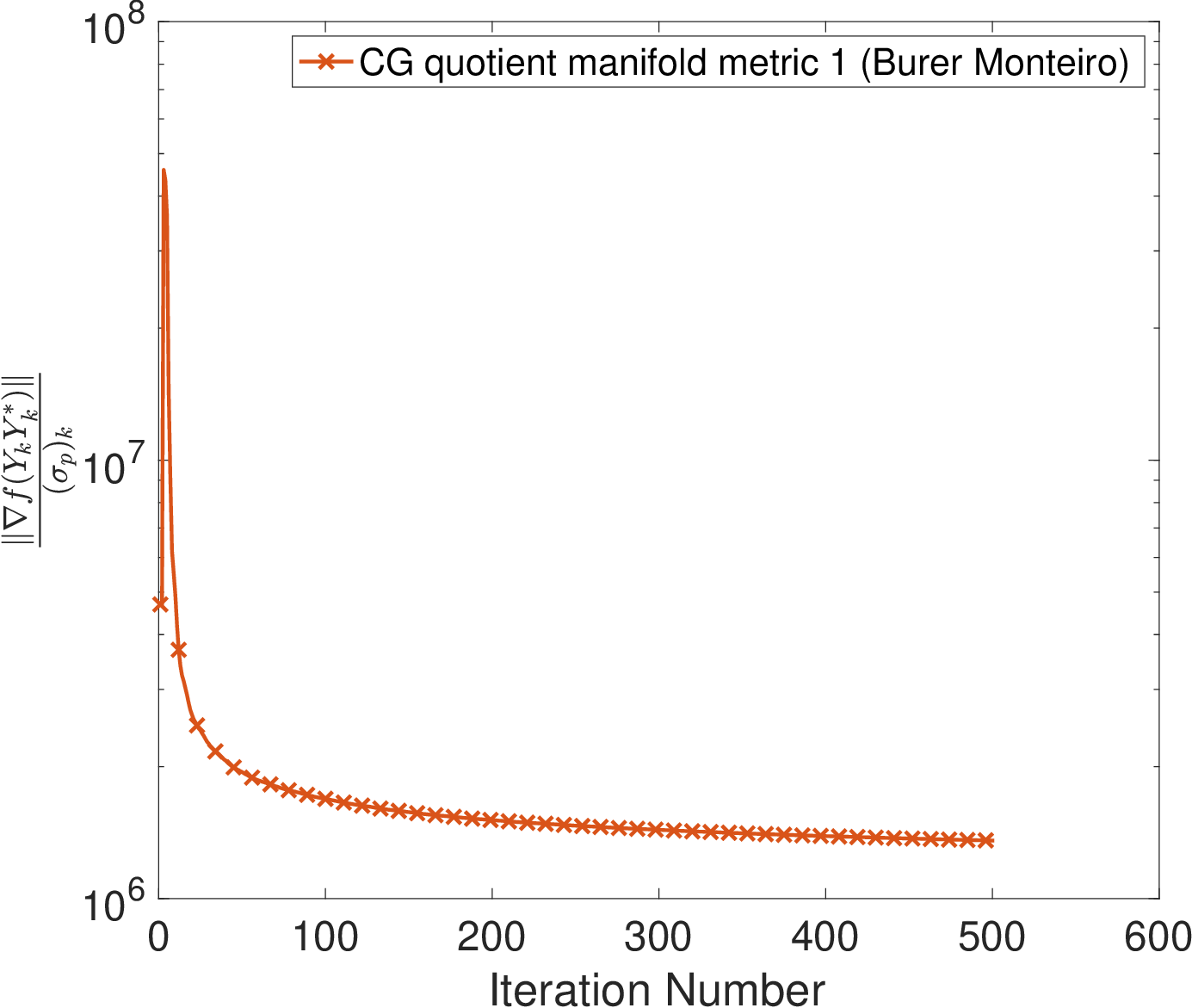}
}
\subfigure[CG quotient manifold metric 2]{
\includegraphics[width=0.23\textwidth]{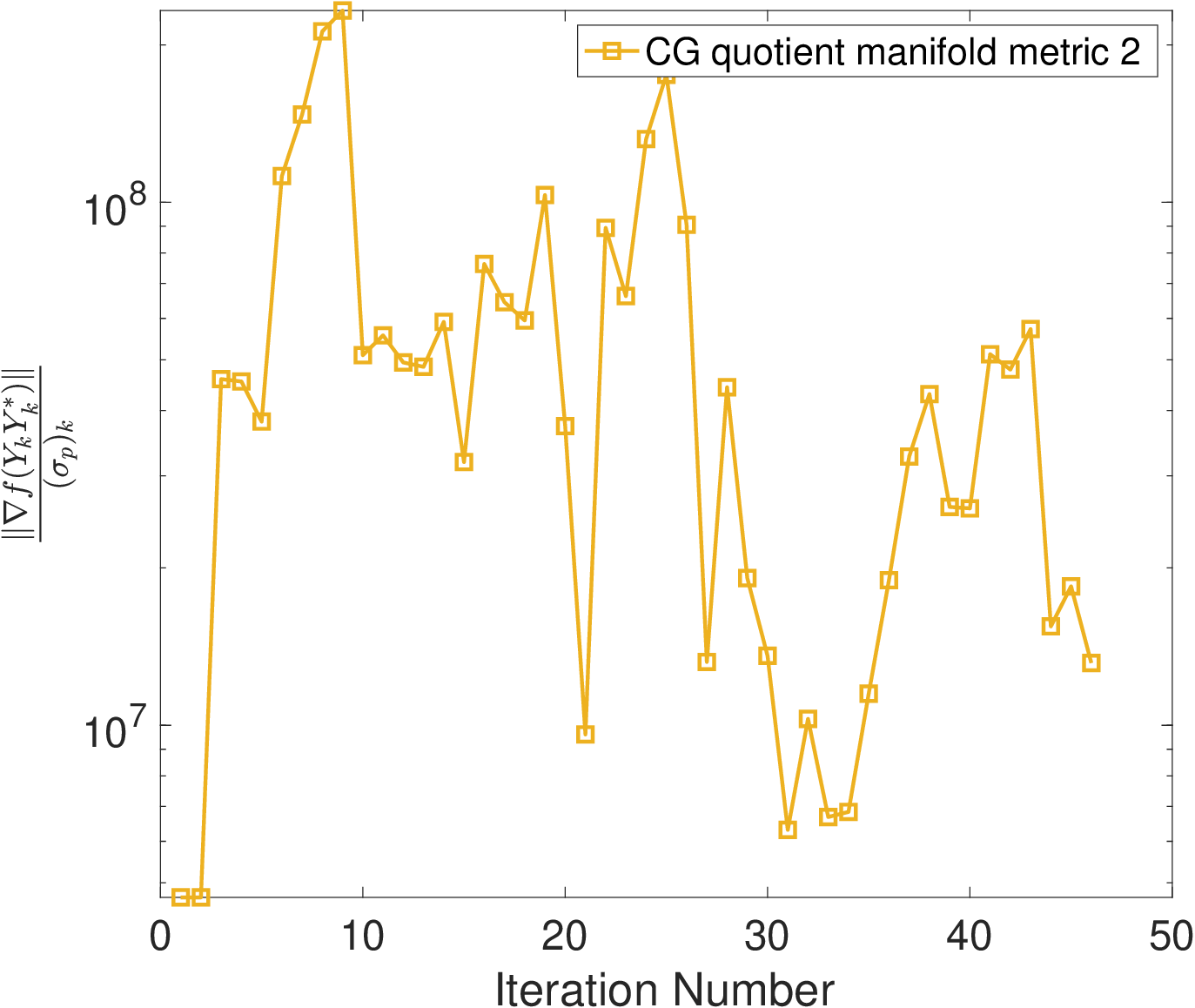}
}
\subfigure[CG quotient manifold metric 3 (Embedded geometry)]{
\includegraphics[width=0.23\textwidth]{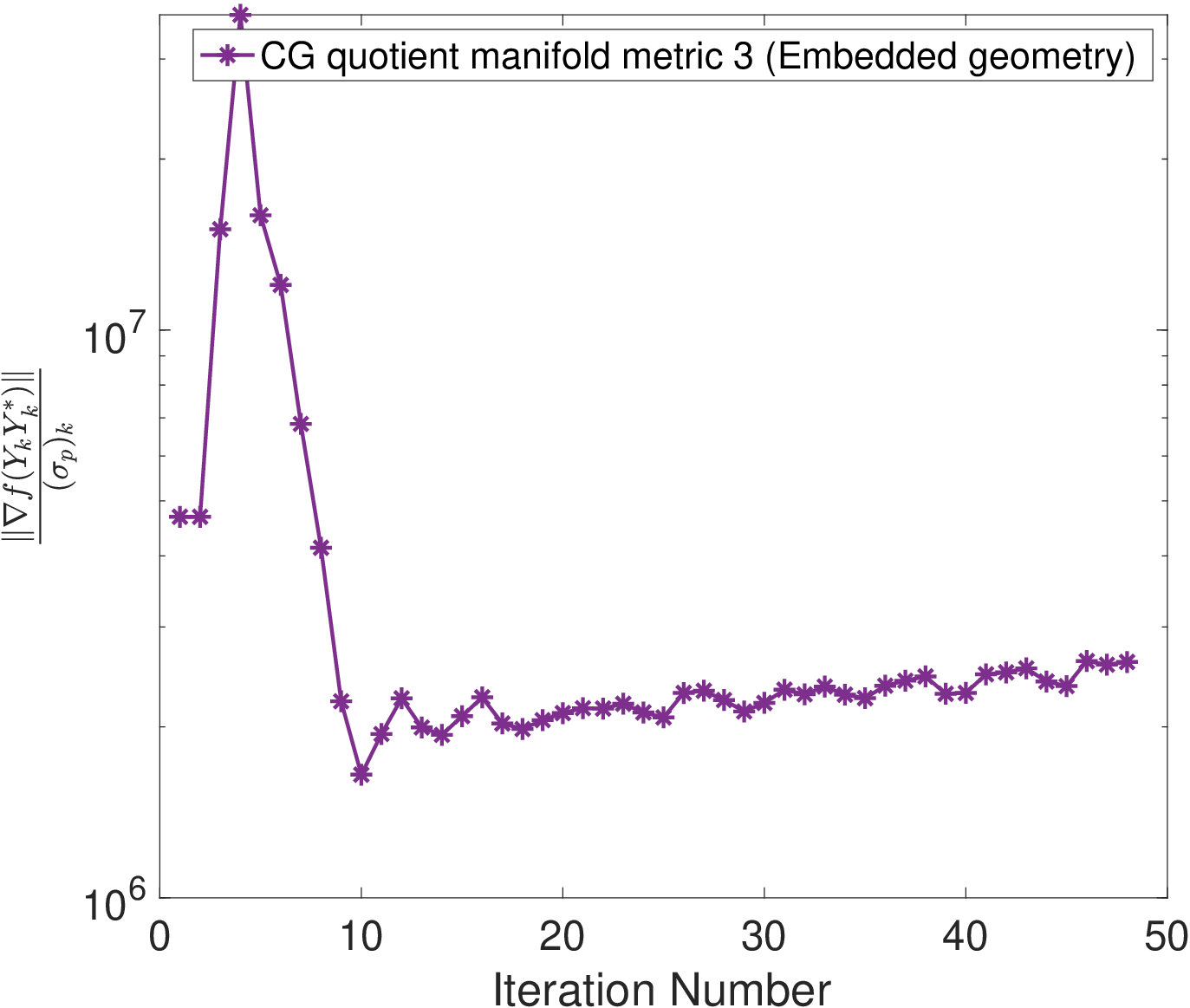}
}
\caption{Numerical justification of Assumption 
 \ref{assm:gradient_vanish_speed} for the interferometry recovery problem of a random 10\,000-by-1000 $F$ with 1\% sampling solved on a rank 3 manifold. (same setup as the numerical test shown in Fig \ref{fig:inter}). Plots show the ratio term $\frac{\norm{\nabla f(Y_k Y_k^*)} }{ ({\sigma_p})_k }$ in the Assumption 
 \ref{assm:gradient_vanish_speed} versus the iteration number $k$ for L-BFGS approach and CG method with metric $g^i,i=1,2,3$.}
\label{fig:gradOverSigmap_interferometry}
\end{figure}

\subsection{Interferometry recovery problem}
As last example, we consider solving the interferometry recovery problem
described in \cite{demanet2017convex}.
Consider solving the linear system $Fx = d$ where $F\in \mathbb{C}^{m\times n}_*$ with $m > n $ and $x\in \mathbb{C}^{n\times 1}$. 
For the sake of robustness, the interferometry recovery \cite{demanet2017convex} requires solving the lifted problem   
\begin{equation*}
    \MINone{X}{f(X) = \frac{1}{2}\norm{P_\Omega(FXF^*- dd^*)}_F^2}{ X \in \mathcal{H}^{n,p}_+ },
\end{equation*}
where
 $\Omega$ is a sparse and symmetric sampling index that includes all of the diagonal.

Straightforward calculation again shows
\begin{equation*}
    \nabla f(X) = F^*P_\Omega(FXF^*-dd^*)F,\quad    \nabla^2f(X)[\zeta_X] = F^*P_\Omega(F\zeta_XF^*)F \quad \text{for all $\zeta_X \in \mathbb{C}^{n\times n}$}.
\end{equation*}

We solve an interferometry problem with a randomly generated $F\in \mathbb C^{10\,000\times 1000}$. Hence $n=1000$  in \eqref{min-psd} or \eqref{lowrank_prob}. The sampling operator $\Omega$ is also randomly generated, with 1\% density. 
In Figure\ref{fig:inter},  $p=3$ and $r=1$ and we can see that the Burer--Monteiro approach has slower asymptotic convergence rates.
In Figure\ref{fig:inter2}, $p=r=1$ and we can see now that all algorithms have more or less the same asymptotic convergence rates.
 
In the second test of Figure \ref{fig:gradOverSigmap_interferometry}, we show that the ratio term $\frac{\norm{\nabla f(Y_k Y_k^*)} }{ ({\sigma_p})_k }$ in Assumption 
 \ref{assm:gradient_vanish_speed} versus the iteration number $k$ does not blow up as $\pi(Y_k)$ converges to $\pi(\hat{Y})$.

\section{Conclusion}

In this paper, we have shown that the nonlinear conjugate gradient method on the unconstrained Burer--Monteiro formulation for Hermitian PSD fixed-rank constraints is equivalent to a Riemannian conjugate gradient method on a quotient manifold
$\mathbb{C}^{n\times p}_*/\mathcal{O}_p$ with  the Bures-Wasserstein metric $g^1$, retraction, and vector transport. 
We have also shown that the Riemannian conjugate gradient method on the embedded geometry of $\mathcal H_+^{n,p}$ is 
equivalent to a Riemannian conjugate gradient method on a quotient manifold
$\mathbb{C}^{n\times p}_*/\mathcal{O}_p$ with a metric $g^3$,
a special retraction, and a special vector transport. 
We have analyzed the condition numbers of the Riemannian Hessians on
$(\mathbb{C}^{n\times p}_*/\mathcal{O}_p, g^i)$ for these metrics $g^1, g^3$ and another metric $g^2$ used in the literature. As a noteworthy result, we have shown that when the rank $p$ of the optimization manifold is larger than the rank of the minimizer to the original PSD constrained minimization, the condition number of the Riemannian Hessian on
$(\mathbb{C}^{n\times p}_*/\mathcal{O}_p, g^1)$ can be unbounded, which is consistent with the observation that the Burer--Monteiro approach often has a slower asymptotic convergence rate in numerical tests.

\newpage 

\begin{appendices}

\section{Derivatives}

\label{appendix-derivative}

See A.5 in \cite{absil_optimization_2008} for more details in this section.
\subsection{Fr\'{e}chet derivatives}\label{frechet_derivatives}
For any two finite-dimensional inner product vector spaces $\mathcal{U}$ and $\mathcal{V}$ over $\mathbb{R}$, a mapping $F: \mathcal{U}\rightarrow 
\mathcal{V}$ is {\it Fr\'{e}chet differentiable} at $x\in \mathcal{U}$ if there exists a linear operator
\[
\begin{array}{lll}
 \mathrm{D} F(x):& \mathcal{U}&\rightarrow \mathcal{V}\\
                 & h&\mapsto \mathrm{D} F(x)[h]
\end{array}
\] 
such that 
\[ F(x+h)=F(x)+\mathrm{D} F(x)[h]+o(\|h\|).\]
The operator $\mathrm{D} F(x)$ is called the {\it Fr\'{e}chet differential} and $\mathrm{D} F(x)[h]$ is called the {\it directional derivative} of $F$ at $x$ along $h$. 
The derivative satisfies the chain rule
\[\mathrm{D}(f\circ g)(x)[h]=\mathrm{D}f(g(x))[\mathrm{D} g(x)[h]].\]
For a smooth real-valued function $f: \mathcal{U}\rightarrow\mathbb{R}$, the {\it Fr\'{e}chet gradient} of $f$ at $x$, denoted by $ \nabla f(x)$, is the unique element in $\mathcal{U}$
satisfying
\begin{equation*}
\label{frechet}
\langle  \nabla f(x), h \rangle_\mathcal{U}=\mathrm{D}f(x)[h],\quad \forall h\in\mathcal{U}, 
\end{equation*}
where $\langle \cdot ,\cdot \rangle_\mathcal{U}$ is the inner product in $\mathcal{U}$.

In particular, regard $\mathcal{U} = \mathbb{C}^{n\times n}$ as a vector space over $\mathbb{R}$ with the standard inner product $\langle X ,Y \rangle_{\mathbb{C}^{n\times n}}
=\Re(\mbox{tr}(X^*Y))$. Regard $X$ as $(\Re(X),\Im(X))$ and regard $f(X)$ as $f(\Re(X),\Im(X))$. By the multivariate Taylor theorem to the function $f(\Re(X),\Im(X))$, we get
\begin{eqnarray*}
&&\left\lvert f(X+h) - f(X) - \langle  \nabla f(X),h \rangle_{\mathbb{C}^{n\times n}} \right\rvert  = \\ 
&&\left\lvert f\left(\Re(X)+ \Re(h), \Im(X) + \Im(h)) - f(\Re(X) , \Im(X)\right) - \left(\left\langle \frac{\partial f}{\partial \Re(X)}, \Re(h) \right\rangle_{\mathbb{R}^{n\times n}}+ \left\langle \frac{\partial f}{\partial \Im(X)} ,\Im(h) \right\rangle_{\mathbb{R}^{n\times n}} \right) \right\rvert\\
&&=
o(\norm{h}_{\mathbb{C}^{n\times n}}).
\end{eqnarray*}
Notice 
\[\left\langle \frac{\partial f}{\partial \Re(X)}, \Re(h) \right\rangle_{\mathbb{R}^{n\times n}}+ \left\langle \frac{\partial f}{\partial \Im(X)} ,\Im(h) \right\rangle_{\mathbb{R}^{n\times n}}=  \left\langle \frac{\partial f(X)}{\partial \Re_X}+\ci\frac{\partial f(X)}{\partial \Im_X}, h\right\rangle_{\mathbb{C}^{n\times n}}. \]

Thus the expression 
\begin{equation*}
    \nabla f(X) = \frac{\partial f(X)}{\partial \Re_X}+\ci\frac{\partial f(X)}{\partial \Im_X}
\end{equation*}
coincides with the Fr\'{e}chet gradient for $f(X)$ under the real inner product \eqref{real-inner-product}.

\begin{prop}\label{prop:Frechet_gradient_hermitian}
Regard $\mathcal{U} = \mathbb{C}^{n\times n}$ as a vector space over $\mathbb{R}$ with the standard inner product $\langle X ,Y \rangle_{\mathbb{C}^{n\times n}}
=\Re(\mbox{tr}(X^*Y))$ $\mathbb{C}^{n\times n}$.  Let $X \in \mathbb{C}^{n\times n}$.  If $X=X^*$, then $\nabla f(X) = (\nabla f(X))^*$. 
\end{prop}
\begin{proof}
Let $g: \mathbb{C}^{n\times n} \rightarrow \mathbb{C}^{n\times n}: X\mapsto X^*$. Then it is straightforward to verify that 
\[\D g(X)[h] = h^*, \forall h\in\mathbb{C}^{n\times n}.\]
Therefore for any $f: \mathbb{C}^{n\times n} \rightarrow \mathbb{R}$, by chain rule we have $\forall h\in \mathbb{C}^{n\times n}$
\begin{eqnarray*}
     \D (f\circ g)(X)[h] &=& \D f(g(X))[\D g(X) [h]] \\
     &=& \D f(X^*)[h^*]\\
     &=& \ip{\nabla f(X^*), h^*}_{\mathbb{C}^{n\times n}} \\
     &=& \ip{(\nabla f(X^*))^*, h}_{\mathbb{C}^{n\times n}}.
\end{eqnarray*}
Therefore we have 
\begin{equation*}
\nabla (f\circ g)(X) = (\nabla f(X^*))^*. 
\end{equation*}
So by the definition of Fr\'{e}chet derivative of $f\circ g$ at $X \in \mathbb{C}^{n\times n}$, we have the following.
\begin{equation*}
(f\circ g)(X+h) = (f\circ g)(X) + \ip{(\nabla f(X^*))^*, h}_{\mathbb{C}^{n\times n}} + o(\norm{h}), \quad \forall h\in \mathbb{C}^{n\times n}.
\end{equation*}

Let $\mathcal{H}^{n\times n} =\{X\in \mathbb C^{n\times n}: X^*=X\}$. Then $\mathcal{H}^{n\times n}$ is a linear subspace of the vector space $\mathbb{C}^{n\times n}$. Hence we can restrict $f$ to $\mathcal{H}^{n\times n}$ and define its Fr\'echet gradient in $\mathcal{H}^{n\times n}$. Let $\nabla^{\mathcal{H}} f$ denote the Fr\'{e}chet gradient of $f$ in $\mathcal{H}^{n\times n}$.
In particular, consider $X, h \in \mathcal{H}^{n\times n}$, then the above equality turns into
\begin{equation*}
    f(X+h) = f(X) + \ip{(\nabla f(X))^*, h}_{\mathbb{C}^{n\times n}} + o(\norm{h}), \quad \forall h \in \mathcal{H}^{n\times n}.
\end{equation*}
Hence we have
\begin{equation}\label{frechet_hermitian_1}
\forall X\in  \mathcal{H}^{n\times n},    \nabla^{\mathcal{H}}f(X) = (\nabla f(X))^*. 
\end{equation}
On the other hand, by the definition of Fr\'{e}chet derivative of $f$, we have 
\begin{equation*}
    f(X+h) = f(X) + \ip{\nabla f(X), h}_{\mathbb{C}^{n\times n}} + o(\norm{h}), \quad \forall h \in \mathbb{C}^{n\times n}.
\end{equation*}
In particular consider $X, h \in\mathcal{H}^{n\times n}$, then the above equality turns into 
\begin{equation*}
    f(X+h) = f(X) + \ip{\nabla f(X), h}_{\mathbb{C}^{n\times n}} + o(\norm{h}), \quad \forall h \in \mathcal{H}^{n\times n}.
\end{equation*}
This give us 
\begin{equation}\label{frechet_hermitian_2}
    \nabla^{\mathcal{H}}f(X) = \nabla f(X). 
\end{equation}

Combining \eqref{frechet_hermitian_1} and \eqref{frechet_hermitian_2}, we obtain the desired result.

\end{proof}

\begin{prop}
Let $\beta: \mathbb{C}^{n\times p} \rightarrow \mathbb{C}^{n\times n}: Y \mapsto YY^*$ and the the inner product on $\mathbb{C}^{n\times p}$ as the standard inner product $\ip{A,B}_{\mathbb{C}^{n\times p}} = \Re(\mbox{tr}(A^*B))$.  Then the Fr\'{e}chet gradient of $F := f\circ \beta$ satisfies
\begin{equation}\label{eqn:frechet_chain_rule}
    \nabla F(Y) = 2\nabla f(YY^*) Y.
\end{equation}
\end{prop}
\begin{proof}
Indeed, by the chain rule of Fr\'{e}chet derivative we have 
\begin{equation*}
    \D F(Y)[h] = \D f(\beta(Y))\left[ \D \beta(Y) [h]\right],\quad \forall h\in \mathbb{C}^{n\times p}.
\end{equation*}
Hence 
\begin{equation*}
    \ip{\nabla F(Y), h }_{\mathbb{C}^{n\times p}} = \ip{\nabla f(YY^*), \D \beta(Y)[h]}_{\mathbb{C}^{n\times n}}.
\end{equation*}
One can check by definition that $\D \beta(Y)[h] = Yh^* + hY^*$. 
Hence 
\begin{equation*}
    \ip{\nabla F(Y), h }_{\mathbb{C}^{n\times p}}= \ip{\nabla f(YY^*), Yh^* + hY^*}_{\mathbb{C}^{n\times n}} = \ip{2\nabla f(YY^*) Y, h}_{\mathbb{C}^{n \times p}}. 
\end{equation*}
This proves (\ref{eqn:frechet_chain_rule}). 
\end{proof}

\begin{prop}
If $f$ takes the form of $f(X)=\frac12\|\mathcal A(X)-b\|_F^2$  for a linear operator $\mathcal A$, then the Fr\'{e}chet gradient of $f(X)$ is given by 
\begin{equation*}
    \nabla f(X) = \mathcal{A}^*(\mathcal{A}(X)-b),
\end{equation*}
where $\mathcal{A}^*$ is the conjugate operator of $\mathcal{A}$. 
\end{prop}
\begin{proof}
We know by the definition of Fr\'{e}chet gradient that
\begin{equation*}
 \nabla f(X) =\frac{\partial f}{\partial \Re_X}+\mathbbm{i}\frac{\partial f}{\partial \Im_X},
\end{equation*}

Now for $f(X)=\frac{1}{2} \|\mathcal{A}(X)-b\|^2 = \frac{1}{2} \langle \mathcal{A}(X)  - b, \mathcal{A}(X) - b\rangle$,  by the linearity of $\mathcal{A}$, we have
\begin{eqnarray*}
	\nabla f(X) &=& \frac{1}{2} \left. \frac{\partial } {\partial X} \langle \mathcal{A} (X)-b, \mathcal{A}(\Delta ) -b \rangle \right|_{\Delta=X} + \frac{1}{2} \left. \frac{\partial } {\partial \Delta} \langle \mathcal{A} (X)-b, \mathcal{A}(\Delta ) -b \rangle \right|_{\Delta=X} \\  
	&=& \frac{1}{2} \left. \frac{\partial } {\partial X} \langle \mathcal{A} (X)-b, \mathcal{A}(\Delta ) -b \rangle \right|_{\Delta=X} + \frac{1}{2} \left. \frac{\partial } {\partial \Delta} \langle \mathcal{A} (\Delta)-b, \mathcal{A}(X ) -b \rangle \right|_{\Delta=X} \\ 
	&=& \frac{1}{2} \left. \frac{\partial } {\partial X} \langle \mathcal{A} (X)-b, \mathcal{A}(\Delta ) -b \rangle_{\mathbb{C}^{n\times n}} \right|_{\Delta=X} + \frac{1}{2} \left. \frac{\partial } {\partial \Delta} \langle \mathcal{A} (\Delta)-b, \mathcal{A}(X ) -b \rangle_{\mathbb{C}^{n\times n}} \right|_{\Delta=X} \\ 
	&=& \frac{1}{2} \left. \frac{\partial }{\partial X} \langle X, \mathcal{A}^*(\mathcal{A}(\Delta ) - b)\rangle_{\mathbb{C}^{n\times n }} \right|_{\Delta =X} + \frac{1}{2}\left. \frac{\partial }{\partial \Delta} \langle \Delta, \mathcal{A}^*(\mathcal{A}(X) - b) \rangle_{\mathbb{C}^{n\times n}} \right|_{\Delta = X}\\
	&=& \frac{1}{2} \left. \frac{\partial }{\partial X} \left( \langle \Re(X) , \Re(\mathcal{A}^*(\mathcal{A}(\Delta ) - b)) \rangle +  \langle \Im(X) , \Im(\mathcal{A}^*(\mathcal{A}(\Delta ) - b)) \rangle  \right) \right|_{\Delta=X} \\
	&& + \frac{1}{2} \left. \frac{\partial }{\partial \Delta} \left( \langle \Re(\Delta) , \Re(\mathcal{A}^*(\mathcal{A}(X ) - b)) \rangle +  \langle \Im(\Delta) , \Im(\mathcal{A}^*(\mathcal{A}(X ) - b)) \rangle  \right) \right|_{\Delta=X} \\
	&=& \frac{1}{2} \left. \left(\frac{\partial }{\partial \Re(X)}+ \mathbbm{i}\frac{\partial }{\partial \Im(X)} \right) \left( \langle \Re(X) , \Re(\mathcal{A}^*(\mathcal{A}(\Delta ) - b)) \rangle +  \langle \Im(X) , \Im(\mathcal{A}^*(\mathcal{A}(\Delta ) - b)) \rangle  \right) \right|_{\Delta=X} \\
	&& + \frac{1}{2} \left. \left(\frac{\partial }{\partial \Re(\Delta)}+ \mathbbm{i}\frac{\partial }{\partial \Im(\Delta)} \right)\left( \langle \Re(\Delta) , \Re(\mathcal{A}^*(\mathcal{A}(X ) - b)) \rangle +  \langle \Im(\Delta) , \Im(\mathcal{A}^*(\mathcal{A}(X ) - b)) \rangle  \right) \right|_{\Delta=X} \\
	&=& \left. \frac{1}{2}(\Re(\mathcal{A}^*(\mathcal{A}(\Delta ) - b)) + \mathbbm{i} \Im(\mathcal{A}^*(\mathcal{A}(\Delta ) - b))) \right|_{\Delta =X}\\&& + \left. \frac{1}{2}(\Re(\mathcal{A}^*(\mathcal{A}(X ) - b)) + \mathbbm{i} \Im(\mathcal{A}^*(\mathcal{A}(X ) - b))) \right|_{\Delta =X} \\
	&=& \mathcal{A}^*(\mathcal{A}(X)-b). 
\end{eqnarray*}
\end{proof}

\subsection{Fr\'{e}chet Hessian}\label{frechet_hessian}
For a Euclidean space $\mathcal{E}$ and a twice-differentiable, real-valued function $f$ on $\mathcal{E}$, the \textit{Fr\'{e}chet Hessian operator} of $f$ at $x$ is the unique symmetric operator $\nabla^2 f(x): \mathcal{E}\rightarrow\mathcal{E}$ defined by 
\begin{equation*}
   \nabla^2 f(x) [h] = \D (\nabla f)(x)[h]
\end{equation*}
for all $h \in \mathcal{E}$. 

\begin{prop}
Regard $\mathcal{E} = \mathbb{C}^{n\times n}$ as a Euclidean space over $\mathbb{R}$ with the standard inner product $\langle X ,Y \rangle_{\mathbb{C}^{n\times n}}
=\Re(\mbox{tr}(X^*Y))$. If $X=X^*$ and $h = h^*$, then 
\[
\nabla^2f(X)[h] = (\nabla^2f(X)[h])^*.
\]
\end{prop}
\begin{proof}
Let $g: \mathbb{C}^{n\times n} \rightarrow \mathbb{C}^{n\times n}: X\mapsto X^*$. 
Consider the Fr\'{e}chet Hessian of $f\circ g$.  
By the definition of Fr\'{e}chet Hessian we have
\[
\nabla(f\circ g)(X+h) = \nabla (f\circ g)(X) + \nabla^2(f\circ g)(X)[h] + o(\norm{h}^2).
\]
We know from the proof of Proposition \ref{prop:Frechet_gradient_hermitian} that 
\[
\nabla (f\circ g) (X) = (\nabla f(X^*))^*.
\]
Hence 
\[
(\nabla f(X^*+h^*))^* = (\nabla f(X^*))^* + \nabla^2(f\circ g)(X)[h]  + o(\norm{h}^2).
\]

Therefore 
\[
\nabla^2(f\circ g)(X)[h] = (\nabla^2 f(X^*)[h^*])^*. 
\]

Now restrict $f\circ g$ on the subspace $\mathcal{H}^{n\times n}$, we have $f\circ g|_{\mathcal{H}^{n\times n}} = f|_{\mathcal{H}^{n\times n}}$. Hence the Fr\'{e}chet Hessian operator of $f$ on $\mathcal{H}^{n\times n}$ is $(\nabla^2f(X^*)[h^*])^* = (\nabla^2f(X)[h])^*$. On the other hand, the Fr\'{e}chet Hessian operator of $f$ on $\mathcal{H}^{n\times n}$ is denoted as $\nabla^2 f(X) [h]$. Hence if $X, h \in \mathcal{H}^{n\times n}$, we have
\[
\nabla^2 f(X) [h] = (\nabla^2 f(X) [h])^*. 
\]
This proves the proposition.

\end{proof}

\subsection{Taylor's formula}
Let $\mathcal{E}$ be finite-dimensional Euclidean space. Let $f$ be a twice-differentiable real-valued function on an open convex domain $\Omega \subset \mathcal{E}$. Then for all $x$ and $x+h \in \Omega $, 
\begin{equation*}
    f(x+h) = f(x) + \ip{\nabla f(x), h}_\mathcal{E} + \frac{1}{2}\ip{\nabla^2 f(x)[h],h}_\mathcal{E} + O\left(\norm{h}_\mathcal{E}^3 \right).
\end{equation*}

\section{Embedded manifold $\mathcal{H}^{n,p}_+$}\label{appen:embedded}
The geometry of the real case, i.e., $\mathcal{S}^{n,p}_+$ has been explored in  \cite{vandereycken_embedded_2009}. However, it is not straightforward to extend these results directly to the complex case. Although the methods of proofs of the complex case turn out to be similar to the real case, we still need to provide. In this paper, recall that a complex matrix manifold is viewed as a manifold over $\mathbb{R}$ instead of $\mathbb{C}$. One way is to identify a complex matrix with the pair of its real and imaginary part; another way is to identify the matrix with its \textit{realification}.
\begin{defn}[Realification]
The realification is an injective mapping $\mathcal{R}: \mathbb{C}^{n\times n} \rightarrow \mathbb{R}^{2n \times 2n}$ defined by replacing each entry $a_{ij}$ of $A = (a_{ij})_{n\times n}\in \mathbb{C}^{n\times n}$ by the  $2\times 2$ matrix $\bmat{\Re(a_{ij}) & - \Im(a_{ij}) \\ \Im(a_{ij}) & \Re(a_{ij})} $. It can be shown that $\mathcal{R}$ preserves the algebraic structure:
\begin{itemize}
    \item $\mathcal{R}(A+B) = \mathcal{R}(A) + \mathcal{R}(B)$
    \item $\mathcal{R}(AB) = \mathcal{R}(A)\mathcal{R}(B)$
    \item $\mathcal{R}(aA) = a\mathcal{R}(A) \quad \forall a\in \mathbb{R}$
    \item $\mathcal{R}(I) = I $
    \item $\mathcal{R}(A^*) = (\mathcal{R}(A))^T$
\end{itemize}
Hence $A \in \mathbb{C}^{n\times n}$ is invertible if and only if $\mathcal{R}(A)$ is invertible. \footnote{. See for example \url{https://www.maths.tcd.ie/pub/coursework/424/LieGroups.pdf}}
\end{defn}

\begin{lem}\label{lem:semialgebraic}
Let $\Gl(n,\mathbb{C})$ be the general linear group viewed as a real Lie group. Then it is a semialgebraic set.
\end{lem}
\begin{proof}
Recall that a subset of $\mathbb{R}^{m}$ is a \textit{semialgebraic} set if it can be obtained by finitely many intersections, union and set differences starting from sets of the from $\{x \in \mathbb{R}^{m}: P(x) >0  \}$ with $P$ a  polynomial on $\mathbb{R}^m$ \cite[Appendix B]{gibson_singular_1979}. Since  $\Gl(n,\mathbb{C})$  is viewed as a real Lie group, $\Gl(n,\mathbb{C})$ is understood as a subset of $\Gl(2n,\mathbb{R})$ through realification. It can be shown that 
\begin{equation*}
    \Gl(n,\mathbb{C}) = \left\{ X\in \Gl(2n,\mathbb{R}): XJ = JX\right\},\quad \text{ with } J = \mathcal{R}(\mathbbm{i}I).
\end{equation*}
We know that $\Gl(2n,\mathbb{R})$ is a semialgebraic set since it is the non-vanishing points of determinant; and $\{ X\in \mathbb{R}^{2n \times 2n}: XJ = JX\}$ is also a semialgebraic set by definition. Hence $\Gl(n,\mathbb{C})$ is a semialgebraic set. 
\end{proof}

\subsection{Riemannian Hessian operator}\label{Appen:riemannian_hessian_embedded}
Let $f$ be a smooth real-valued function on $\mathcal{H}^{n,p}_+$. In this section we derive the Riemannian Hessian operator of $f$. 

By \cite[section 4]{absil_projection-like_2012} we know that the retraction $R$ defined in (\ref{eqn:embedded_retraction_by_projection}) is a second-order retraction.

 Proposition 5.5.5 in \cite{absil_optimization_2008} states that if $R$ is a second-order retraction, then the Riemannian Hessian of $f$ can be computed in the following nice way:
\begin{equation*}
    \Hess f(X) = \Hess (f\circ R_X)(0_X).
\end{equation*}
 Notice that now $f\circ R_X$ is a smooth function defined on a vector space. Hence, we obtain
\begin{eqnarray*}
    g_X\left(\Hess f(X)[\xi_X],\xi_X\right) = \left.\frac{d^2}{dt^2}f(R_X(t\xi_X))\right|_{t=0}.
\end{eqnarray*}
However, it is difficult to obtain a second-order derivative of $f\circ R_X $ using the retraction $R_X$ defined in  (\ref{eqn:embedded_retraction_by_projection}). The references \cite{vandereycken_riemannian_2010} and \cite{vandereycken_low-rank_2013} proposed a method to compute $\Hess f(X)$ by constructing a second-order retraction $R^{(2)}$ that has a second-order series expansion which makes it simple to derive a series expansion of $f\circ R_X^{(2)}$ up to second order and thus obtain the Hessian of $f$. 
We will summarize the derivation below.
\begin{lem}
For any $X\in \mathcal{H}^{n,p}_+$ with  $X^\dagger$ the pseudoinverse, the mapping $R_X^{(2)}: T_X \mathcal{H}^{n,p}_+ \rightarrow \mathcal{H}^{n,p}_+$ given by 
\begin{equation*}
    \xi_X \mapsto w X^\dagger  w^*, \text{ with } w = X + \frac{1}{2}\xi_X^s + \xi_X^p - \frac{1}{8}\xi_X^s X^\dagger \xi_X^s - \frac{1}{2}\xi_X^p X^\dagger\xi_X^s,
\end{equation*}
 is a second-order retraction on $\mathcal{H}^{n,p}_+$, 
where $\xi_X^s = P_X^s(\xi_X)$ and $\xi_X^p = P_X^p(\xi_X)$ as defined in \eqref{eqn:projection_sp}. Moreover we have 
\begin{equation*}
    R_X^{(2)}(\xi_X) = X + \xi_X + \xi_X^pX^\dagger\xi_X^p + O(\norm{\xi_X}^3).
\end{equation*}
\begin{proof}
It follows the same proof of \cite[Proposition 5.10]{vandereycken_riemannian_2010} . 
\end{proof}
\end{lem}

From this the Riemannian Hessian operator of $f$ can be computed in essentially the same way as in \cite[Section A.2]{vandereycken_low-rank_2012} but applied to  the general cost function $f(X)$ instead of a least square cost function. Consider the Taylor expansion of $\hat{f}_X^{(2)} := f\circ R_X^{(2)}$, which is a real-valued function on a vector space. We get
\begin{eqnarray*}
    \hat{f}_X^{(2)}(\xi_X) &=& f(R^{(2)}_X(\xi_X)) \\
    &=& f\left(X + \xi_X + \xi_X^pX^\dagger\xi_X^p + O(\norm{\xi_X}^3)\right) \\
    &=& f(X) + \ip{\nabla f(X),\xi_X+ \xi_X^pX^\dagger\xi_X^p}_{\mathbb{C}^{n\times n}} + \frac{1}{2}\ip{\nabla^2f(X)[\xi_X+  \xi_X^pX^\dagger\xi_X^p], \xi_X+ \xi_X^pX^\dagger\xi_X^p}_{\mathbb{C}^{n\times n}} + O(\norm{\xi_X}^3)\\
    &=& f(X) + \ip{\nabla f(X),\xi_X}_{\mathbb{C}^{n\times n}} + \ip{\nabla f(X),\xi_X^pX^\dagger\xi_X^p}_{\mathbb{C}^{n\times n}} +\frac{1}{2}\ip{\nabla^2f(X)[\xi_X],\xi_X}_{\mathbb{C}^{n\times n}}+ O(\norm{\xi_X}^3). 
\end{eqnarray*}
We can immediately recognize the first order term and the second order term that contribute to the Riemannian gradient and Hessian, respectively. That is, 
\begin{eqnarray*}
    &&g_X\left(\grad f(X),\xi_X\right) = \ip{\nabla f(X),\xi_X}_{\mathbb{C}^{n\times n}}, \\
    && g_X\left(\Hess f(X)[\xi_X],\xi_X\right) = \underbrace{2\ip{\nabla f(X),\xi_X^pX^\dagger\xi_X^p}_{\mathbb{C}^{n\times n}}}_{f_1:= \ip{\mathcal{H}_1(\xi_X),\xi_X}_{\mathbb{C}^{n\times n}}} + \underbrace{\ip{\nabla^2 f(X)[\xi_X],\xi_X}_{\mathbb{C}^{n\times n}}}_{f_2:= \ip{\mathcal{H}_2(\xi_X),\xi_X}_{\mathbb{C}^{n\times n}}}.
\end{eqnarray*}
The first equation immediately gives us
\begin{equation*}
    \grad f(X) = P^t_X(\nabla f(X)). 
\end{equation*}

For the second equation, the inner product of the Riemannian Hessian consists of the sum of $f_1$ and $f_2$; and the Riemannian Hessian operator is the sum of two operators $\mathcal{H}_1$ and $\mathcal{H}_2$. Since $\xi_X$ is already separated in $f_2$, the contribution to Riemannian Hessian from $\mathcal{H}_2$ is readily given by 
\begin{equation*}
    \mathcal{H}_2(\xi_X) = P^t_X(\nabla^2 f(X)[\xi_X]). 
\end{equation*}

Now, we still need to separate $\xi_X$ in $f_1$ to see the contribution to Riemannian Hessian from $\mathcal{H}_1$. Since we can choose to bring over $\xi_X^p X^\dagger $ or $X^\dagger \xi_X^p$ to the first position of $\ip{.,.}_{\mathbb{C}^{n\times n}}$, we write $\mathcal{H}_1(\xi_X)$ as the linear combination of both:
\begin{equation*}
    f_1 = 2c \ip{\nabla f(X)(X^\dagger\xi_X^p)^*, \xi_X^p}_{\mathbb{C}^{n\times n}}+  2(1-c) \ip{(\xi_X^p X^\dagger)^*\nabla f(X), \xi_X^p}_{\mathbb{C}^{n\times n}}.
\end{equation*}

Operator $ \mathcal{H}_1$ is clearly linear. Since $\mathcal{H}_1$ is symmetric, we must have $\ip{\mathcal{H}_1(\xi_X),\nu_X}_{\mathbb{C}^{n\times n}} = \ip{\nu_X,\mathcal{H}_1(\xi_X)}_{\mathbb{C}^{n\times n}}$ for all tangent vector $\nu_X$. Hence we must have $c = \frac{1}{2}$ and we obtain
\begin{equation*}
    \mathcal{H}_1(\xi_X) = P^p_X\left(\nabla f(X)(X^\dagger\xi_X^p)^* + (\xi_X^p X^\dagger)^*\nabla f(X)\right). 
\end{equation*}

Putting $\mathcal{H}_1$ and $\mathcal{H}_2$ together, we obtain 
\begin{equation*}
    \Hess f(X) [\xi_X] = P^t_X(\nabla^2 f(X)[\xi_X]) + P^p_X\left(\nabla f(X)(X^\dagger\xi_X^p)^* + (\xi_X^p X^\dagger)^*\nabla f(X)\right).
\end{equation*}

\section{Quotient manifold $\mathbb{C}^{n\times p}_*/\mathcal{O}_p$}\label{appen:quotient}

\subsection{Calculations for the Riemannian Hessian}\label{Appen:riemannian_hessian_quotient}
In this section, we outline the computations of the Riemannian Hessian operators of the cost function $h$ defined on $\mathbb{C}^{n\times p}_*/\mathcal{O}_p$ under the three different metrics $g^i$. 
\begin{defn}\cite[Definition 5.5.1]{absil_optimization_2008}
Given a real-valued function $f$ on a Riemannian manifold $\mathcal{M}$, the Riemannian Hessian of $f$ at a point $x$ in $\mathcal{M}$ is the linear mapping $\Hess f(x)$ of $T_x\mathcal{M}$ into itself defined by 
\begin{equation*}
    \Hess f(x)[\xi_x] = \nabla_{\xi_x}\grad f(x)
\end{equation*}
for all $\xi_x$ in $T_x\mathcal{M}$, where $\nabla$ is the Riemannian connection on $\mathcal{M}$.
\end{defn}

\begin{lem}\label{lemma:horizontal_lift_of_quotient_hessian}
The Riemannian Hessian of $h: \mathbb{C}^{n\times p}_*/\mathcal{O}_p \mapsto \mathbb{R}$ is related to the Riemannian Hessian of $F: \mathbb{C}^{n\times p}_* \mapsto \mathbb{R}$ in the following way:
\begin{equation*}
    \overline{\left( \Hess h(\pi(Y))[\xi_{\pi(Y)}]\right)}_{Y} = P_Y^\mathcal{H}\left( \Hess F(Y)[\overline{\xi}_Y]\right),
\end{equation*}
where $\overline{\xi}_Y$ is the horizontal lift of $\xi_{\pi(Y)}$ at $Y$.
\end{lem}
\begin{proof}
 The result follows from \cite[Proposition 5.3.3]{absil_optimization_2008} and the definition of the Riemannian Hessian.
\end{proof}

\subsubsection{Riemannian Hessian for the metric $g^1$}
Using the Riemannian metric $g^1$, $\mathbb{C}^{n\times p}_*$ is a Riemannian submanifold of a Euclidean space.
By \cite[Proposition 5.3.2]{absil_optimization_2008}, the Riemannian connection on $\mathbb{C}^{n\times p}_*$ is the classical directional derivative
\begin{equation*}
    \nabla_{\eta_{Y}}\xi = \D\xi(Y)[\eta_Y]. 
\end{equation*}
Recall that for $g^1$, $\grad F(Y) = 2\nabla f(YY^*)Y$.  Hence, the Riemannian Hessian of $F$ at $Y$ is given by 
\begin{eqnarray*}
    \Hess F(Y)[\xi_Y] &=& \nabla_{\xi_Y}\grad F \\
    &=& \D\grad F(Y) [\xi_Y] \\
    &=& 2\nabla^2 f(YY^*)[Y\xi_Y^*+\xi_YY^*]Y + 2\nabla f(YY^*)\xi_Y.
\end{eqnarray*}
The last line is by product rule and chain rule of differential. 
Therefore we obtain
\begin{equation*}
    \overline{\left( \Hess h(\pi(Y))[\xi_{\pi(Y)}]\right)}_{Y} = P^{\mathcal{H}^1}_Y \left(2\nabla^2 f(YY^*)[Y\overline{\xi}_Y^*+\overline{\xi}_YY^*]Y + 2\nabla f(YY^*)\overline{\xi}_Y \right).
\end{equation*}

\subsubsection{Riemannian Hessian under metric $g^2$}
First, for any Riemannian metric $g$, $g$ satisfies the Koszul formula
\begin{eqnarray*}
    2g_x(\nabla_{\xi_x}\lambda, \eta_x) &=& \xi_x g(\lambda,\eta) + \lambda_x g(\eta,\xi) -\eta_xg(\xi,\lambda) \\ 
    &&- g_x(\xi_x,[\lambda,\eta]_x)+ g_x(\lambda_x,[\eta,\xi]_x)+g_x(\eta,[\xi,\lambda]_x) \\
    &=& \D g(\lambda,\eta)(x)[\xi_x] + \D g(\eta,\xi)(x)[\lambda_x] - \D g(\xi,\lambda)(x)[\eta_x] \\
    && -  g_x(\xi_x,[\lambda,\eta]_x)+ g_x(\lambda_x,[\eta,\xi]_x)+g_x(\eta,[\xi,\lambda]_x),
\end{eqnarray*}
where the {\it Lie bracket} $[\cdot ,\cdot ]$ is defined in \cite{absil_optimization_2008}.

In particular, for $g^2$ the above Koszul formula turns into
\begin{eqnarray*}
    2g^2_Y(\nabla_{\xi_Y}\lambda, \eta_Y) &=& \D g^2(\lambda,\eta)(Y)[\xi_Y] + \D g^2(\eta,\xi)(Y)[\lambda_Y] - \D g^2(\xi,\lambda)(Y)[\eta_Y] \\
    && -  g^2_Y(\xi_Y,[\lambda,\eta]_Y)+ g^2_Y(\lambda_Y,[\eta,\xi]_Y)+g^2_Y(\eta,[\xi,\lambda]_Y).
\end{eqnarray*}

Recall that $g^2(\lambda,\eta)(Y) = \Re(tr(Y^*Y\lambda_Y^*\eta_Y))$. Hence, the first term in the above sum equals
\begin{equation*}
    \D g^2(\lambda,\eta)(Y)[\xi_Y] = g_Y^2(\D \lambda(Y)[\xi_Y],\eta_Y) + g_Y^2(\lambda_Y,\D\eta(Y)[\xi_Y]) + \Re(tr(\xi_Y^*Y\lambda_Y^*\eta_Y))+\Re(tr(Y^*\xi_Y\lambda_Y^*\eta_Y)).
\end{equation*}

Following \cite[Section  5.3.4]{absil_optimization_2008}, 
since $\mathbb{C}^{n\times p}_*$ is an open subset of $\mathbb{C}^{n\times p}$, we also have
\begin{equation*}
    [\lambda, \eta]_Y  = \D \eta(Y)[\lambda_Y] - \D \lambda(Y)[\eta_Y]. 
\end{equation*}

Summarizing, we get  
\begin{eqnarray*}
    2g^2_Y(\nabla_{\xi_Y}\lambda, \eta_Y) &=& \D g^2(\lambda,\eta)(Y)[\xi_Y] + \D g^2(\eta,\xi)(Y)[\lambda_Y] - \D g^2(\xi,\lambda)(Y)[\eta_Y] \\
    && -g^2(\xi_Y, \D\eta(Y)[\lambda_Y]- \D \lambda(Y)[\eta_Y]) \\
    && + g^2(\lambda_Y, \D\xi(Y)[\eta_Y]- \D\eta(Y)[\xi_Y])\\
    && +g^2(\eta_Y, \D\lambda(Y)[\xi_Y]- \D\xi(Y)[\lambda_Y])\\
    &=& 2g_Y^2(\eta_Y,\D\lambda(Y)[\xi_Y]) \\
    && + \Re(tr(\eta_Y^*(\lambda_Y(\xi_Y^*Y+Y^*\xi_Y)+\xi_Y(Y^*\lambda_Y+\lambda_Y^*Y)-Y\lambda_Y^*\xi_Y-Y\xi_Y^*\lambda_Y))) \\
    &=& 2g_Y^2(\eta_Y,\D\lambda(Y)[\xi_Y])  \\
    && + g_Y^2(\eta_Y,(\lambda_Y(\xi_Y^*Y+Y^*\xi_Y)+\xi_Y(Y^*\lambda_Y+\lambda_Y^*Y)-Y\lambda_Y^*\xi_Y-Y\xi_Y^*\lambda_Y)(Y^*Y)^{-1}).
\end{eqnarray*}
We therefore obtain a closed-form expression for Riemannian connection on $\mathbb{C}^{n\times p}_*$ for $g^2$:
\begin{equation*}
    \nabla_{\xi_Y}\lambda = \D\lambda(Y)[\xi_Y] + \frac{1}{2}\left(\lambda_Y(\xi_Y^*Y+Y^*\xi_Y)+\xi_Y(Y^*\lambda_Y+\lambda_Y^*Y)-Y\lambda_Y^*\xi_Y-Y\xi_Y^*\lambda_Y\right)(Y^*Y)^{-1}.
\end{equation*}
Recall that for the Riemannian metric $g^2$, we have $\grad F(Y) = 2\nabla f(YY^*)Y(Y^*Y)^{-1}$. Hence we have 
\begin{eqnarray*}
\Hess F(Y)[\xi_Y] &=&  \nabla_{\xi_Y}\grad F\\
&=& \D_Y \grad F(Y) [\xi_Y] \\ && + \frac{1}{2}\{ \grad F(Y) (\xi_Y^*Y+Y^*\xi_Y)+\xi_Y(Y^* \grad F(Y) + \grad F(Y) ^*Y)-\\&& Y \grad F(Y) ^*\xi_Y -Y\xi_Y^* \grad F(Y) \}(Y^*Y)^{-1}   \\
&=& 2\nabla^2f(YY^*)[Y\xi_Y^*+\xi_YY^*]Y(Y^*Y)^{-1} + 2\nabla f(YY^*)\xi_Y(Y^*Y)^{-1} \\
&& - 2\nabla f(YY^*)Y(Y^*Y)^{-1}(Y^*\xi_Y+\xi_Y^*Y)(Y^*Y)^{-1} \\
&&+ \nabla f(YY^*)Y(Y^*Y)^{-1}(Y^*\xi_Y+\xi_Y^*Y)(Y^*Y)^{-1} \\
&&+\xi_Y\{Y^*\nabla f(YY^*)Y(Y^*Y)^{-1}+(Y^*Y)^{-1}Y^*\nabla f(YY^*)Y\}(Y^*Y)^{-1}\\
&&-\{Y(Y^*Y)^{-1}Y^*\nabla f(YY^*)\xi_Y + Y\xi_Y^*\nabla f(YY^*)Y(Y^*Y)^{-1}\}(Y^*Y)^{-1} \\
&=& 2\nabla^2f(YY^*)[Y\xi_Y^*+\xi_YY^*]Y(Y^*Y)^{-1} + 2\nabla f(YY^*)\xi_Y(Y^*Y)^{-1} \\ 
&& - \nabla f(YY^*)Y(Y^*Y)^{-1}(Y^*\xi_Y+\xi_Y^*Y)(Y^*Y)^{-1}\\
&& + \xi_Y\{Y^*\nabla f(YY^*)Y(Y^*Y)^{-1}+(Y^*Y)^{-1}Y^*\nabla f(YY^*)Y \}(Y^*Y)^{-1}\\
&&-\{Y(Y^*Y)^{-1}Y^*\nabla f(YY^*)\xi_Y+Y\xi_Y^*\nabla f(YY^*)Y(Y^*Y)^{-1} \}(Y^*Y)^{-1} \\
&=& 2\nabla^2f(YY^*)[Y\xi_Y^*+\xi_YY^*]Y(Y^*Y)^{-1} + 2\nabla f(YY^*)\xi_Y(Y^*Y)^{-1} \\ 
&& -\nabla f(YY^*)P_Y\xi_Y(Y^*Y)^{-1}-\nabla f(YY^*)Y(Y^*Y)^{-1}\xi_Y^*Y(Y^*Y)^{-1}\\
&&+\xi_YY^*\nabla f(YY^*)Y(Y^*Y)^{-2}+\xi_Y(Y^*Y)^{-1}Y^*\nabla f(YY^*)Y(Y^*Y)^{-1}\\
&&-P_Y\nabla f(YY^*)\xi_Y(Y^*Y)^{-1}-Y\xi_Y^*\nabla f(YY^*)Y(Y^*Y)^{-2}\\
&=& 2\nabla^2f(YY^*)[Y\xi_Y^*+\xi_YY^*]Y(Y^*Y)^{-1} \\
&&+\nabla f(YY^*)\xi_Y(Y^*Y)^{-1}-\nabla f(YY^*)P_Y\xi_Y(Y^*Y)^{-1}\\
&&+\nabla f(YY^*)\xi_Y(Y^*Y)^{-1} - P_Y\nabla f(YY^*)\xi_Y(Y^*Y)^{-1}\\
&& +2skew(\xi_YY^*)\nabla f(YY^*)Y(Y^*Y)^{-2}\\
&&+2skew\{\xi_Y(Y^*Y)^{-1}Y^*\nabla f(YY^*)\}Y(Y^*Y)^{-1} \\
&=& 2\nabla^2f(YY^*)[Y\xi_Y^*+\xi_YY^*]Y(Y^*Y)^{-1}  \\
&& + \nabla f(YY^*)P_Y^\perp\xi_Y(Y^*Y)^{-1}+P_Y^\perp \nabla f(YY^*)\xi_Y(Y^*Y)^{-1}\\
&& +2skew(\xi_YY^*)\nabla f(YY^*)Y(Y^*Y)^{-2}\\
&&+2skew\{\xi_Y(Y^*Y)^{-1}Y^*\nabla f(YY^*)\}Y(Y^*Y)^{-1}.
\end{eqnarray*}
To conclude, we obtain
\begin{eqnarray*}
    \overline{\left( \Hess h(\pi(Y))[\eta_{\pi(Y)}]\right)}_{Y} &=& P^{\mathcal{H}^2}_Y \left\{2\nabla^2f(YY^*)[Y\overline{\xi}_Y^*+\overline{\xi}_YY^*]Y(Y^*Y)^{-1} \right. \\
    &&+ \nabla f(YY^*)P_Y^\perp\overline{\xi}_Y(Y^*Y)^{-1}+P_Y^\perp \nabla f(YY^*)\overline{\xi}_Y(Y^*Y)^{-1}  \\
    &&  +2skew(\overline{\xi}_YY^*)\nabla f(YY^*)Y(Y^*Y)^{-2}\\
    &&+ \left. 2skew\{\overline{\xi}_Y(Y^*Y)^{-1}Y^*\nabla f(YY^*)\}Y(Y^*Y)^{-1} \right\}.
\end{eqnarray*}

\subsubsection{Riemannian Hessian under metric  $g^3$}

Recall that the Riemannian metric $g^3$  on $\mathbb{C}^{n\times p}_*$ satisfies 
\begin{eqnarray*}
    g^3_Y(\xi_Y,\eta_Y) &=& \tilde{g}_Y(\xi_Y,\eta_Y)+ g^2_Y(P_Y^\mathcal{V}(\xi_Y),P_Y^\mathcal{V}(\eta_Y)) \\
    &=& 2\Re( tr(Y^*\xi_YY^*\eta_Y+Y^*Y\xi_Y^*\eta_Y) )+ \Re (tr(YP_Y^\mathcal{V}(\xi_Y)^*P_Y^\mathcal{V}(\eta_Y)Y^*))
\end{eqnarray*}
where 
\begin{equation*}
    \tilde{g}_Y(\xi_Y,\eta_Y) = \ip{Y\xi_Y^*+\xi_YY^*,Y\eta_Y^*+\eta_YY^*}_{\mathbb{C}^{n\times n}}.
\end{equation*}
\begin{equation*}
    P_Y^\mathcal{V}(\lambda_Y) = Y skew((Y^*Y)^{-1}Y^*\lambda_Y).
\end{equation*}

Hence 
\begin{eqnarray*}
    && \D g^3(\lambda,\eta)(Y)[\xi_Y]\\
    &=& \tilde{g}_Y(\D\lambda(Y)[\xi_Y],\eta_Y) + \tilde{g}(\lambda_Y,D\eta(Y)[\xi_Y]) + 2\Re(tr(\xi_Y^*\lambda_Y Y^*\eta_Y +Y^*\lambda_Y\xi_Y^*\eta_Y + \xi_Y^*Y\lambda_Y^*\eta_Y + Y^*\xi_Y\lambda_Y^*\eta_Y)) \\
    && + g^2_Y(P_Y^\mathcal{V}(\lambda_Y),DP_Y^\mathcal{V}(\eta_Y)[\xi_Y]) + g^2(\D P_Y^\mathcal{V}(\lambda_Y)[\xi_Y],P_Y^\mathcal{V}(\eta_Y)) \\
    &&+ \Re(tr(\xi_Y P_Y^\mathcal{V}(\lambda_Y)^*P_Y^\mathcal{V}(\eta_Y)Y^*+YP_Y^\mathcal{V}(\lambda_Y)^*P_Y^\mathcal{V}(\eta_Y)\xi_Y^*)). 
\end{eqnarray*}

Suppose $\lambda,\eta$ and $\xi$ are horizontal vector fields, then many terms in the above equation vanish:
\begin{eqnarray*}
    \D g^3(\lambda,\eta)(Y)[\xi_Y] &=& \tilde{g}_Y(\D\lambda(Y)[\xi_Y],\eta_Y) + \tilde{g}_Y(\lambda_Y,\D\eta_Y[\xi_Y]) \\
    && + 2\Re(tr(\xi_Y^*\lambda_Y Y^*\eta_Y  +Y^*\lambda_Y\xi_Y^*\eta_Y + \xi_Y^*Y\lambda_Y^*\eta_Y + Y^*\xi_Y\lambda_Y^*\eta_Y)).
\end{eqnarray*}

Combining the above equation and the Koszul formul with $\xi, \eta, \lambda$ horizontal vector fields, we obtain
\begin{eqnarray*}
    &&2g^3_Y(\nabla_{\xi_Y}\lambda,\eta_Y) \\
    &=& \D g^3(\lambda,\eta)(Y)[\xi_Y] + \D g^3(\eta,\xi)(Y)[\lambda_Y] - \D g^3(\xi,\lambda)(Y)[\eta_Y] \\
    && - g^3_Y(\xi_Y,\D\eta(Y)[\lambda_Y]-\D\lambda(Y)[\eta_Y])  \\
    && + g^3_Y(\lambda_Y,\D\xi(Y)[\eta_Y] - \D\eta(Y)[\xi_Y])  \\
    && + g^3_Y(\eta_Y,\D\lambda(Y)[\xi_Y]- \D\xi(Y)[\lambda_Y])\\
    &=&\tilde{g}_Y(\D\lambda(Y)[\xi_Y],\eta_Y) + \tilde{g}_Y(\lambda_Y,\D\eta(Y)[\xi_Y]) + 2\Re(tr(\xi_Y^*\lambda_Y Y^*\eta_Y+ Y^*\lambda_Y \xi_Y^*\eta_Y + \xi_Y^*Y\lambda_Y^*\eta_Y + Y^*\xi_Y \lambda_Y^* \eta_Y))\\ 
    &&+ \tilde{g}_Y(\D\eta(Y)[\lambda_Y],\xi_Y) + \tilde{g}_Y(\eta_Y,\D\xi(Y)[\lambda_Y]) + 2\Re(tr(\lambda_Y^*\eta_Y Y^*\xi_Y+ Y^*\eta_Y \lambda_Y^*\xi_Y + \lambda_Y^*Y\eta_Y^*\xi_Y + Y^*\lambda_Y \eta_Y^* \xi_Y))\\ 
    &&- \tilde{g}_Y(\D\xi(Y)[\eta_Y],\lambda_Y)- \tilde{g}_Y(\xi_Y,\D\lambda(Y)[\eta_Y]) - 2\Re(tr(\eta_Y^*\xi_Y Y^*\lambda_Y+ Y^*\xi_Y \eta_Y^*\lambda_Y + \eta_Y^*Y\xi_Y^*\lambda_Y + Y^*\eta_Y \xi_Y^* \lambda_Y))\\ 
    &&- \tilde{g}_Y(\xi_Y,\D\eta(Y)[\lambda_Y]) + \tilde{g}_Y(\xi_Y,\D\lambda(Y)[\eta_Y])\\
    &&+\tilde{g}_Y(\lambda_Y,\D\xi(Y)[\eta_Y]) - \tilde{g}_Y(\lambda_Y,\D\eta(Y)[\xi_Y])\\
    &&+\tilde{g}_Y(\eta_Y,\D\lambda(Y)[\xi_Y]) - \tilde{g}_Y(\eta_Y,\D\xi(Y)[\lambda_Y]) \\
    &=& 2\tilde{g}_Y(\D\lambda(Y)[\xi_Y],\eta_Y) + 4\Re(tr(Y^*\xi_Y\lambda_Y^*\eta_Y + Y^*\lambda_Y \xi_Y^*\eta_Y)).
\end{eqnarray*}

It follows that 
\begin{equation*}
    g^3_Y(\nabla_{\xi_Y}\lambda, \eta_Y) = \tilde{g}_Y(\D\lambda(Y)[\xi_Y],\eta_Y) + 2\Re(tr(Y^*\xi_Y \lambda_Y^* \eta_Y + Y^*\lambda_Y \xi^* \eta_Y )).
\end{equation*}

By definition, we have $\Hess F(Y)[\xi_Y] = \nabla_{\xi_Y}\grad F$. By Lemma (\ref{lemma:horizontal_lift_of_quotient_hessian}), it suffices to assume that $\xi_Y$ is a horizontal vector in order to obtain the Hessian operator of $h$.  Therefore,
\begin{eqnarray*}
    &&g^3_Y(\Hess F(Y)[\xi_Y],\eta_Y)\\
    &=& g^3_Y(\nabla_{\xi_Y}\grad F, \eta_Y) \\
    &=& \tilde{g}(\eta_Y, \D\grad F(Y)[\xi_Y])+2\Re(tr(Y^*\xi_Y\grad F(Y)^*\eta_Y + Y^*\grad F(Y)\xi_Y^*\eta_Y)) \\
    &=&\tilde{g}(\eta_Y, \D\grad F(Y)[\xi_Y]) + \Re(tr((Y\eta_Y^*+\eta_YY^*)(\grad F(Y)\xi_Y^*+\xi_Y\grad F(Y)^*))) \\
    &=& \tilde{g}(\eta_Y, \D\grad F(Y)[\xi_Y]) \\
    && + \tilde{g}\left(\eta_Y,\left(I-\frac{1}{2}P_Y\right)(\grad F(Y)\xi_Y^*+\xi_Y\grad F(Y)^*)Y(Y^*Y)^{-1}\right).
\end{eqnarray*}

Recall that for Riemannian metric $g^3$, we have $\grad F(Y) = \left(I - \frac{1}{2}P_Y \right) \nabla f(YY^*)Y(Y^*Y)^{-1}$. Hence 
\begin{eqnarray*}
\D \grad F(Y) [\xi_Y] 
    &=& \left(I-\frac{1}{2}P_Y\right)\nabla^2 f(YY^*)[Y\xi_Y^*+\xi_YY^*]Y(Y^*Y)^{-1}  \\
    && - \frac{1}{2}(\D(P_Y)[\xi_Y])\nabla f(YY^*)Y(Y^*Y)^{-1}\\
    &&+\left(I-\frac{1}{2}P_Y\right)\nabla f(YY^*)\D(Y(Y^*Y)^{-1})[\xi_Y],
\end{eqnarray*}
where we have 
\begin{eqnarray*}
    \D(P_Y)[\xi_Y] &=& \D(Y(Y^*Y)^{-1}Y^*)[\xi_Y] \\
    &=&\xi_Y(Y^*Y)^{-1}Y^* -Y(Y^*Y)^{-1}(\xi_Y^*Y+Y^*\xi_Y)(Y^*Y)^{-1}Y^*+Y(Y^*Y)^{-1}\xi_Y^*
\end{eqnarray*}
and 
\begin{eqnarray*}
    \D(Y(Y^*Y)^{-1})[\xi_Y] &=& \xi_Y(Y^*Y)^{-1} - Y(Y^*Y)^{-1}(\xi_Y^*Y+Y^*\xi_Y)(Y^*Y)^{-1}. 
\end{eqnarray*}

Combining these equations we have 
\begin{eqnarray*}
    &&g^3_Y(\Hess F(Y)[\xi_Y],\eta_Y) \\
    &=& \tilde{g}\left(\eta_Y,\left(I-\frac{1}{2}P_Y\right)\nabla^2f(YY^*)[Y\xi_Y^*+\xi_YY^*]Y(Y^*Y)^{-1}\right) \\
    && -\tilde{g}\left(\eta_Y,\frac{1}{2}(\xi_Y(Y^*Y)^{-1}Y^*-Y(Y^*Y)^{-1}(\xi_Y^*Y+Y^*\xi_Y)(Y^*Y)^{-1}Y^*+Y(Y^*Y)^{-1}\xi_Y^*)\nabla f(YY^*)Y(Y^*Y)^{-1}\right) \\
    && + \tilde{g}\left(\eta_Y,\left(I-\frac{1}{2}P_Y\right)\nabla f(YY^*)\left(\xi_Y(Y^*Y)^{-1}-Y(Y^*Y)^{-1}(\xi_Y^*Y+Y^*\xi_Y)(Y^*Y)^{-1}\right)\right)\\
    &&+\tilde{g}\left(\eta_Y,\left(I-\frac{1}{2}P_Y\right)\left(\left(I-\frac{1}{2}P_Y\right)\nabla f(YY^*)Y(Y^*Y)^{-1}\xi_Y^* \right.\right.\\
    && +\left.\left.\xi_Y(Y^*Y)^{-1}Y^*\nabla f(YY^*)\left(I-\frac{1}{2}P_Y\right)\right) Y(Y^*Y)^{-1}\right) \\
    &=&\tilde{g}\left(\eta_Y,\left(I-\frac{1}{2}P_Y\right)\nabla^2f(YY^*)[Y\xi_Y^*+\xi_YY^*]Y(Y^*Y)^{-1}\right) \\
    && -\tilde{g}\left(\eta_Y,\frac{1}{2}(\xi_Y(Y^*Y)^{-1}Y^*-Y(Y^*Y)^{-1}(\xi_Y^*Y+Y^*\xi_Y)(Y^*Y)^{-1}Y^*+Y(Y^*Y)^{-1}\xi_Y^*)\nabla f(YY^*)Y(Y^*Y)^{-1}\right) \\
    && + \tilde{g}\left(\eta_Y,\left(I-\frac{1}{2}P_Y\right)\nabla f(YY^*)\left(\xi_Y(Y^*Y)^{-1}-Y(Y^*Y)^{-1}(\xi_Y^*Y+Y^*\xi_Y)(Y^*Y)^{-1}\right)\right)\\
    && + \tilde{g}\left(\eta_Y,\left(I-\frac{3}{4}P_Y\right)\nabla f(YY^*)Y(Y^*Y)^{-1}\xi_Y^*Y(Y^*Y)^{-1}\right)\\
    && + \tilde{g}\left(\eta_Y,\frac{1}{2}\left(I-\frac{1}{2}P_Y\right)\xi_Y(Y^*Y)^{-1}Y^*\nabla f(YY^*)Y(Y^*Y)^{-1}\right) \\
    &=& \tilde{g}\left(\eta_Y,\left(I-\frac{1}{2}P_Y\right)\nabla^2f(YY^*)[Y\xi_Y^*+\xi_YY^*]Y(Y^*Y)^{-1}\right) \\
    &&  - \tilde{g}\left(\eta_Y,\frac{1}{2}\xi_Y(Y^*Y)^{-1}Y^*\nabla f(YY^*)Y(Y^*Y)^{-1}\right)\\
    && - \tilde{g}\left(\eta_Y,\frac{1}{2}Y(Y^*Y)^{-1}\xi_Y^*\nabla f(YY^*)Y(Y^*Y)^{-1}\right)\\
    && +  \tilde{g}\left(\eta_Y,\frac{1}{2}Y(Y^*Y)^{-1}\xi_Y^*P_Y\nabla f(YY^*)Y(Y^*Y)^{-1}\right)\\
    && + \tilde{g}\left(\eta_Y,\frac{1}{2}P_Y\xi_Y(Y^*Y)^{-1}Y^*\nabla f(YY^*)Y(Y^*Y)^{-1}\right)\\ 
    && + \tilde{g}\left(\eta_Y, \left(I-\frac{1}{2}P_Y\right)\nabla f(YY^*)\left((I-P_Y)\xi_Y(Y^*Y)^{-1}-Y(Y^*Y)^{-1}\xi_Y^*Y(Y^*Y)^{-1}\right)\right)\\
    && + \tilde{g}\left(\eta_Y,\left(I-\frac{1}{2}P_Y\right)\nabla f(YY^*)Y(Y^*Y)^{-1}\xi_Y^*Y(Y^*Y)^{-1}-\frac{1}{4}P_Y\nabla f(YY^*)Y(Y^*Y)^{-1}\xi_Y^*Y(Y^*Y)^{-1}\right)\\
    && + \tilde{g}\left(\eta_Y,\frac{1}{2}\left(I-P_Y\right)\xi_YY(Y^*Y)^{-1}Y^*\nabla f(YY^*)Y(Y^*Y)^{-1}+\frac{1}{4}P_Y\xi_Y(Y^*Y)^{-1}Y^*\nabla f(YY^*)Y(Y^*Y)^{-1}\right) \\
    &=& \tilde{g}\left(\eta_Y,\left(I-\frac{1}{2}P_Y\right)\nabla^2f(YY^*)[Y\xi_Y^*+\xi_YY^*]Y(Y^*Y)^{-1}\right) \\
    &&+ \tilde{g}\left(\eta_Y,(I-P_Y)\nabla f(YY^*)(I-P_Y)\xi_Y(Y^*Y)^{-1}\right) \\
    &&+ \tilde{g}\left(\eta_Y,\frac{1}{2}Yskew\left((Y^*Y)^{-1}Y\xi_Y(Y^*Y)^{-1}Y^*\nabla f(YY^*)Y(Y^*Y)^{-1}\right)\right)\\
    &&+\tilde{g}\left(\eta_Y,Yskew\left((Y^*Y)^{-1}Y^*\nabla f(YY^*)(I-P_Y)\xi_Y(Y^*Y)^{-1}\right)\right) \\
    &=&\tilde{g}\left(\eta_Y,\left(I-\frac{1}{2}P_Y\right)\nabla^2f(YY^*)[Y\xi_Y^*+\xi_YY^*]Y(Y^*Y)^{-1}\right) \\
    && + \tilde{g}\left(\eta_Y,(I-P_Y)\nabla f(YY^*)(I-P_Y)\xi_Y(Y^*Y)^{-1}\right) \\
    &=& g^3_Y\left(\eta_Y,\left(I-\frac{1}{2}P_Y\right)\nabla^2f(YY^*)[Y\xi_Y^*+\xi_YY^*]Y(Y^*Y)^{-1}+(I-P_Y)\nabla f(YY^*)(I-P_Y)\xi_Y(Y^*Y)^{-1}\right)
\end{eqnarray*}

Hence for $\xi_Y \in \mathcal{H}_Y$, we have 
\begin{eqnarray*}
     \Hess F(Y) [\xi_Y] &=& \left(I-\frac{1}{2}P_Y\right)\nabla^2f(YY^*)[Y\xi_Y^*+\xi_YY^*]Y(Y^*Y)^{-1}\\
    && + (I-P_Y)\nabla f(YY^*)(I-P_Y)\xi_Y(Y^*Y)^{-1}
\end{eqnarray*}
To summarize, we obtain 
\begin{eqnarray*}
    \overline{\left( \Hess h(\pi(Y))[\eta_{\pi(Y)}]\right)}_{ Y} &=& P^{\mathcal{H}^3}_Y(\Hess F(Y)[\overline{\xi}_Y]) \\
    &=& \left(I-\frac{1}{2}P_Y\right)\nabla^2f(YY^*)[Y\overline{\xi}_Y^*+\overline{\xi}_YY^*]Y(Y^*Y)^{-1}\\
    && + (I-P_Y)\nabla f(YY^*)(I-P_Y)\overline{\xi}_Y(Y^*Y)^{-1}.
\end{eqnarray*}

\section{Proof of lemmas}\label{sec-appendix-lemma}
\subsection{Proof of Lemma \ref{lemma: rayleigh_quotient_second_order_term}}
\begin{proof}
 It is straightforward to see $C^3_{\pi(Y)} = D^3_{\pi(Y)} = 1$ by the definition of $g^3$. 
 
 For metric 2, write $\overline{\xi}_Y = YS + Y_\perp K$ for some $S = S^* \in \mathbb{C}^{p \times p}$ and $K \in \mathbb{C}^{n\times p}$. We have 
 \[
    \frac{\norm{Y\overline{\xi}_Y^*+\overline{\xi}_YY^*}^2_F}{g^2_Y(\overline{\xi}_Y,\overline{\xi}_Y)} = 2 + \frac{2\norm{YSY^*}^2_F}{\norm{YSY^*}^2_F + \norm{KY^*}^2_F}. 
\]
Hence it is easy to see $C^2_{\pi(Y)} = 2$ when $S$ is zero matrix and $D^2_{\pi(Y)} =4$ when $YSY^*$ is nonzero and $K$ is zero matrix. 

For metric 1, by its horizontal space, we can write $\overline{\xi}_Y = Y(Y^*Y)^{-1}S + Y_\perp K$ for some $S = S^* \in \mathbb{C}^{p \times p}$ and $K \in \mathbb{C}^{n\times p}$. 
Notice that the SVD of $Y$ can be given as $Y=U\Sigma^{\frac12} V^*$ where $V$ is unitary. Let $\bar S=V^*SV$ and $\bar K=KV$, and $\bar K_i$ be the $i$th column of $\bar K$,
then
\begin{eqnarray}
     \frac{\norm{Y\overline{\xi}_Y^*+\overline{\xi}_YY^*}^2_F}{g^1_Y(\overline{\xi}_Y,\overline{\xi}_Y)} \notag &=& \frac{\norm{Y((Y^*Y)^{-1}S + S(Y^*Y)^{-1})Y^*}_F^2+ 2\norm{KY^*}^2_F}{\norm{Y(Y^*Y)^{-1}S}^2_F + \norm{K}^2_F} \\ \notag
    &=& \frac{\norm{\Sigma^{-\frac{1}{2}} \bar S \Sigma^{\frac{1}{2}} + \Sigma^{\frac{1}{2}} \bar S \Sigma^{-\frac{1}{2}}}_F^2 + 2\norm{\bar K\Sigma^{\frac{1}{2}}}_F^2}{\norm{\Sigma^{-\frac{1}{2}}\bar S}_F^2 + \norm{\bar K}_F^2} \\\notag
    &=&\frac{\sum\limits_{i,j=1}^p \left( \frac{{\sigma_i}}{{\sigma_j}}+ \frac{{\sigma_j}}{{\sigma_i}} + 2 \right) \abs{\bar S_{ij}}^2 + 2 \sum\limits_{i=1}^p \sigma_i \norm{\bar K_i}_F^2}{\sum\limits_{i,j =1}^p \frac{\abs{\bar S_{ij}}^2}{\sigma_i} + \sum\limits_{i=1}^p \norm{\bar K_i}_F^2}\\
    &=&\label{eqn:rayleigh_quotient_easy_quotient1} \frac{2\sum\limits_{i,j=1}^p  \frac{{\sigma_j}}{{\sigma_i}}\abs{\bar S_{ij}}^2 + 2\sum\limits_{i,j=1}^p  \abs{\bar S_{ij}}^2 + 2 \sum\limits_{i=1}^p \sigma_i \norm{K_i}_F^2}{\sum\limits_{i,j =1}^p \frac{\abs{\bar S_{ij}}^2}{\sigma_i} + \sum\limits_{i=1}^p \norm{\bar K_i}_F^2},
\end{eqnarray} 
where symmetry $\bar S^*=\bar S$ is used in the last step. 
The lower bound is given by
\begin{eqnarray*}
\frac{2\sum\limits_{i,j=1}^p  \frac{{\sigma_j}}{{\sigma_i}}\abs{\bar S_{ij}}^2 + 2\sum\limits_{i,j=1}^p  \abs{\bar S_{ij}}^2 + 2 \sum\limits_{i=1}^p \sigma_i \norm{\bar K_i}_F^2}{\sum\limits_{i,j =1}^p \frac{\abs{\bar S_{ij}}^2}{\sigma_i} + \sum\limits_{i=1}^p \norm{\bar K_i}_F^2} &\geq& \frac{2\left(\frac{\sigma_p}{\sigma_1}+1\right)\sum\limits_{i,j=1}^p \abs{\bar S_{ij}}^2+2\sigma_p \sum\limits_{i=1}^p\norm{\bar K_i}^2_F}{\frac{1}{\sigma_p}\sum\limits_{i,j=1}^p \abs{\bar S_{ij}}^2 + \sum\limits_{i=1}^p\norm{\bar K_i}^2_F} \\
&=& \frac{2\left(\frac{\sigma_p^2}{\sigma_1}+\sigma_p\right)\sum\limits_{i,j=1}^p \abs{\bar S_{ij}}^2+2\sigma_p^2 \sum\limits_{i=1}^p\norm{\bar K_i}^2_F}{\sum\limits_{i,j=1}^p \abs{\bar S_{ij}}^2 + \sigma_p \sum\limits_{i=1}^p\norm{\bar K_i}^2_F} \\
&\geq& 2 \sigma_p. 
\end{eqnarray*}
This lower bound is sharp as one can choose $S = 0$ and $K$ with $\norm{\bar K_p}_F = 1$ and $\norm{\bar K_i}_F = 0 $  for $i < p$.

We have the upper bound as follows.
\begin{eqnarray*}
\frac{2\sum\limits_{i,j=1}^p  \frac{{\sigma_j}}{{\sigma_i}}\abs{\bar S_{ij}}^2 + 2\sum\limits_{i,j=1}^p  \abs{\bar S_{ij}}^2 + 2 \sum\limits_{i=1}^p \sigma_i \norm{\bar K_i}_F^2}{\sum\limits_{i,j =1}^p \frac{\abs{\bar S_{ij}}^2}{\sigma_i} + \sum\limits_{i=1}^p \norm{\bar K_i}_F^2} &\leq& \frac{2\left(\frac{\sigma_1}{\sigma_p}+1\right)\sum\limits_{i,j=1}^p \abs{\bar S_{ij}}^2+2\sigma_1 \sum\limits_{i=1}^p\norm{\bar K_i}^2_F}{\frac{1}{\sigma_1}\sum\limits_{i,j=1}^p \abs{\bar S_{ij}}^2 + \sum\limits_{i=1}^p\norm{\bar K_i}^2_F} \\
&=& \frac{2\left(\frac{\sigma_1^2}{\sigma_p}+\sigma_1\right)\sum\limits_{i,j=1}^p \abs{\bar S_{ij}}^2+2\sigma_1^2 \sum\limits_{i=1}^p\norm{\bar K_i}^2_F}{\sum\limits_{i,j=1}^p \abs{\bar S_{ij}}^2 + \sigma_1 \sum\limits_{i=1}^p\norm{\bar K_i}^2_F} \\
&\leq& 2\left(\frac{\sigma_1^2}{\sigma_p}+\sigma_1\right), 
\end{eqnarray*}
where the last inequality is obtained by investigating the range of the rational function $f(x,y) = \frac{ax+by}{x+dy}$ with $a = 2\left(\frac{\sigma_1^2}{\sigma_p}+\sigma_1\right), b = 2\sigma_1^2$  and $d = \sigma_1$ on $\{(x,y) | x\geq 0, y\geq 0, xy \neq 0 \}$.  

This upper bound $2\left(\frac{\sigma_1^2}{\sigma_p}+\sigma_1\right)$ may not be the supremum as the inequalities are not sharp. However, we can show that $D^1_{\pi(Y)} \geq 2\sigma_1$. To see this, choose $\bar S = 0$ and $K$ with $\norm{\bar K_1}_F = 1$ and $\norm{\bar K_i}_F = 0$ for $i>1$. Then \eqref{eqn:rayleigh_quotient_easy_quotient1} reaches the value $2\sigma_1$. Hence the supremum must be at least $2\sigma_1$. So we have
\begin{equation}
    2\sigma_1 \leq D^1_{\pi(Y)} \leq 2\left(\frac{\sigma_1^2}{\sigma_p}+\sigma_1\right). 
\end{equation}
\end{proof}

\subsection{Proof of Lemma \ref{lemma: rayleigh_quotient_first_order_term}}
\begin{proof}
We will use the inequality $\|B^* A^*\|_F=\|AB\|_F\leq \|A\| \|B\|_F\leq \|A\|_F \|B\|_F$ for two matrices where $\|A\|$ is the spectral norm. In particular, if $X$ is Hermitian, then we also have
$\|AX\|_F=\|XA^*\|_F\leq \|X\| \|A^*\|_F=\|X\| \|A\|_F.$

For the embedded manifold,  recall that  $\xi_X^s = P_X^s(\xi_X)$ and $\xi_X^p = P_X^p(\xi_X)$ and $P^t_X $ and $P^p_X$ are defined in \eqref{eqn:projection_sp}, and
the bound for the FOT is given by
\begin{eqnarray*}
    &&\frac{\abs{g_X\left(P^p_X\left(\nabla f(X)(X^\dagger\zeta_X^p)^* + (\zeta_X^p X^\dagger)^*\nabla f(X)\right), \zeta_X\right)}}{g_X(\zeta_X,\zeta_X)} =   \frac{\abs{\ip{P_X^p\left(\nabla f(X) \zeta_X^p X^\dagger + X^\dagger \zeta_X^p \nabla f(X) \right),\zeta_X}_{\mathbb{C}^{n\times n}}}}{\ip{\zeta_X,\zeta_X}_{\mathbb{C}^{n\times n}}}\\
    &\leq &\frac{\abs{\ip{P_X^p\left(\nabla f(X) \zeta_X^p X^\dagger  \right),\zeta_X}_{\mathbb{C}^{n\times n}}}}{\ip{\zeta_X,\zeta_X}_{\mathbb{C}^{n\times n}}}+\frac{\abs{\ip{P_X^p\left( X^\dagger \zeta_X^p \nabla f(X) \right),\zeta_X}_{\mathbb{C}^{n\times n}}}}{\ip{\zeta_X,\zeta_X}_{\mathbb{C}^{n\times n}}}
    \\
    &\leq &   2\frac{\|\nabla f(X) \zeta_X^p X^\dagger\|_F \|\zeta_X\|_F}{\ip{\zeta_X,\zeta_X}_{\mathbb{C}^{n\times n}}}\leq 2\frac{\|\nabla f(X)\| \| \zeta_X^p X^\dagger\|_F \|\zeta_X\|_F}{\ip{\zeta_X,\zeta_X}_{\mathbb{C}^{n\times n}}}
    \leq 2\frac{\|\nabla f(X)\| 
\| X^\dagger\| \| \zeta_X^p\|_F \|\zeta_X\|_F}{\ip{\zeta_X,\zeta_X}_{\mathbb{C}^{n\times n}}}
    \\
    &\leq& \frac{2\norm{\nabla f(X)} \norm{X^\dagger}\norm{\zeta_X}^2_F}{\ip{\zeta_X,\zeta_X}_{\mathbb{C}^{n\times n}}} = 2\norm{\nabla f(X)} \norm{X^\dagger}
   =\frac{2}{\sigma_p} \norm{\nabla f(X)}. 
\end{eqnarray*}

For the quotient manifold with $g^1$, the FOT is bounded by
\begin{eqnarray*}
\frac{\abs{g^1_Y(2\nabla f(YY^*)\overline{\xi}_Y ,\overline{\xi}_Y)}}{g^1_Y(\overline{\xi}_Y,\overline{\xi}_Y)} &=&  \frac{\abs{\ip{2\nabla f(YY^*)\overline{\xi}_Y ,\overline{\xi}_Y}_{\mathbb{C}^{n\times p}}}}{\ip{\overline{\xi}_Y,\overline{\xi}_Y}_{\mathbb{C}^{n\times p}}}\leq \frac{2\norm{\nabla f(YY^*)\overline{\xi}_Y}_F\norm{\overline{\xi}_Y}_F}{\ip{\overline{\xi}_Y,\overline{\xi}_Y}_{\mathbb{C}^{n\times p}}}  \\
&\leq& \frac{2\norm{\nabla f(YY^*)}\norm{\overline{\xi}_Y}_F^2}{\ip{\overline{\xi}_Y,\overline{\xi}_Y}_{\mathbb{C}^{n\times p}}}= 2 \norm{\nabla f(YY^*)}.
\end{eqnarray*}

For the quotient manifold with $g^2$, the FOTs are 
\begin{eqnarray}
\text{FOTs} &=& \frac{\ip{\nabla f(YY^*)P_Y^\perp \overline{\xi}_Y, \overline{\xi}_Y}_{\mathbb{C}^{n\times p}}}{g^2_Y(\overline{\xi}_Y,\overline{\xi}_Y)} + \frac{\ip{ P_Y^\perp \nabla f(YY^*) \overline{\xi}_Y, \overline{\xi}_Y}_{\mathbb{C}^{n\times p}}}{g^2_Y(\overline{\xi}_Y,\overline{\xi}_Y)}  \\ 
&& + \frac{\ip{Y\overline{\xi}^*_Y\overline{\xi}_Y ,2 \nabla f(YY^*)Y(Y^*Y)^{-1}}_{\mathbb{C}^{n\times p}}}{g_Y^2(\overline{\xi}_Y, \overline{\xi}_Y)} \\
&& - \frac{\ip{ \overline{\xi}_YY^*\overline{\xi}_Y,2 \nabla f(YY^*)Y(Y^*Y)^{-1}}_{\mathbb{C}^{n\times p}}}{g_Y^2(\overline{\xi}_Y, \overline{\xi}_Y)}.
\end{eqnarray}
We can estimate each term separately.
Notice that the SVD of $Y$ can be given as $Y=U\Sigma^{\frac12} V^*$ where $V$ is unitary. Let $\bar S=V^*SV$ and $\bar K=KV$, and $\bar K_i$ be the $i$th column of $\bar K$. For the first summand we have
\begin{eqnarray*}
\frac{\abs{\ip{\nabla f(YY^*)P_Y^\perp \overline{\xi}_Y, \overline{\xi}_Y}_{\mathbb{C}^{n\times p}}}}{g^2_Y(\overline{\xi}_Y,\overline{\xi}_Y)}  &=& \frac{\abs{\ip{\nabla f(YY^*)P_Y^\perp \overline{\xi}_Y, \overline{\xi}_Y}_{\mathbb{C}^{n\times p}}}}{\ip{\overline{\xi}_YY^*,\overline{\xi}_YY^*}_{\mathbb{C}^{n\times n}}} \\
&\leq& \frac{\norm{\nabla f(YY^*)}\norm{\overline{\xi}_Y}_F^2}{\ip{\overline{\xi}_YY^*,\overline{\xi}_YY^*}_{\mathbb{C}^{n\times n}}}. \\
&=& \frac{\norm{YS}^2_F+ \norm{K}^2_F}{\norm{YSY^*}^2_F+ \norm{KY^*}^2_F}\norm{\nabla f(YY^*)}\\
&\leq& \left(\frac{\norm{YS}^2_F}{\norm{YSY^*}^2_F}+ \frac{\norm{K}^2_F}{\norm{KY^*}^2_F} \right) \norm{\nabla f(YY^*)} \\
&=& \left(\frac{\norm{\sqrt{\Sigma}\bar S}^2_F}{\norm{\sqrt{\Sigma}\bar S\sqrt{\Sigma}}^2_F}+ \frac{\norm{\bar K}^2_F}{\norm{\bar K\sqrt{\Sigma}}^2_F} \right) \norm{\nabla f(YY^*)} \\
&\leq& \frac{2}{\sigma_p} \norm{\nabla f(YY^*)}. 
\end{eqnarray*}
Similarly we  have the bounds for the second term:
\begin{equation*}
    \frac{\abs{\ip{P_Y^\perp\nabla f(YY^*) \overline{\xi}_Y, \overline{\xi}_Y}_{\mathbb{C}^{n\times p}}}}{g^2_Y(\overline{\xi}_Y,\overline{\xi}_Y)}  \leq  \frac{2}{\sigma_p} \norm{\nabla f(YY^*)}. 
\end{equation*}

For the third term, with the fact $\|A^*A\|_F=\|A\|_F^2$, we have 
\begin{eqnarray*}
\frac{\abs{\ip{Y\overline{\xi}^*_Y\overline{\xi}_Y ,2 \nabla f(YY^*)Y(Y^*Y)^{-1}}_{\mathbb{C}^{n\times p}}}}{g_Y^2(\overline{\xi}_Y, \overline{\xi}_Y)} &=& \frac{\abs{\ip{Y\overline{\xi}_Y^*\overline{\xi}_YY^*,2\nabla f(YY^*)Y(Y^*Y)^{-2}Y^*}_{\mathbb{C}^{n\times n}}}}{g^2_Y(\overline{\xi}_Y,\overline{\xi}_Y)} \\
&\leq& \frac{\norm{Y\overline{\xi}_Y^*\overline{\xi}_YY^*}_F \norm{2\nabla f(YY^*) Y(Y^*Y)^{-2}Y^*}_F}{g^2_Y(\overline{\xi}_Y,\overline{\xi}_Y)} \\
&\leq &  \frac{\norm{\overline{\xi}_Y Y^*}_F^2\norm{2\nabla f(YY^*)}\norm{Y(Y^*Y)^{-2}Y^*}_F}{g^2_Y(\overline{\xi}_Y,\overline{\xi}_Y)} \\
&=& 2\norm{Y(Y^*Y)^{-2}Y^*}_F \norm{\nabla f(YY^*)} \\
&\leq& \frac{2\sqrt{p}}{\sigma_p} \norm{\nabla f(YY^*)}.
\end{eqnarray*}
Similarly we can bound the fourth term:
\[
\frac{\abs{\ip{ \overline{\xi}_YY^*\overline{\xi}_Y,2 \nabla f(YY^*)Y(Y^*Y)^{-1}}}_{\mathbb{C}^{n\times p}}}{g_Y^2(\overline{\xi}_Y, \overline{\xi}_Y)} \leq  \frac{2\sqrt{p}}{\sigma_p} \norm{\nabla f(YY^*)}.
\]

Thus, for the quotient manifold with $g^2$  we have
\[
\abs{\text{FOTs}} \leq \frac{4(\sqrt{p}+1)}{\sigma_p} \norm{\nabla f(YY^*)}.
\]

For the quotient manifold with $g^3$, recall that $P_Y^\perp=I-P_Y=I-Y(Y^*Y)^{-1}Y^*$, 
with the property \eqref{metric-3-horizontal} and the fact $(I-P_Y)^*Y=0$,
the FOT can be bounded as follows: 
\begin{eqnarray*}
    \abs{\text{FOT}} &=& \frac{\abs{g^3_Y((I-P_Y)\nabla f(YY^*)(I-P_Y)\overline{\xi}_Y(Y^*Y)^{-1},\overline{\xi}_Y)}}{g^3_Y(\overline{\xi}_Y,\overline{\xi}_Y)} \\
    &=& \frac{2\abs{\ip{P_Y^\perp \nabla f(YY^*)P_Y^\perp \overline{\xi}_Y, \overline{\xi}_Y}_{\mathbb{C}^{n\times p}}}}{g^3_Y(\overline{\xi}_Y, \overline{\xi}_Y)} \\
    &=& \frac{2\abs{\ip{
    \nabla f(YY^*)Y_\perp K,Y_\perp K}_{\mathbb{C}^{n\times p}}}}{g^3_Y(\overline{\xi}_Y, \overline{\xi}_Y)} \\
    &=& \frac{2\abs{\ip{
    \nabla f(YY^*)Y_\perp K,Y_\perp K}_{\mathbb{C}^{n\times p}}}}{\norm{Y\overline{\xi}_Y^*+ \overline{\xi}_YY^*}^2_F} \\
    &=& \frac{2\abs{\ip{
    \nabla f(YY^*)Y_\perp K,Y_\perp K}_{\mathbb{C}^{n\times p}}}}{\norm{2YSY^*+Y_\perp KY^* + YK^*Y_\perp^*}^2_F}\\
    &=& \frac{2\abs{\ip{
    \nabla f(YY^*)Y_\perp K,Y_\perp K}_{\mathbb{C}^{n\times p}}}}{\norm{2YSY^*}_F^2+\norm{Y_\perp KY^*}_F^2 + \norm{YK^*Y_\perp^*}^2_F}\\
    &=& \frac{\abs{\ip{
    \nabla f(YY^*)Y_\perp K,Y_\perp K}_{\mathbb{C}^{n\times p}}}}{2\norm{YSY^*}_F^2+\norm{Y_\perp KY^*}_F^2}\\
    &\leq& \frac{\abs{\ip{
    \nabla f(YY^*)Y_\perp K,Y_\perp K}_{\mathbb{C}^{n\times p}}}}{\norm{Y_\perp KY^*}_F^2}\\
    &\leq& \frac{\norm{Y_\perp K}^2_F}{\norm{Y_\perp K Y^*}^2_F} \norm{\nabla f(YY^*)} \\ 
    &\leq& \frac{1}{\sigma_p}\norm{\nabla f(YY^*)}.
\end{eqnarray*}
\end{proof}

\end{appendices}

%Bibliography
\bibliographystyle{unsrt}  
\bibliography{references}

\end{document}